\documentclass[11pt,a4paper,reqno]{amsart}
\usepackage[T1]{fontenc}
\usepackage{amsmath}
\usepackage{amsthm}
\usepackage{amsfonts}
\usepackage{amssymb}
\usepackage{amsbsy}
\usepackage{mathrsfs}
\usepackage{bbm}
\newtheorem{theorem}{Theorem}[section]
\newtheorem{lemma}[theorem]{Lemma}

\newtheorem{proposition}[theorem]{Proposition}
\newtheorem{claim}{Claim}

\theoremstyle{remark}
\newtheorem{remark}[theorem]{\it \bf{Remark}\/}

\numberwithin{equation}{section}
\catcode`@=11
\def\section{\@startsection{section}{1}%
  \z@{1.5\linespacing\@plus\linespacing}{.5\linespacing}%
  {\normalfont\bfseries\large\centering}}
\catcode`@=12
\newcommand{\be}{\begin{equation}}
\newcommand{\ee}{\end{equation}}
\newcommand{\bea}{\begin{eqnarray}}
\newcommand{\eea}{\end{eqnarray}}
\newcommand{\bee}{\begin{eqnarray*}}
\newcommand{\eee}{\end{eqnarray*}}

\def\pa{\partial}

\def\RR{\mathbb{R}}

\def\fref#1{{\rm (\ref{#1})}}

\catcode`@=11
\def\supess{\mathop{\operator@font Sup\,ess}}
\catcode`@=12

\def\RR{\mathbb{R}}

\def\e{\varepsilon}

\def\bar#1{{\overline #1}}
\def\fref#1{{\rm (\ref{#1})}}

\def\R2+{\RR ^2_+}

\def\tz{\tilde{z}}

\def\lsl{\frac{\lambda_s}{\lambda}}

\def\pa{\partial}

\def\lim{\mathop{\rm lim}}

\def\sup{\mathop{\rm sup}}
\def\exp{{\rm exp}}
\def\l{\lambda}

\def\log{{\rm log}}

\def\et{\tilde{\e}}

\def\lsl{\frac{\lambda_s}{\lambda}}
\def\xsl{\frac{x_s}{\lambda}}

\def\cal{\mathcal}

\def\ut{\tilde{u}}

\def\matchal{\mathcal}

\def\pa{\partial}

\def\et{\tilde{\e}}

\def\pa{\partial}

\title[Blow up for the critical gKdV I]{Blow up for the critical gKdV equation I: dynamics near the soliton}
\author[Y. Martel]{Yvan Martel}
\address{Universit\'e de Versailles St-Quentin and Institut Universitaire de France, LMV  CNRS UMR8100}
\email{yvan.martel@uvsq.fr}
\author[F. Merle]{Frank Merle}
\address{Universit\'e de Cergy Pontoise and Institut des Hautes \'Etudes Scientifiques, AGM CNRS UMR8088}
\email{merle@math.u-cergy.fr}
\author[P. Rapha\"el]{Pierre Rapha\"el}
\address{Universit\'e Paul Sabatier and Institut Universitaire de France, IMT CNRS UMR 5219}
\email{pierre.raphael@math.univ-toulouse.fr}
\begin{document}

\begin{abstract}
We consider the mass critical (gKdV) equation $u_t + (u_{xx} + u^5)_x =0$ for initial data in $H^1$ close to the soliton, which is a canonical {\it mass critical} problem. In the earlier works, \cite{MMjmpa,Mjams,MMannals}, finite or infinite time blow up is proved for non positive energy solutions, and the solitary wave is shown to be the universal blow up profile. For well localized initial data, finite time blow up with an upper bound on blow up rate is obtained in \cite{MMjams}.\\ 
In this paper, we fully revisit the analysis for (gKdV) in light of the recent progress made on the study of critical dispersive blow up problems \cite{MRjams, RR2009,MRR,MRS}. For a class of initial data close to the soliton, we show that three scenario only can occur: (i) the solution leaves any small neighborhood of the modulated family of solitary waves in the scale invariant $L^2$ norm; (ii) the solution is global and converges to a solitary wave as $t\to +\infty$; (iii) the solution blows up in finite time in a universal regime with speed: $$\|u_x(t)\|_{L^2}\sim \frac{C(u_0)}{T-t}.$$ The regimes (i) and (iii) are moreover {\it stable}. We also show that nonpositive energy initial data yield finite time blow up, and obtain the classification of the solitary wave at zero energy as in \cite{MRjams}.
\end{abstract}

\maketitle
\section{Introduction}
\subsection{Setting of the problem}
We consider the $L^2$-critical generalized Korteweg--de Vries equation (gKdV) 
\begin{equation}\label{kdv}
{\rm (gKdV)}\ \ \left\{ \begin{array}{ll}
 u_t + (u_{xx} + u^5)_x =0, \quad & (t,x)\in [0,T)\times {\mathbb R}, \\
 u(0,x)= u_0(x), & x\in {\mathbb R}.
\end{array}
\right.
\end{equation}
The Cauchy problem is locally well posed in the energy space $H^1$ from Kenig, Ponce and Vega \cite{KPV}, and given $u_0 \in H^1$, there exists a unique\footnote{in a certain sense} maximal solution $u(t)$ of \eqref{kdv} in $C([0,T), H^1)$ with either $T=+\infty$, or $T<+\infty$ and then $\lim_{t\to T} \|u_x(t)\|_{L^2} = +\infty$. The mass and the energy are conserved by the flow: $\forall t\in [0,T)$,  $$M(u(t))=\int u^2(t)= M_0, \ \ E(u(t))= \frac 12 \int u_x^2(t) - \frac 16 \int u^6(t)= E_0,$$ 
where $M_0=M(u_0)$, $E_0=E(u_0)$, and the scaling symmetry $(\l>0$) $$u_\lambda(t,x)=\lambda^{\frac12}u(\lambda^3t,\lambda x)$$ leaves invariant the $L^2$ norm so that the problem is {\it mass critical}.

The family of travelling wave solutions $$u(t,x) = \lambda_0^{-\frac  12}Q\left(\lambda_0^{-1} (x-\lambda_0^{-2} t-x_0)\right), \ \ (\lambda_0,x_0)\in\RR^*_+\times \RR,$$ with  
\begin{equation}\label{eq:Q}
Q(x) =   \left(\frac {3}{\cosh^{2}\left( 2 x\right)}\right)^{\frac14}, \quad
Q''+Q^5 =Q,\quad E(Q)=0,
\end{equation}
plays a distinguished role in the analysis. From variational argument \cite{W1983}, $H^1$ initial data with subcritical mass $\|u_0\|_{L^2}<\|Q\|_{L^2}$ generate global and $H^1$ bounded solutions $T=+\infty$.

For $\|u_0\|_{L^2}\geq \|Q\|_{L^2}$, the existence of blow up solutions has been a long standing open problem. In particular, unlike for the analogous Schr\"odinger problem, there exists no simple obstruction to global existence. The study of singularity formation for small super critical mass $H^1$ initial data 
\begin{equation}
\label{utwosmall}
\|Q\|_{L^2}\leq\|u_0\|_{L^2}<\|Q\|_{L^2}+\alpha^*, \ \ \alpha^*\ll1 
\end{equation} 
has been developed in a series of works by Martel and Merle \cite{MMjmpa,MMgafa,Mjams,MMannals, MMduke, MMjams} where two new sets of tools are introduced:

-- monotonicity formula and $L^2$ type localized virial identities to control the flow near the solitary wave;

-- rigidity Liouville type theorems to classify the asymptotic dynamics of the flow. In particular, the first  proof of blow up in finite or infinite time is obtained for 
initial data 
\be\label{classe}
\hbox {$u_0 \in H^1$ with \eqref{utwosmall} and  $E(u_0)<0$.}
\ee
The proof is indirect and based on a classification argument:
the solitary wave is characterized as the unique universal attractor of the flow in the singular regime. If $u(t)$ blows up in finite  or infinite  time $T$ with (\ref{utwosmall}), then the flow admits near blow up time a decomposition 
\begin{equation}
\label{asymptstab}
u(t,x)=\frac{1}{\lambda^{\frac12}(t)}(Q+\e)\left(t,\frac{x-x(t)}{\lambda(t)}\right)\ \ \mbox{with} \ \ \e(t)\to 0\ \ \mbox{in}\ \ L^2_{\rm loc}\ \ \mbox{as}\ \ t\to T.
\end{equation}

Then, in \cite{MMjams}, for well localized initial data
\be\label{classe2}
\hbox{$u_0$ satisfying \eqref{classe} and } 
\int_{x'>x} u_0^2(x') dx' < \frac {C}{x^6} \hbox{ for $x>0$},\ee
 blow up is   proved to occur in finite time $T$ with an upper bound on a sequence $t_n\to T$: 
\be
\label{ippenouf}
\|u_x(t_n)\|_{L^2}\leq  \frac{C(u_0)}{T-t_n},
\ee
by a dynamical proof\footnote{arguing directly on the solution itself.}.

For the critical mass problem $\|u_0\|_{L^2}=\|Q\|_{L^2}$, assuming in addition the following decay $\int_{x'>x} u_0^2(x') dx' < \frac {C}{x^3}$ for $x>0$, it was proved in \cite{MMduke} that the solution is global and does not blowup in infinite time.
 
\subsection{Generic blow up for critical problems}

In the continuation of these works, the program developed by Merle and Rapha\"el \cite{MRgafa,MRinvent,MRannals,FMR,Rannalen,MRcmp,MRjams} for the mass critical nonlinear Schr\"odinger equation \begin{equation}\label{eq:10}
{\rm (NLS)} \ \ \left\{\begin{array}{ll} i \partial_t u+\Delta u +|u|^{\frac 4N} u=0,\\ u_{|t=0}=u_0\end{array}\right .  \ \  (t,x) \in [0,T)\times \RR^N.
\end{equation}
in dimensions $1\leq N \leq 5$ has led to a complete description of the stable blow up scenario near the solitary wave $Q$ which is the unique $H^1$ nonnegative solution up to translation to $\Delta Q-Q+Q^{1+\frac4N}=0$.  This problem displays a similar structure like the critical (gKdV). Initial data in $H^1$   with $\|u_0\|_{L^2}<\|Q\|_{L^2}$ are global and bounded, \cite{W1983}. 
For $u_0\in H^1$ with $\|u_0 \|_{L^2} = \|Q\|_{L^2}$, Merle \cite{Mduke} proved that the only blow up solution (up to the symmetries of the equation) is 
\be\label{solS}
S(t,x) = \frac 1 {t^{N/2}} e^{-i  (\frac {|x|^2}{4t}- \frac 1 t )} Q\left(\frac xt\right). 
\ee
For small super critical mass $H^1$ initial data 
\begin{equation}
\label{smallsupercriticalmassinitial}
\|Q\|_{L^2}<\|u_0\|_{L^2}<\|Q\|_{L^2}+\alpha^*, \ \ \alpha^*\ll1,
\end{equation}
an $H^1$ open set of solutions is exhibited where solutions blow up in finite time at log--log speed:
\begin{equation}
\label{loglolowaw}
\|\nabla u(t)\|_{L^2}\sim C^*\sqrt{\frac{\log |\log(T-t)|}{T-t}}.
\end{equation}
Moreover, nonpositive energy solutions belong to this set of generic blow up. This double log correction to self similarity for stable blow up was conjectured from numerics  by Landman, Papanicolou, Sulem and Sulem \cite{PSS}, and a family of such solutions was rigorously constructed by a different approach by Perelman in dimension $N=1$, \cite{PE}.
Blow up solutions of the type \eqref{solS} ($\|u(t)\|_{H^1} \sim \frac 1t$), constructed by Bourgain, Wang \cite{BW}, see also Krieger, Schlag \cite{KSNLS}, correspond to an {\it unstable threshold dynamics} as proved in Merle, Rapha\"el, Szeftel \cite{MRS}.
Finally, under \eqref{smallsupercriticalmassinitial}, the quantization of the focused mass at blow up is proved 
\begin{equation}
\label{quanta}
|u(t)|^2\rightharpoonup \|Q\|_{L^2}^2\delta_{x=x(T)}+|u^*|^2, \ \ u^*\in L^2.
\end{equation}

\medskip

More recently, natural connections have been made between mass critical problems and {\it energy critical}   problems. For the energy critical wave map problem, after the pioneering work \cite{RodSter}, a complete description of a generic finite time blow up dynamics (log correction to the self similar speed) was given by Rapha\"el, Rodnianski \cite{RR2009}, while {\it unstable regimes} with different speeds were constructed by Krieger, Schlag, Tataru \cite{KST}. See also Merle, Raphael, Rodnianski \cite{MRR} for the treatment of the Schr\"odinger map system and Rapha\"el, Schweyer \cite{Rstud} for the parabolic harmonic heat flow.

\medskip

The general outcome of these works is twofold.

First the {\it sharp} derivation of the blow up speed in the {\it generic} regime relies on a detailed analysis of the structure of the solution near collapse, and takes in particular into account slowly decaying tails in the computation of the leading order blow up profile. These tails correspond to the leading order dispersive phenomenon which drives the speed of concentration and the rate of dispersion, both being intimately linked.

Second, a robust analytic approach has been developed in a nowadays more unified framework. In particular, the control of the solution in the singular regime relies on mixed energy/Morawetz or Virial  type estimates adapted to the flow which have been used in various settings, see in particular \cite{MRjams}, \cite{RS2010}, \cite{RR2009}, \cite{MRR}.

%%%%%%%%%%%%%%%%%%%%%%%%%%%%%%%%
%%%%%%%%%%%%%%%%%%%%%%%%%%%%%%%%

\subsection{Statement of the results}

%%%%%%%%%%%%%%%%%%%%%%%%%%%%%%%%%
%%%%%%%%%%%%%%%%%%%%%%%%%%%%%%%%%

The aim of the paper is to classify the gKdV dynamics for $H^1$ solutions  close to the soliton and with decay on the right. In particular, we aim at recovering the more refined description of the flow obtained for the $L^2$ critical NLS equation.

More precisely, let us define the $L^2$ modulated tube around the soliton manifold:
\be\label{tube}
\mathcal T_{\alpha^*}=\left\{u\in H^1\ \ \mbox{with}\ \ \inf_{\l_0>0, \ x_0\in \RR}
\Big\|u-\frac 1{\lambda_0^{\frac 12}} Q\left(\frac {.-x_0}{\lambda_0} \right) \Big\|_{L^2} <\alpha^*\right\}\ee 
and consider the set of initial data
$$
\mathcal{A}=
\left\{
u_0=Q+\e_0   \hbox{ with } \|\e_0\|_{H^1}<\alpha_0 \hbox{ and }
\int_{y>0} y^{10}\e_0^2< 1
\right\}.
$$
Here $\alpha_0,\alpha^*$ are universal constants with 
\be
\label{relationalpha}
0<\alpha_0\ll \alpha^*\ll1.
\ee
Our aim is to classify the flow for data $u_0\in\mathcal A$.  First, we  fully describe the blow up solutions in the tube $\mathcal{T}_{\alpha^*}$: there is only one blow up type, which is stable. We then show that in fact only three scenario occur: 
\medskip

- stable blow up with $1/(T-t)$ speed;

- convergence to a solitary wave in large time;

- stable defocusing behavior (the solution leaves the tube $\mathcal{T}_{\alpha^*}$ in finite time).

\medskip

We first claim:

\begin{theorem}[Blow up near the soliton in $\mathcal A$]\label{th:1} 
There exist universal constants $0<\alpha_0\ll\alpha^*\ll1 $ such that the following holds. Let $u_0 \in \mathcal{A}$.
\\
{\rm (i) Nonpositive energy blow up.} If $E(u_0)\leq 0 $ and  $u_0$ is not a soliton, then $u(t)$ blows up in finite time and, for all $t\in [0,T),$ $u(t)\in \mathcal{T}_{\alpha^*}$.
 \\
{\rm (ii) Description of blow up.} 
Assume that $u(t)$ blows up in finite time $T$ and that for all $t\in [0,T),$ $u(t)\in \mathcal{T}_{\alpha^*}$. Then there exists $ \ell_0=\ell_0(u_0)>0$ such that
\be\label{blow}
\|u_x(t)\|_{L^2} \sim \frac {\|Q'\|_{L^2}}{\ell_0 (T-t)} \quad 
\hbox{as $t\to T$.}
\ee
Moreover,
there exist $\lambda(t)$, $x(t)$ and $u^*\in H^1$, $u^*\neq 0$,
such that 
\begin{equation}\label{th1.1}
u(t,x)-\frac{1}{\lambda^{\frac12}(t)}Q\left(\frac{x-x(t)}{\lambda(t)}\right)\to u^*\ \ \mbox{in}\ \ L^2\ \ \mbox{as}\ \ t\to T,
\end{equation} 
where
\be\lambda(t)\sim \ell_0 (T-t), \ \ x(t)\sim \frac 1 {\ell_0^2(T-t)}\ \ \mbox{as} \ \ t\to T,\ee
\begin{equation}
\label{th:1:4}
\int_{x>R} ({u^\star})^2(x) dx \sim \frac {\|Q\|_{L^1}^2}{8 \ell_0 R^2}\ \ \mbox{as} \ \ R\to +\infty.
\end{equation}
\\
{\rm (iii) Openness of the stable blow up.}   Assume that $u(t)$ blows up in finite time $T$ and that for all $t\in [0,T),$ $u(t)\in \mathcal{T}_{\alpha^*}$.  Then there exists $\rho_0=\rho_0(u_0) >0$ such that for all $v_0\in \mathcal{A}$ with $\|v_0-u_0\|_{H^1}< \rho_0$, the corresponding solution $v(t)$ blows up in finite time $T(v_0)$ as in {\rm (ii)}.
\end{theorem}

\noindent{\bf Comments on Theorem \ref{th:1}}

\medskip

\noindent{\it 1. Blow up speed}:  An important feature of Theorem \ref{th:1} is the derivation of {\it the} stable blow up speed for $u_0\in \mathcal{A}$: \be\label{bll}\|u_x(t)\|_{L^2}\sim \frac{C}{T-t}\ee 
which implies that $x(t)\to +\infty$ as $t\to T$.
Such a blow up rate confirms the conjecture formulated in \cite{MMjams} for $E_0<0$. 
Recall that for $u_0 \in \mathcal{A}$ and $E_0<0$,  
assuming some a priori global information on the $\dot H^1$ norm for all time in \cite{MMjams}, one could deduce \eqref{bll}. The derivation of such a bound is the key to the proof of Theorem~\ref{th:1}. This blow up speed is very far above the scaling law $\|u_x\|_{L^2}\sim 1/(T-t)^{\frac13}$
(see  \cite{MRR}, \cite{Rstud} for a similar phenomenon for energy critical geometrical problems  ).\\

\noindent{\it 2. Structure of $u^*$}:  The decay of $u^*$ in $L^2$ is directly related to the blow up speed $\frac {\|Q'\|_{L^2}}{\ell_0 (T-t)} $,  itself related 
to the speed of ejection of mass in time from the rescaled soliton, similarly like for the critical (NLS), see \cite{MRcmp}. Note that the Cauchy problem is wellposed in $L^2$, so that the $L^2$ convergence \eqref{th1.1} is relevant.
It is an open question but very likely that the convergence in \eqref{th1.1} holds in $H^1$ since the
left hand side is shown to be bounded in $H^1$ and $u^*$ is in $H^1$. The fact that $u^*\in H^1$ is in contrast with the stable regime for critical NLS, where the accumulation of ejected mass from the rescaled soliton implies that $u^*\not \in L^p$, $p>2$.
Here we still observe some ejection of mass from the soliton, but since the concentration point $x(t)$ of the soliton is going to infinity,
the mass does not accumulate at a fixed point and gives the tail of $u*$. More generally, the regularity of $u^*$ is directly connected to the blow up speed and the strength of deviation from self similarity, see \cite{Rstud}, \cite{MRR}.
\medskip

\noindent{\it 3.   On localization on the right}: Let us stress the importance of the decay  assumption on the right in space for the initial data which was already essential in \cite{MMjams}, \cite{MMduke}.
Indeed, in contrast with the NLS equation, the universal dynamics can not be seen in $H^1$ since an additional assumption of decay to the right   is required:

- In part II of this work \cite{MMR2}, we construct a minimal mass blow up solution with $1/(T-t)$ blow up. The initial data is in $H^1$ and decays slowly on the right\footnote{this is mandatory from \cite{MMduke}: no minimal mass blow up for data with decay on the right}. Thus, the blow up set without decay assumption on the right is {\it not open} in $H^1$.

- For negative energy solutions with initial data with slow decay on the right (so that Theorem \ref{th:1} and \cite{MMjams} do not apply), we expect the existence of solutions with different blow up speeds $1/(T-t)^\alpha$, $\alpha> 1$.

Note that  there is however no sharpness in the $y^{10}$ weight in Theorem \ref{th:1}.

\medskip

\noindent{\it 4.   Dynamical characterization  of $Q$}: Recall from the variational characterization of $Q$ that $E(u_0)\leq 0$ implies $\|u_0\|_{L^2}>\|Q\|_{L^2}$, unless $u_0\equiv Q$ up to scaling and translation symmetries. Theorem \ref{th:1} therefore recovers  the dynamical classification of $Q$ as the unique global zero energy solution in $\mathcal{A}$ like for the mass critical (NLS), see \cite{MRjams}. The proof of this type of result is delicate, and one needs to rule out a scenario of vanishing of the energy of the radiation specific to the  zero energy case.
Here, we expect this result to hold without decay assumption (no global $H^1$ energy zero solution close to $Q$ exists except $Q$).\\

We now claim the following rigidity of the flow for data in $\mathcal A$:

\begin{theorem}[Rigidity of the dynamics in $\mathcal{A}$]
\label{th:2}
There exist universal constants $0<\alpha_0\ll\alpha^*\ll1 $ such that the following holds. Let $u_0 \in \mathcal{A}$.
\\
 Then, one of the following three scenarios occurs:

\medskip

\noindent{\rm (Exit)}  There exists $t^*\in (0,T)$ such that 
$u(t^*)\not \in \mathcal{T}_{\alpha^*}$. 

 \medskip

\noindent{\em (Blow up)}  For all $t\in [0,T),$ $u(t)\in \mathcal{T}_{\alpha^*}$ and the solution blows up in finite time $T<+\infty$ in the regime described by Theorem \ref{th:1}.
\medskip

\noindent{\em (Soliton)}  The solution is global, for all $t\geq 0, $ $u(t)\in \mathcal{T}_{\alpha^*}$, and  there exist $\lambda_{\infty}>0$, $x(t)$ such that  
\begin{equation}
\l_\infty^{\frac 12} u(t,\l_\infty \cdot +x(t))\to Q\quad \hbox{in $H^1_{\rm loc}$ as $t\to +\infty$},
\end{equation} 
\be
\label{nveonveonoene}
|\l_{\infty}-1|\leq o_{\alpha_0\to 0}(1),\quad 
 x(t)\sim \frac{t}{\lambda_{\infty}^2} \ \ \mbox{as}\ \ t\to +\infty
\ee 
\end{theorem}

\noindent{\bf Comments on Theorem \ref{th:2}}
\medskip

\noindent{\it 1. Stable/unstable manifold}:
All three possibilities are known to occur for an infinite set of initial data. Moreover, the sets of initial data leading to (Exit)  and  (Blow up)  are both open in $\mathcal{A}$ by perturbation of the data in $H^1$. For $\int u_0^2<\int Q^2$, only the (Exit) case can occur and for $E_0<0$, only (Blow up) can occur. From the proof of Theorem \ref{th:2}, the (Soliton) dynamics can be achieved as  threshold dynamics between the two stable regimes (Exit) and (Blow up) as in \cite{MMC}, \cite{RH}, \cite{MRR}. More precisely, given $b\in \RR$ small, let $Q_b$ be the suitable perturbation of $Q$ build in Lemma \ref{le:1}, and $\e_0$ be a suitable small perturbation satisfying the orthogonality conditions \fref{ortho1}. Then there exists $b_0=b(\e_0)$ such that the solution to (gKdV) with initial data $Q_{b_0}+\e_0$ satisfies (Soliton). The Lipschitz regularity of the flow $\e_0\to b(\e_0)$ needed to build a smooth manifold remains to be proved, see \cite{KSNLS} for related constructions. Note also that solutions that scatter to Q in the regime (Soliton) where constructed dynamically by C\^ote \cite{Cotekdv}.

\medskip

\noindent{\it 2.   Classification of the flow in $\mathcal{A}$.} Theorem \ref{th:2} is a first step towards a complete classification of the flow for initial data in $\mathcal{A}$. Its structure is reminiscent from   classification results obtained by Nakanishi and Schlag \cite{NS1}, \cite{NS2} for super critical wave and Schr\"odinger equations. These results were proved using   classification arguments based on the Kenig, Merle concentration compactness approach \cite{KM1}, the classification of critical dynamics by Duyckaerts, Merle \cite{DM}, see also \cite{DR}, and eventually a {\it no return lemma}. In  the analogue of the (Exit) regime, this lemma shows that the solution cannot come back close to solitons  and   in fact scatters. In the critical situations, such an analysis is more delicate and incomplete, see \cite{KNS}, and both the blow statements and the no return lemma in \cite{NS1}, \cite{NS2} rely on a specific algebraic structure - the virial identity - which does not exist for (gKdV).\\
In the continuation of Theorem \ref{th:2}, what remains to be done to fully describe the flow for data $u_0\in \mathcal A$ is to answer the question: $$\mbox{what happens after $t^*$ in the (Exit) regime?}$$

In \cite{MMR2}, the second part of this work, we propose a new approach to answer this question related to the understanding of the threshold dynamics. We will proceed in two steps:
\begin{enumerate}
\item We prove the {\it existence and uniqueness in $H^1$} of a minimal mass blow up solution $\|u_0\|_{L^2}=\|Q\|_{L^2}$. From \cite{MMduke}, this solution has slow decay to the right and is global on the left in time. 
\item We then show that in the (Exit) case of Theorem \ref{th:2}, the solution is at time $t^*$ $L^2$ close to the unique minimal mass blow up solution.
\end{enumerate}

Having in mind the properties of threshold solutions for $H^1$ critical NLS and wave equations (\cite{DM2,DM}), and the case of the $L^2$ critical NLS equation (the solution $S(t)$ in \eqref{solS} scatters), it is natural to expect that the minimal mass blow up solution of (gKdV)  also scatters in negative time. {\it Assuming this} and because scattering is open in the critical $L^2$ space, we obtain that (Exit) implies scattering. In other words, we prove in \cite{MMR2} that all solutions scatter in the (Exit) regime if and only if the unique $H^1$ minimal mass blow up solution scatters to the left. This ends the classification of the flow in $\mathcal A$, in particular the only blow-up regime is the $1/(T-t)$ universal blow-up regime of Theorem \ref{th:1} and it is stable.  

\medskip

\noindent{\it 3. Finite/Infinite dimensional dynamics}. The proof of Theorem \ref{th:2} relies on a detailed description of the flow. We will show that before the (Exit) time $t^*$, the solution admits a decomposition $$
u(t,x) =\frac 1 {\lambda^{\frac 12}(t)} (Q_{b(t)} +\e) \left(t, \frac{x-x(t)}{\lambda(t)}\right)$$ where $Q_b$ is a suitable $O(b)$ deformation of the solitary wave profile, and there holds the bound $$\|\e\|_{H^1_{\rm loc}}\ll b.$$ We then extract the universal finite dimensional system which drives the geometrical parameters: 
\be
\label{infinoen}
\frac{ds}{dt}=\frac{1}{\l^3}, \ \  -\frac{\l_s}{\l}=b, \ \ b_s+2b^2=0.
\ee
It is easily seen that starting from $\lambda(0)=1$, $b(0)=b_0$, the phase portrait of the dynamical system \fref{infinoen} is:
\begin{enumerate}
\item
 for $b_0<0$, $\lambda(t)=1+|b_0|t$, $t\geq 0$, stable;
 \item for $b_0=0$, $\lambda(t)=1$, $t\geq 0$, unstable;
 \item for $b_0>0$, $\lambda(t)=b_0(T-t)$ with $T=\frac{1}{b_0}$, stable.
 \end{enumerate}
We may then reword Theorem \ref{th:2} by saying that the infinite dimensional system (gKdV) for data $u_0\in \mathcal A$ is governed to leading order by the universal finite dimensional dynamics
\fref{infinoen}. This is a non trivial claim due to the non linear structure of the problem, and the proof relies on a {\it rigidity} formula when measuring the interaction of the radiative term $\e$ with the ODE's \fref{infinoen}, see Lemma \ref{le:4.1}. Let us stress that the assumption of decay to the right is fundamental here, and we expect that slow decaying tails may force a different coupling with new leading order ODE's.\\
Finally, note that like for the finite dimensional system \fref{infinoen}, the three scenarios of Theorem \ref{th:2} can be seen on $\l(t)$ only and are equivalently characterized by:
\smallskip

(Soliton) for all $t$, $\l(t) \in [\frac 12, 2]$;
\smallskip

(Exit) there exists $t_0>0$ such that $\l(t_0)>2$;
\smallskip

(Blow up) there exists $t_0>0$ such that $\l(t_0)<\frac 12$.
\smallskip

\medskip
\noindent We expect that results such as  Theorem \ref{th:2} (classification of the dynamics close to the solitary waves) can be proved similarly for other problems like the NLS equation, the wave equation, etc.

\bigskip

\noindent{\bf Notation.} Let the linearized operator close to $Q$ be:
\be
\label{deflplus}
Lf=-f''+f-5Q^4f.
\ee 
We introduce the generator of $L^2$ scaling: $$\Lambda f=\frac12f+yf'.$$
 For a given    generic small  constant $0<\alpha^*\ll1 $,   $\delta(\alpha^*)$ denotes a generic small constant with $$\delta(\alpha^*)\to 0\ \ \mbox{as}\ \ \alpha^*\to 0.$$
We note the $L^2$ scalar product: $$(f,g)=\int f(x)g(x)dx.$$ 

%%%%%%%%%%%%%%%%%%%%%%%%%%%%%%%%%%%%%%%%%%%

\subsection{Strategy of the proof}

%%%%%%%%%%%%%%%%%%%%%%%%%%%%%%%%%%%%%%%%%%%

We give in this section a brief insight into the proof of Theorems   \ref{th:1} and \ref{th:2}. As mentioned before, we are pushing further the dynamical analysis of the problem initiated in \cite{MMjams}. We will not use rigidity arguments as for the theory in $H^1$ (see \cite{Mjams}, \cite{MMannals}). Nevertheless, we will use tools introduced to prove such rigidity arguments, such as modulation theory, $L^2$ and energy monotonicity, local Virial identities and weighted estimates for $x>0$.
However, the proofs here are self-contained, except for the   Virial estimates, for which we refer to \cite{MMjmpa} and \cite{MMannals}.

\medskip

\noindent \emph{(i). Formal derivation of the law}

\noindent We start as in \cite{MRgafa}, \cite{MRjams}, \cite{RR2009} by refining the blow up profile and considering an approximation to the renormalized equation. We look for a solution to (gKdV) of the form 
\be
\label{transforu}
u(t,x)=\frac 1 {\lambda^{\frac 12}(t)}  Q_{b(s)}\left(\frac{x-x(t)}{\lambda(t)}\right), \ \ \frac{ds}{dt}=\frac{1}{\lambda^3}, \ \  \frac{x_s}{\lambda}=1, \ \ b=-\lsl,
\ee 
which leads to the slowly modulated self similar equation: 
\be
\label{noehnfoeheo}
b_s\frac{\partial Q_b}{\partial b}+b\Lambda Q_b+(Q_b''-Q_b+Q_b^5)'=0.
\ee A formal derivation of the generic blow up speed can be obtained as follows: look for a slowly modulated ansatz $$Q_b=Q+bP+b^2 P_2 +\dots, \ \ b_s=-c_2b^2+c_3 b^3+\dots$$ where the unknowns are $P$ and  $c_2,$ $c_3$.  Let the linearized operator close to $Q$ be given by \eqref{deflplus}, then the order $b$ expansion leads to the equation $$(L P)' =\Lambda Q$$ which thanks to the critical orthogonality condition $(Q,\Lambda Q)=0$ can be solved for a function $P$ that decays exponentially to the right, but displays a non trivial tail on the left $\lim_{y\to -\infty}P(y)\neq 0.$ At the level $b^2$, a similar {\it flux type} computation\footnote{see \fref{computationflux}} reveals that the $P_2$ equation can be solved with a similar profile for the value $c_2=2$ only\footnote{otherwise, $P_2$ {\it grows} exponentially on the right or the left.}. This corresponds to the formal dynamical system \be
\label{appromodulation}
-\lsl=b, \ \ b_s+2b^2=\lambda^2\frac{d}{ds}\left(\frac{b}{\l^2}\right)=0, \ \ \frac{ds}{dt}=\frac{1}{\lambda^3}
\ee 
which after reintegration yields finite time blow up for $b(0)>0$ with $$\lambda(t)=c(u_0)(T-t).$$ 

\noindent \emph{(ii). Decomposition of the flow and modulation equations (section 2)}

\noindent For the analysis, it is enough to work with the localized approximate self similar profile $$Q_b=Q+\chi(|b|^\gamma y)P(y)$$ for some well chosen\footnote{see Lemma \ref{le:1}, we can take $\gamma=\frac34$} $\gamma>0$. As long as the solution remains in the tube $\matchal T_{\alpha^*}$, we may introduce the nonlinear decomposition of the flow: 
\begin{equation}\label{eq:23b}
	u(t,x) =\frac 1 {\lambda^{\frac 12}(t)} (Q_{b(t)} +\e) \left(t, \frac{x-x(t)}{\lambda(t)}\right),
\end{equation}
where the three time dependent parameters are adjusted to ensure suitable orthogonality conditions\footnote{see \fref{ortho1}} for $\e$. A specific feature of the (KdV) flow is that the generalized null space of the full linearized operator $(L)'$ close to $Q$ involves {\it badly} localized functions in the right, and hence the modulations equations driving the parameters are roughly speaking of the form 
\be
\label{vnkonveoneonv|}
\lsl+b=\frac{d J_1}{ds}+O(\|\e\|_{H^1_{\rm loc}}^2), \ \ b_s+b^2\sim \frac{d J_2}{ds}+O(\|\e\|_{H^1_{\rm loc}}^2)
\ee with $$ |J_i|\lesssim \|\e\|_{H^1_{\rm loc}}+\int_{y>0}|\e|,  \ \ i=1,2.$$ This explains the need for a control of radiation on the right as slow tails and large $J_i$ might otherwise perturb the formal system \fref{appromodulation} (see also \cite{MMjams}).\\

\noindent \emph{(iii). The mixed energy/Virial estimate (section 3)}

The main new input of our analysis is the derivation of a dispersive control on the local norm $\|\e\|_{H^1_{\rm loc}}$ which is relevant in all three regimes, and therefore must display some scaling invariant structure. For this, we adapt and revisit the construction of mixed energy/Virial functionals as introduced in \cite{MMT}, \cite{RodSter}, \cite{RR2009}, \cite{RS2010}. Indeed, we build a nonlinear functional $$\matchal F\sim \int\psi\e_y^2+\varphi\e^2-\frac 13  \psi\left[(Q_b+\e)^6-Q_b^6-6Q_b^5\e\right]$$ for well chosen cut off functions $(\psi,\varphi)$ which are exponentially decaying to the left, and polynomially growing to the right. The leading order quadratic term relates to the linearized Hamiltonian and is coercive from our choice of orthogonality conditions: $$\matchal F \gtrsim \|\e\|_{H^1_{\rm loc}}^2.$$ The essential feature now is the structure of the cut off which is manufactured to also reproduce on the ground state the leading order virial quadratic form which measures some repulsivity properties of the linearized operator $L'$ as derived in \cite{MMannals}, and leads to the {\it Lyapounov monotonicity}: 
\be
\label{cnkneneoe}
\frac{d}{ds}\left\{\frac{\mathcal F}{\lambda^{2j }}\right\}+\frac{\|\e\|_{H^1_{\rm loc}}^2}{\lambda^{2j }}\lesssim \frac{|b|^4}{\lambda^{2j }}, \ \ j=0,1.
\ee
The $b^4$ term relates to the error in the construction of the $Q_b$ profile as an approximate solution to \fref{noehnfoeheo}. The case $j=0$ in \fref{cnkneneoe} is a scaling invariant estimate which will be crucial in all three regimes to control the dynamics, and the case $j=1$ is an $H^1$ improvement in the blow up regime $\lambda\to 0$.\\

\noindent \emph{(iv). Rigidity (section \ref{sectionfour})}

The combinaison of the modulation equations \fref{vnkonveoneonv|} with the dispersive bound \fref{cnkneneoe} leads roughly speaking to\footnote{see \fref{conrolbintegre}}:
\be
\label{ridgiti}
 \frac{b(t)}{\l^2(t)} \sim \ell,
\ee 
for some constant $\ell$. Then the selection of the dynamics depends on:

- either  
$
\forall t,$  $|b(t)|\lesssim \|\e(t)\|^2_{H^1_{\rm loc}},
$

- or there exists a time $t_1^*$ such that 
$
|b(t_1^*)|\gg \|\e(t_1^*)\|^2_{H^1_{\rm loc}}.$

The second condition means that the finite dimensional dynamics measured by $b$ takes control over the infinite dimensional dynamics at some time $t_1^*$. We claim that {\it this regime is trapped} and that $|b(t)|\gg \|\e(t)\|^2_{H^1_{\rm loc}}$ for $t\geq t_1^*$ as long as the solution remains in the tube $\mathcal T_{\alpha^*}$. Reintegrating the modulation equations driven to leading order by \fref{appromodulation}, we show that this leads  to (Blow up) if $b(t_1^*)>0$ and to (Exit) if $b(t_1^*)<0$. The first case leads to the threshold (Soliton) dynamics.
The condition on $b(t_1)$ which determines the (Blow up) and (Exit) regimes is by continuity of the flow an open condition on the data.

\medskip

\noindent \emph{(v). End of the proof of Theorem  \ref{th:1}}

The case $E_0\leq 0$ is treated in section 5. here the variational characterization of $Q$ and a standard concentration compactness ensures that the solution must remain in $\mathcal T_{\alpha^*}$, and then we show (Blow up) by proving that (Soliton) cannot happen. For $E_0<0$, this a classical consequence of the energy conservation law and  local dispersive estimates (asymptotic stability) obtained in the previous step. The case $E_0=0$ is substantially more subtle, and we show that (Soliton) behavior at zero energy implies $L^2$ compactness, and hence asymptotic stability implies that the solution has minimal mass, and hence is exactly a solitary wave.\\
Finally, we complete in section 6 the sharp description of the singularity formation and the universality of the focusing bubble stated by Theorem \ref{th:1}. This requires propagating the dispersive estimates, which involve local norms around the soliton, further away on the left of the soliton, in particular to compute the trace of the reminder \fref{th:1:4}. This is done using suitable $H^1$ monotonicity formula in the spirit of the analysis in \cite{Mjams}, \cite{MMjams}.

\medskip

\noindent{\bf Acknowledgement.} P.R. is supported by the French ERC/ANR project SWAP. Part of this work was done while P.R was visiting the ETH, Zurich, which he would like to thank for its kind hospitality.
This work is also partly supported by the project ERC 291214 BLOWDISOL.

%%%%%%%%%%%%%%%%%%%%%%%%%%%%%%%%%%%%%%%%%%%%%%
%%%%%%%%%%%%%%%%%%%%%%%%%%%%%%%%%%%%%%%%%%%%%%

\section{Nonlinear profiles and decomposition close to the soliton}

%%%%%%%%%%%%%%%%%%%%%%%%%%%%%%%%%%%%%%%%%%%%%%
%%%%%%%%%%%%%%%%%%%%%%%%%%%%%%%%%%%%%%%%%%%%%%

In this section, we introduce refined nonlinear profiles following the strategy developed in 
\cite{MRgafa}, \cite{RR2009}. The strategy is to produce approximate solutions to the renormalized flow \eqref{noehnfoeheo} which are as well localized as possible, which turns out to lead to a strong rigidity for the scaling law.

\subsection{Structure of the linearized operator}
Denote by
       ${\mathcal Y}$  the set of functions $f\in C^\infty({{\mathbb R}},{{\mathbb R}})$ such that
\begin{equation}\label{Y}
    \forall k\in \mathbb{N},\ \exists C_k,\, r_k>0,\ \forall y\in {{\mathbb R}},\quad |f^{(k)}(y)|\leq C_k (1+|y|)^{r_k} {e^{-|y|}}. 
\end{equation}

We recall without proof the following standard result (see e.g. \cite{W1985}, \cite{MMgafa}).

\begin{lemma}[Properties of the linearized operator $L$]\label{cl:1}
The self-adjoint operator $L$ on $L^2$ satisfies: 
    \begin{itemize}
        \item[{\rm (i)}] Eigenfunctions : $L Q^{3} = - 8 Q^{3}$; $L Q'=0$; ${\rm Ker} L=\{a Q', a \in {{\mathbb R}}\}$;
        \item[{\rm (ii)}] Scaling : $L ( {\Lambda} Q ) = - 2Q$;
        \item[{\rm (iii)}] For any   function $h \in L^2({{\mathbb R}})$ orthogonal to $Q'$ for the $L^2$ scalar product, 
        there exists a unique function $f \in H^2({{\mathbb R}})$ orthogonal to $Q'$ such that $L f=h$; moreover,
        if $h$ is even (respectively, odd), then $f$ is even (respectively, odd).
       \item[{\rm (iv)}]    If $f\in L^2({{\mathbb R}})$ is such that $L f \in {\mathcal Y}$, then $f\in {\mathcal Y}$.
       \item[{\rm (v)}] Coercivity of $L$:   for all $f\in H^1$,
       \begin{equation}\label{eq:A5}
       (f,Q^3)=(f,Q')=0 \ \Rightarrow \ (L f,f)\geq \|f\|_{L^2}^2.
       \end{equation}
       Moreover, there exists $\mu_0>0$ such that for all $f\in H^1$,
       \begin{equation}\label{eq:A6}
      (L f,f)\geq \mu_0\|f\|_{H^1}^2-\frac{1}{\mu_0}\left[(\e,Q)^2+(\e,y\Lambda Q)^2+(\e,\Lambda Q)^2\right].
       \end{equation}
    \end{itemize}
\end{lemma}
 
\subsection{Definition and estimates of localized profiles}

We now look for a slowly modulated approximate solution to the renormalized flow \fref{transforu}, \fref{noehnfoeheo}. In fact, in our setting, an order $b$ expansion is enough.

\begin{proposition}[Nonlocalised profiles]
\label{cl:2}
There exists a unique smooth function $P$ such that $P' \in \mathcal{Y}$ and
\be
\label{eq:23}
 (L  P)'={\Lambda} Q, \ \  \lim_{y \to -\infty}   P(y) = \frac 12 \int Q,  \ \ \lim_{y \to +\infty}   P(y) =0,
 \ee
 \be
 \label{PQ}
 (P, Q) = \frac 1{16} \left(\int Q\right)^2 > 0,\quad (P,  Q')=0.
 \ee
 Moreover, $${\widetilde Q_b} =Q+b P $$ is an approximate solution to \fref{noehnfoeheo} in the sense that:
\begin{equation}
\label{eq:24}
\left\|\left({\widetilde Q_b} ''-{\widetilde Q_b} +{\widetilde Q_b} ^5\right)'+b {\Lambda} {\widetilde Q_b} \right\|_{L^{\infty}}\lesssim b^2.
\end{equation}
\end{proposition}

\begin{proof}[Proof of Proposition \ref{cl:2}]

We look for $ P $ of the form $ P = \widetilde   P -\int_{y}^{+\infty} {\Lambda} Q.$
Since $\int {\Lambda} Q = -\frac 12 \int Q$,
the function $y\mapsto \int_{y}^{+\infty} {\Lambda} Q$ is bounded and has decay only as
$y\to +\infty$.
Then,  $ P $ solves \eqref{eq:23} if 
$$ ({L} \widetilde   P )' = {\Lambda} Q + \left({L} \int_{y}^{+\infty} {\Lambda} Q\right)'
=R'\quad \text{
where} \quad R=({\Lambda} Q)'- 5 Q^4 \int_{y}^{+\infty} {\Lambda} Q.$$
Note that $R \in \mathcal{Y}$.
Since $\int ( {\Lambda} Q ) Q=0 $ and ${L} Q'=0$, 
we have $\int R Q'=-\int R' Q=0$ and so from Lemma \ref{cl:1},
there exists a unique $\widetilde   P \in \mathcal{Y}$ (and smooth), orthogonal to $Q'$,
such that ${L} \widetilde   P =R$.
Then  $ P =\widetilde   P  -\int_{y}^{+\infty} {\Lambda} Q$ satisfies \eqref{eq:23} and
$\int P Q'=0$. We now compute from ${L} ({\Lambda} Q)= -2 Q$:
\begin{equation}\label{eq:209}
2 \int  P  Q = -\int ({L}  P ) {\Lambda} Q =
\int   {\Lambda} Q \int_y^{+\infty} {\Lambda} Q 
= \frac 12 \left(\int {\Lambda} Q\right)^2
=  \frac 18 \left(\int Q\right)^2
.
\end{equation}

Finally, for $\widetilde Q_b = Q+bP$, we have 
\begin{align*}
&(\widetilde Q_b''- \widetilde Q_b + \widetilde Q_b^5)' + b \Lambda Q
\\& = b(-(LP)' + \Lambda Q) + b^2 ((10 Q^3P^2)' + \Lambda P)
+b^3 (10 Q^2P^3)' + b^4(5Q P^4)' + b^5 (P^5)'
\end{align*}
which yields \eqref{eq:24}.
\end{proof}

\begin{remark}
Since 
$
\int {\Lambda} Q = - \frac 12 \int Q \neq 0,
$ a solution $P$ of $(LP)'={\Lambda} Q$ cannot belong to
 $L^2({\mathbb R})$. We have chosen the only solution $ P $  which converges to $0$  at $+\infty$ and orthogonal to $Q'$. 
The fact that $P$ displays a non trivial tail on the left from \fref{eq:23} is an essential feature of the critical (gKdV) problem and will be central in the derivation of the blow up speed, see the proof of \eqref{eq:bl2}. Such nonlocal profile are substitute to dispersive tail (see a similar use in  \cite{MMcol1}).
\end{remark}

We now proceed to a simple localization of the profile to avoid some artificial growth at $-\infty$.
Let $\chi\in {\cal C}^{\infty}({\mathbb R})$ be such that $0\leq \chi \leq 1$, $\chi'\geq 0$ on $\mathbb{R}$,
$\chi\equiv 1$ on $[-1,+\infty)$, $\chi\equiv 0$ on $(-\infty,-2]$.
We fix:
\begin{equation}\label{defgamma}\gamma = \frac 34,
\end{equation}
(note that any $\gamma \in (2/3,1)$ works and $3/4$ has no specific meaning here)
and define the localized profile:
\begin{equation}\label{eq:210}
\chi_b(y)= \chi\left(|b|^{\gamma} {y} \right),\quad 
Q_b(y) = Q(y) + b \chi_b(y)   P (y).
\end{equation}

\begin{lemma}[Definition of localized profiles and properties]\label{le:1} There holds for $|b|<b^*$ small enough:\\
{\rm (i) Estimates on $Q_b$:} For all $y \in \mathbb{R}$,
\begin{align}
&	|Q_b(y)|\lesssim  e^{-|y|} + |b| \left(  {\mathbf{1}}_{[-2,0]}(|b|^\gamma y) +  e^{-\frac {|y|}{2}} \right),
\label{eq:001}\\
&	|Q_b^{(k)}(y)|\lesssim  e^{-|y|} +   |b|  e^{-\frac {|y|}{2}} +|b|^{1+k \gamma} {\mathbf{1}}_{[-2,-1]}(|b|^\gamma y),\quad  \text{for $k\geq 1$}.
\label{eq:002}
% \\& \left| \frac {\partial Q_b}{\partial b} (y) \right| \lesssim  \label{eq:002bis}
\end{align} 
where ${\mathbf{1}}_I$ denotes the characteristic function of the interval $I$.

\noindent {\rm (ii) Equation of $Q_b$:}
Let
\begin{equation}\label{eq:201}
-\Psi_b=\left(Q_b''- Q_b+ Q_b^5\right)'+b {\Lambda} Q_b .
\end{equation}
Then, for all $y\in \mathbb{R}$,
\begin{align}
& |\Psi_b(y)|\lesssim |b|^{1+ \gamma}  {\mathbf{1}}_{[-2 ,-1]}(|b|^{\gamma}y) 
+  b^2  \left(e^{-\frac {|y|} 2}+ {\mathbf{1}}_{[- 2,0]}(|b|^{\gamma}y) \right),\label{eq:202}
\\
& |\Psi_b^{(k)}(y)|\lesssim    |b|^{1+ (k+1) \gamma}  {\mathbf{1}}_{[-2 ,-1]}(|b|^{\gamma}y)  + b^2   e^{-\frac {|y|} 2} ,\quad 
 \text{for $k\geq 1$}. \label{eq:203}	
% \\ & \left| \frac {\partial \Psi_b}{\partial b} (y) \right| \lesssim  \label{eq:203bis}
\end{align}
{\rm (iii) Mass and energy properties of $Q_b$:}
\begin{align}
& \left|\int Q_b^2 - \left( \int Q^2 + 2b \int PQ \right)\right| \lesssim |b|^{2-\gamma}, \label{eq:204}\\
& \left| E(Q_b) + b \int PQ\right|\lesssim b^2 .
\label{eq:205}
\end{align}

\end{lemma}
\begin{proof}[Proof of Lemma \ref{le:1}]
{\it Proof of (i):} First, from \eqref{eq:Q}, for all $k\geq0 $, $ | Q^{(k)}(y) | \lesssim e^{-|y|}$ on $\mathbb{R}$.
Since $P'\in \mathcal{Y}$ and $\lim_{+\infty} P= 0$, we have $|P(y)|\lesssim e^{-\frac {|y|}2}$ for $y>0$.
Estimates \eqref{eq:001} and \eqref{eq:002} then follow from the definition of $\chi$.

{\it Proof of (ii):} Expanding $Q_b=Q+b\chi_bP$ in the expression of $\Psi_b$ and using 
$Q''-Q + Q^5 =0$, $(L  P)'={\Lambda} Q$, we find
\bea
\label{formulapsib}
-\Psi_b & = & b (1-\chi_b) \Lambda Q\\
\nonumber & +& b \left( (\chi_b)_{yyy}  P + 3 (\chi_b)_{yy}  P' + 3 (\chi_b)_y  P'' - (\chi_b)_y  P +5 (\chi_b)_y Q^4 P\right) \\
\nonumber & +  &b^2 \left(\left(10 Q^3 \chi_b^2 P^2\right)_y+P\Lambda \chi_b  + \chi_b y  P'\right) \\
\nonumber & +  &b^3 \left( 10 Q^2 \chi_b^3  P^3\right)_y 
+  b^4 \left( 5 Q \chi_b^4  P^4\right)_y   
+  b^5 \left( \chi_b^5P\right)_y. 
\eea
Therefore, estimates \eqref{eq:202} and \eqref{eq:203} follow from the properties of   $Q$, $\chi$
and $P$. In particular, note that:
$$
| b (1-\chi_b) \Lambda Q| \lesssim |b| e^{-\frac 34 |y|} \mathbf{1}_{(-\infty,-1]}(|b|^\gamma y) 
\lesssim |b| e^{- \frac {|b|^{-\gamma}}{4}} e^{-\frac {|y|}2}
\lesssim  |b|^2 e^{-\frac {|y|}2},
$$
 $$  b^2|P\Lambda \chi_b|\lesssim b^2  (e^{-\frac {|y|} 2}+ {\mathbf{1}}_{[- 2,-1]}(|b|^{\gamma}y)).$$
 
{\it Proof of (iii):} We first estimate from the explicit form of $P$:
$$\int \chi_b^2 P^2 \sim_{b\to 0} C_0^2 |b|^{-\gamma}
$$ for some universal constant $C_0>0$. Estimate \eqref{eq:204} now follows from 
$$
\int Q_b^2 = \int Q^2 + 2 b \int \chi_b P Q + b^2 \int \chi_b^2 P^2
$$
and then:
$$\int Q_b^2 \geq \int Q^2 + 2 b \int PQ - C_0^2 |b|^{2-\gamma}, \ \ \ \ \|Q_b - Q\|_{L^2}\sim_{b\sim 0} C_0 |b|^{1-\frac \gamma 2}.
$$ Finally, expanding
$Q_b = Q+ b \chi_b P$ in $E(Q_b)$, we get
\begin{equation*}
E(Q_b) = E(Q) - b\int \chi_b  P (Q''+Q^5) + O(b^2) 
\end{equation*}
and using $E(Q)=0$ and $Q''+Q^5 = Q$ yields \eqref{eq:205}. 
\end{proof}

%%%%%%%%%%%%%%%%%%%%%%%%%%%%%%

\subsection{Decomposition of the solution using refined profiles}

%%%%%%%%%%%%%%%%%%%%%%%%%%%%%%

In this paper, we work with an $H^1$ solution $u$ to \fref{kdv} a priori in the modulated tube $\mathcal T_{\alpha^*}$ of functions near the soliton manifold. More explicitely, we assume that there exist $(\lambda_1(t),x_1(t))\in \mathbb{R}^*_+\times \RR$ and $\e_1(t)$ such that
$$\forall t\in [0,t_0), \ \ u(t,x)=\frac{1}{\lambda^{\frac 12}_1(t)}(Q+\e_1)\left(t,\frac{x-x_1(t)}{\lambda_1(t)}\right)$$
with, $\forall t\in [0,t_0),$
\be
\label{hypeprochien}
 \|\e_1(t)\|_{L^2} \leq \kappa \leq \kappa_0
\ee
for a small enough universal constant $\kappa_0>0$. We then have the following standard refined modulation lemma:

\begin{lemma}[Refined modulated flow]
\label{le:2}
Assuming \fref{hypeprochien}, there exist continuous functions $({\lambda}, x,{b}):[0,t_0]\to (0,+\infty)\times \mathbb{R}^2$ such that
\begin{equation}\label{defofeps}
\forall t\in [0,t_0],\quad
\varepsilon(t,y)={\lambda}^{\frac 12}(t) u(t,{\lambda}(t) y + x(t)) - Q_{{b}(t)}(y)
\end{equation}
satisfies the orthogonality conditions:
\be
\label{ortho1}
(\e(t), y {\Lambda}   Q )=(\e(t),\Lambda Q) =(\e(t),Q)=0.
\ee
Moreover, 
\be
\label{controle}
\|\varepsilon(t)\|_{L^2}+ |{b}(t)|+ \left| 1- \frac {\l(t)}{\l_1(t)}\right| \lesssim \delta(\kappa),\quad
\|\e(t)\|_{H^1} \lesssim \delta(\|\e(0)\|_{H^1}).
\ee
\end{lemma}

 \begin{remark} The main novelty here with respect to \cite{MMannals}, \cite{Mjams}, \cite{MMjams} is the use of the modulation parameter $b$ which allows for the extra degeneracy $(\e,Q)=0$.
 At the formal level, the parameter $b$ now plays the role of $(\e,Q)$ for the previous work \cite{MMjams}.
\end{remark}
 
\begin{proof} Lemma \ref{le:2} is a standard consequence of the implicit function theorem applied in $L^2$. We omit the details and refer for example to \cite{MRgafa} for a proof with similar $Q_b$ profiles for the (NLS) case. The heart of the proof is  the non degeneracy of the Jacobian matrix:
$$\left|
  \begin{array}{cc}
    ({\Lambda} Q,{\Lambda} Q) & ({\Lambda} Q,Q) \\ 
    ( P,{\Lambda} Q) & ( P,Q)\\ 
  \end{array}\right|=({\Lambda} Q,{\Lambda} Q) ( P,Q) \not= 0,$$ from
$$\frac {\partial}{\partial {\lambda}}\left\{{\lambda}^{1/2}  Q_b({\lambda}  y )\right\}_{| {\lambda}=1, b=0 }=
{\Lambda} Q,\quad 
 \frac {\partial }{\partial b}\left\{{\lambda}^{1/2}  Q_b({\lambda}  y) \right\}_{|\lambda=1, b=0}=  P,$$ and the explicit computations $$({\Lambda} Q,Q)=0, \ \ ( P,Q)= \frac 1{16} \int Q^2\not =0.$$
\end{proof}

%%%%%%%%%%%%%%%%%%%%%%%%%%%%%%%%%%%%%%%%%%%
%%%%%%%%%%%%%%%%%%%%%%%%%%%%%%%%%%%%%%%%%%%

\subsection{Modulation equations}

%%%%%%%%%%%%%%%%%%%%%%%%%%%%%%%%%%%%%%%%%%%
%%%%%%%%%%%%%%%%%%%%%%%%%%%%%%%%%%%%%%%%%%%

In the framework of Lemma \ref{le:2}, we introduce the new time variable
\be
\label{rescaledtime}
{s}=   \int_0^t \frac {dt'}{{\lambda}^3(t')} \quad 
\text{or equivalently}\quad
\frac {d{s}}{dt} = \frac 1{{\lambda}^3}.
\ee
All functions depending on $t\in [0,t_0]$, for some $t_0>0$ can now be seen as depending on ${s}\in [0,s_0]$, where
$s_0=s(t_0)$. We now claim the following properties of the decomposition of $u(t)$, possibly taking a smaller universal $\kappa_0>0$.

\begin{lemma}[Modulation equations]
\label{le:3}
Assume for all $t\in [0,t_0)$,
\be\label{petitL2H1loc}
 \|\e(t)\|_{L^2} \leq \kappa \leq \kappa_0 \quad \hbox{and} \quad
 \int \e_y^2(t,y) e^{-\frac {|y|}2} dy \leq \kappa_0
\ee
for a small enough universal constant $\kappa_0>0$. Then the map ${s}\in [0,s_0]\mapsto ({\lambda}({s}), x({s}),{b}({s}))$ is $\mathcal C^1$  and the following holds:\\
{\em (i) Equation of $\varepsilon$}: For all ${s}\in [0,s_0]$,
\begin{align}  \nonumber  
  \varepsilon_{s} - (L \varepsilon)_y + {b}{\Lambda} \varepsilon
&= \left(\frac {{\lambda}_{s}}{{\lambda}}+{b}\right)({\Lambda} Q_b+{\Lambda} \varepsilon)
+ \left(\frac { x_{s}}{\lambda} -1\right) (Q_b + \varepsilon)_y \\ & 
+ \Phi_{b}  
 + \Psi_{{b}} - (R_{{b}}(\varepsilon))_y -(R_{\rm NL}(\varepsilon))_y,
\label{eqofeps}\end{align}
where $\Psi_{b}$ is defined in \eqref{eq:201} and
\begin{align}
&\Phi_{b} =- {b}_{s} \left(\chi_b   + \gamma   y (\chi_b)_y\right) P, \label{rbnl0}
\\ 
&R_{{b}}(\varepsilon)= 5 \left(Q_b^4 - Q^4\right) \varepsilon,
\quad 
R_{\rm NL}(\varepsilon)=(\varepsilon+Q_b)^5
- 5 Q_b^4 \varepsilon - Q_b^5.\label{rbnl}
\end{align}
{\em (ii) Estimates induced by the conservation laws}: on $[0,s_0]$, there holds
\be
\label{twobound}
\|\e\|^2_{L^2}\lesssim |b|^{\frac 12}+\left|\int u_0^2-\int Q^2\right|,
\ee
\be
\label{energbound}
\left|2 \lambda^2E_0+\frac b{8}\|Q\|_{L^1}^2-   \| \e_y\|_{L^2}^2  \right|\lesssim b^2+  \|\e(s)\|_{L^2}^2   + \delta(\|\e\|_{L^2}) \|\e_y\|_{L^2}^2.
\ee
\\
{\em (iii) $H^1$ modulation equations}: for all ${s}\in [0,s_0]$,
\begin{align}
& \left|\frac {{\lambda}_{s}}{\lambda} + {b}\right| + \left| \frac { x_{s}}{\lambda} - 1 \right| \lesssim
\left(\int \varepsilon^2 {e^{-\frac{|y|}{10}}} \right)^{\frac 12} +  b^2;\label{eq:2002}\\
& |{b}_{s}|  \lesssim
  \int   \varepsilon^2   {e^{-\frac{|y|}{10}}}   +  b^2.\label{eq:2003}
\end{align}
{\em (iv) Refined modulation equations in $\mathcal{A}$:}  Assuming the following uniform $L^1$ control on the right:
\be
\label{cnbooeoe}
\forall t\in [0,t_0), \ \ \int_{y>0}|\e(t)|\lesssim \delta(\kappa_0),
\ee
then the quantities $J_1$ and $J_2$   below are well-defined and satisfy the following:
\begin{itemize}
\item{Law of ${\lambda}$:} let
\begin{equation}\label{rho1}
\rho_1(y)=  \frac 4{\left(\int Q\right)^2}  \int_{-\infty}^y {\Lambda} Q, \ \ {J_1}({s})=  (\varepsilon(s),\rho_1), 
\end{equation}
then for some universal constant $c_1$,
\begin{equation}\label{eq:2004} 
\left|  {\frac{{\lambda}_s}{{\lambda}}}   + {b} +c_1 b^2   - 2 \left(({J_1})_{s} + \frac 12 {\frac{{\lambda}_s}{{\lambda}}}  {J_1} \right) \right| \lesssim  \int \varepsilon^2 {e^{-\frac{|y|}{10}}}   +  |b|\left( \int \varepsilon^2 {e^{-\frac{|y|}{10}}} \right)^{\frac 12}+|b|^3. 
\end{equation}
\item{Law of $b$:} let
\be
\label{rho2} 
\rho_2 = \frac {16}{\left( \int  Q\right)^2} \left(\frac {({\Lambda} P, Q)}{\|{\Lambda} Q\|_{L^2}^2} {\Lambda} Q + P-\frac 12 \int Q\right)
- 8 \rho_1,\ \  
{J_2}({s})=    (\varepsilon({s}),  \rho_2), 
\ee
then for some universal constant $c_2$.
\begin{equation}\label{eq:2006}
\left|b_s + 2b^2 +c_2 b^3  + b \left(({J_2})_s + \frac 12 {\frac{{\lambda}_s}{{\lambda}}} {J_2}\right) \right| \lesssim  \int \varepsilon^2 {e^{-\frac{|y|}{10}}}   +  |b|^4.
\end{equation}
\item{Law of $\frac b{{\lambda}^2}$:} let 
\be
\label{defjgg}
\rho = 4\rho_1 + \rho_2 \in \mathcal{Y}, \ \ J (s) = (\e(s),\rho),
\ee
then, for $c_0=c_2-2c_1$,
\begin{equation}\label{eq:bl2}
\left|\frac d{ds}\left( \frac b{{\lambda}^2}\right)  + \frac b{{\lambda}^{2}} \left(J_s+ \frac 12 {\frac{{\lambda}_s}{{\lambda}}} J\right) +c_0 \frac {b^3}{\l^2}  \right|
\lesssim \frac 1{{\lambda}^2}  \left( \int \varepsilon^2 {e^{-\frac{|y|}{10}}} + |b|^4\right).
\end{equation}
\end{itemize}
\end{lemma}

\begin{remark} 
It is a remarkable algebraic fact that the equation of $\frac b{\lambda^2}$   \fref{eq:bl2} is related to $\rho\in \mathcal{Y}$ which means  that $J$ is an $L^2$ quantity,  easier to control than ${J_1}$ and ${J_2}$ separately.

The  equations \eqref{eq:2004}, \eqref{eq:2006} correspond to a sharp improvement -- after integration in time -- of the rough estimates of (iii). However, they hold for initial data in weighted spaces such as $\mathcal{A}$. Here we are facing an intrinsic difficulty of the (gKdV) equation which is that the null space of the full linearized operators $(L)'$ involves badly localized terms, and hence getting geometrical parameters which are quadratic forcing terms of the $\e$ equation \fref{eqofeps} requires some $L^1$ control of the solution on the right. Formally, \eqref{eq:2004}, \eqref{eq:2006} are the sharp analogues of the leading order dynamical system: $$\lsl=-b, \ \ \left(\frac{b}{\l^2}\right)_s=\frac{b_s+2b^2}{\l^3}=0.$$ \end{remark}

\begin{proof}[Proof of Lemma \ref{le:3}]
 {\it Proof of (i)}: The equation of $\varepsilon$, ${\lambda}$, $ x$, ${b}$ follows by direct computations from the equation of $u(t)$. In particular, we use
$$
\frac {\partial }{\partial b}(Q_b) = \frac {\partial}{\partial b} ( b \chi_b(y)) P
=\left( \gamma |b|^\gamma  y \chi'(|b|^\gamma y) + \chi(|b|^\gamma y) \right)P 
= (\chi_b + \gamma y (\chi_b)_y) P.
$$
The rest of the computation is done in Lemma 1 of \cite{MMgafa} for example.

\medskip

{\it Proof of (ii)}: We write down the $L^2$ conservation law: $$\int Q_b^2-\int Q^2+\int\e^2+2(\e,Q_b)=\int u_0^2-\int Q^2$$ and we deduce from \fref{eq:204} using the orthogonality condition \fref{ortho1} that
$$\int\e^2\lesssim |b|+|b|^{1-\gamma}\|\e\|_{L^2}+\left|\int u_0^2-\int Q^2\right|.$$ Then \fref{twobound} follows since $\gamma=3/4$. 

Now, we write down the conservation of energy and use \fref{eq:205}, the equation of $Q$  and the orthogonality condition $(\e,Q)=0$ to estimate:
\bee
2{\lambda}^2 E(u_0) &  = &  2E(Q_b) - 2\int  \varepsilon  (Q_b)_{yy}  + \int \varepsilon_y^2 - \frac 13 \int \left((Q_b+\varepsilon)^6  -Q_b^6\right)\nonumber\\
	&  = & - 2 {b} (P,Q) + O(b^2)+\int\e_y^2 \\
	& - &  2\int \varepsilon\left[(Q_b - Q)_{yy}+(Q_b^5-Q^5)\right]\\
	& - & \frac 13 \int \left[(Q_b+\varepsilon)^6 - Q_b^6-6Q_b^5\e \right].
\eee
We estimate all terms in the above identity. By the properties of $Q_b$:
\bee
\left|\int \varepsilon\left[(Q_b - Q)_{yy}+(Q_b^5-Q^5)\right]\right| &\lesssim& |b|\left( \int \varepsilon^2 {e^{-\frac{|y|}{10}}} \right)^{\frac 12} +|b|^{1+2\gamma} \int_{ -2 |b|^{-\gamma}<y<0} |\e|\\
& \lesssim & b^2+\|\e\|_{L^2}^2 .
\eee
The nonlinear terms are estimated by the homogeneity of the nonlinearity which implies:
\bee
\left|\int \left[(Q_b+\varepsilon)^6   - Q_b^6-6Q_b^5\e\right]\right|& \lesssim & \int  |Q_b|^4\e^2+ |\e|^6 \\
& \lesssim & \|\e\|_{L^2}^2+\|\e_y\|_{L^2}^2\|\e\|_{L^2}^4  .
\eee
 The collection of above estimates yields \fref{energbound}.
 
 \medskip
 
{\it Proof of (iii)}: We   sketch the standard computations\footnote{See e.g. \cite{MMgafa}, Lemma 4, for similar computations.} leading to \eqref{eq:2002} and \eqref{eq:2003}.
Differentiating   the orthogonality conditions $(\varepsilon, {\Lambda} Q) =( \varepsilon, y {\Lambda}  Q )=0$, using the equation of $\varepsilon$ and estimate \eqref{eq:202}, we obtain:
\begin{align*}
& \left| \left( {\frac{{\lambda}_s}{{\lambda}}} + {b} \right) -  \frac {( \e,L(\Lambda Q)')}{\|\Lambda Q\|_{L^2}^2}  \right| + \left|\left( {\frac{x_s}{{\lambda}}} - 1\right) - \frac {( \e,L(y \Lambda Q)')}{\|\Lambda Q\|_{L^2}^2}  \right|
 \\ & \lesssim 
 \left( \left|  {\frac{{\lambda}_s}{{\lambda}}}+ {b}  \right| + \left| {\frac{x_s}{{\lambda}}} - 1 \right|  +|b|\right)\left(|b|  + \left( \int \varepsilon^2 {e^{-\frac{|y|}{10}}} \right)^{\frac 12} \right) 
 \\ &+|{b}_{s}|+\int \varepsilon^2 e^{-\frac{|y|}{10}} + \int |\varepsilon|^5 e^{-\frac{9 }{10}|y|} 
    .
\end{align*}
We estimate the nonlinear term using the Sobolev bound\footnote{which follows by integration by parts.} and the smallness \fref{petitL2H1loc}:
$$\|\e e^{-\frac{|y|}{4}}\|^2_{L^{\infty}}\lesssim \int (|\pa_y\e|^2+|\e|^2)e^{-\frac{|y|}{2}},$$
so that
\be\label{nonlin}
\int |\varepsilon|^5 {e^{-\frac{9}{10}|y|}} 
\lesssim \|\e e^{-\frac{|y|}{4}}\|^3_{L^{\infty}}\int \varepsilon^2 e^{-\frac{|y|}{10}}
\ee
Thus,  \eqref{petitL2H1loc}, and for $\kappa_0$ small enough,
\begin{align}
& \left| \left( {\frac{{\lambda}_s}{{\lambda}}} + {b} \right) -  \frac {( \e,L(\Lambda Q)')}{\|\Lambda Q\|_{L^2}^2}  \right| + \left|\left( {\frac{x_s}{{\lambda}}} - 1\right) - \frac {( \e,L(y \Lambda Q)')}{\|\Lambda Q\|_{L^2}^2}  \right|
 \nonumber \\ & \lesssim 
   |b|^2+
 |{b}_{s}|+\int \varepsilon^2 e^{-\frac{|y|}{10}},\label{eq:2102bis}
\end{align}
and
\be\label{eq:2102}
  \left|   {\frac{{\lambda}_s}{{\lambda}}} + {b}    \right| + \left| {\frac{x_s}{{\lambda}}} - 1   \right|
  \lesssim 
   |b|^2+
 |{b}_{s}|+\left(\int \varepsilon^2 e^{-\frac{|y|}{10}}\right)^{\frac 12}.
\ee

Next, differentiating in time $s$ the relation 
$(\varepsilon, Q)=0,$   using the $\e$ equation, and the following algebraic facts  $LQ'=0$, $(Q, {\Lambda} Q)=(Q, Q')=0$, $(\varepsilon, {\Lambda} Q)=0$, the nondegeneracy $(P,Q)\neq 0$ and the bounds \fref{eq:202}, \fref{eq:203}, we find after integration by parts and Sobolev estimates \eqref{nonlin}:
\begin{equation}\label{eq2103}
|{b}_{s}| \lesssim \left| {\frac{{\lambda}_s}{{\lambda}}} + {b}\right|^2 + \left| {\frac{x_s}{{\lambda}}} - 1 \right|^2 +  {|b|}^2 +  \int   \varepsilon^2   {e^{-\frac{|y|}{10}}}.
\end{equation}
(see below a much detailed computation of $b_s$).

Combining \fref{eq:2102} and \fref{eq2103}   yields \eqref{eq:2002}, \eqref{eq:2003}.

\medskip

{\it Proof of (iv)}:
To begin, we claim the following sharp equation for $b$:
\bea
\label{eq:p13}
\nonumber && b_s + 2 b^2 + cb^3 -\frac {16}{(\int Q)^2} b \left[ 
\frac {(\Lambda P, Q)}{\|{\Lambda} Q\|_{L^2}^2} (\varepsilon, L({\Lambda} Q)') + 20 ( \varepsilon, P Q^3 Q')\right]\\
&& =  O(|b|^4) + O\left(\int \varepsilon^2 {e^{-\frac{|y|}{10}}}\right),
\eea
where $c$ is a universal constant.

To prove \eqref{eq:p13}, we take the scalar product of the equation of $\e$ by $Q$ and we keep track  of  all terms up to order $|b|^3$.

In  this proof, $c$ will denote various universal constants. 
First, we use the explicit formula \fref{formulapsib} to derive:
\bee
(\Psi_b, Q) & = & - b^2\left((10Q^3\chi_b^2P^2)_y+\chi_b\Lambda P,Q\right)
- b^3 (10 Q^2 \chi_b^3 P^3 , Q')+ O(|b|^4)
\\
&=&- b^2   \left(   (10 P^2 Q^3)' +   {\Lambda} P , Q\right)
- b^3 (10 Q^2   P^3 , Q')
+O(|b|^4) \\
	& = &-  \frac{b^2}{8} \|Q\|_{L^1}^2 
	+ c_0 b^3  
	+ O(|b|^4),
\eee
where $c_0 = -10 \int P^3 Q^2 Q'$ and
where we have used in the last step the following fundamental {\it flux computation}:\bee
(\Lambda P, Q)& =& - (P, \Lambda Q) = - (P, (LP)') =(P,(P''-P+5Q^4P)')\\
& = & (P, P''' - P')+10\int Q^3Q'P^2
\eee
from which we indeed obtain
\be
\label{computationflux}
((10 P^2 Q^3)'+\Lambda P, Q)= \frac 12 \lim_{-\infty} P^2=\frac{1}{8}\left(\int Q\right)^2.
\ee
 This computation  is the key to the derivation of the blow up speed.

From \fref{PQ}:
$$(\Phi_b,Q)=-(b_s(\chi_b+\gamma y\chi_b')P,Q)=- b_s  (P,Q) + O(b^{10}) = -  \frac {b_s}{16} \left(\int Q\right)^2 + O(b^{10}).
$$
Next from \fref{PQ}:
$$\left|\left(\xsl-1\right)(Q_b,Q')\right|+
\left| \int (\Lambda Q_b) Q - b (  \Lambda P,Q) \right|\lesssim |b|^{10}.$$ 
We estimate the small linear term as follows
$$
	\int R_b(\varepsilon) Q' = 20 b \int PQ^3 Q' \varepsilon + b^2 O\left(\int \varepsilon^2 {e^{-\frac{|y|}{10}}}\right)^{\frac 12},
$$
and nonlinear terms in $\e$ are simply treated as before by \eqref{nonlin}.

Therefore, we have obtained:
\bea
&&b_s + 2 b^2 + c b^3 -\frac {16}{(\int Q)^2} b \left[
   \left( \frac {\lambda_s}{\lambda} + b\right)({\Lambda} P, Q)  + 20 (\varepsilon ,P Q^3 Q')\right]\nonumber
\\ &&= O(|b|^4) + O\left(\int \varepsilon^2 {e^{-\frac{|y|}{10}}}\right),\label{premb}
\eea

Moreover, we check that when estimating $\frac {\lambda_s}{\lambda} + b $, 
using $$|b_s + 2 b^2 |\leq |b|^3+  |b| \left(\int \varepsilon^2 e^{-\frac{|y|}{10}}\right)^{\frac 12}
+\int \varepsilon^2 e^{-\frac{|y|}{10}},$$
and keepink track of all $b^2$ terms, we can improve \eqref{eq:2102bis} into
\begin{equation}\label{eq:57bis}
\left| \left( \frac {\lambda_s}{\lambda} + b \right) - \frac {( \e,L(\Lambda Q)')}{\|\Lambda Q\|_{L^2}^2} - c  b^2 \right| \lesssim  \int \varepsilon^2 e^{-\frac{|y|}{10}} + |b| \left(\int \varepsilon^2 e^{-\frac{|y|}{10}}\right)^{\frac 12}+|b|^3.
\end{equation}

Estimate  \eqref{eq:p13} follows from \eqref{premb} and \eqref{eq:57bis}.

\medskip

 Thanks to the $L^1$ bound \fref{cnbooeoe}, for any $f \in \mathcal{Y}$, $(\varepsilon , \int_{-\infty}^y f)$ is well defined for all time and by direct computations, we have the following general formula:
\begin{align}
\nonumber   \frac{d}{ds}\left(\e,\int_{-\infty}^{y}f\right)&  =   -   (\e,Lf)+ \left(\lsl+b\right)\left(\Lambda Q_b,\int_{-\infty}^yf\right)+\lsl\left(\Lambda \e, \int_{-\infty}^y f\right) \\\nonumber
& -\left(\xsl-1\right)(Q_b,f)-   \left(\xsl-1\right)(\e,f)-b_s\left((\chi_b+\gamma y\chi_b')P,\int_{-\infty}^yf\right) \\ &+\left(\Psi_b,\int_{-\infty}^yf\right)+ (R_b(\e)+R_{\rm NL}(\e),f)\label{2.41vv}.
\end{align}
Using   \fref{eq:2002}, \fref{eq:2003}, \fref{eq:202} and \eqref{eq:p13}, we obtain
from \eqref{2.41vv}:
\bea
  \frac{d}{ds}\left(\e,\int_{-\infty}^yf\right) &= &  -(\varepsilon, L f ) + \left({\frac{{\lambda}_s}{{\lambda}}} + b\right)  \left({\Lambda} Q, \int_{-\infty}^y  f \right) 
 -\left({\frac{x_s}{{\lambda}}}-1\right)( f,Q)\nonumber \\ 
& -  & \nonumber \frac 12 {\frac{{\lambda}_s}{{\lambda}}}\left(\varepsilon,\int_{-\infty}^y  f \right) 
+  cb^2 + O\left(\int \varepsilon^2 {e^{-\frac{|y|}{10}}}\right)
\\ &+& O\left(|b| \left(\int \varepsilon^2 {e^{-\frac{|y|}{10}}}\right)^{\frac 12}\right)+O(|b|^3),
\label{dfdf}\eea
for some constant $c$ depending on $f$.
\medskip

-- Equation of $J_1$:
We apply \eqref{dfdf} to $f= \Lambda Q$ using  the following algebraic relations
  $$L {\Lambda} Q = - 2 Q, \ \ \left({\Lambda} Q,\int_{-\infty}^y {\Lambda} Q\right)= \frac 18 \left(\int Q\right)^2, \ \ \left(Q', \int_{-\infty}^y {\Lambda} Q\right)=0,$$ 
to  prove
\bee   2 ({J_1})_{s} &  = &\frac {16(\e,Q)}{(\int Q)^2}   + \left({\frac{{\lambda}_s}{{\lambda}}} + {b} \right) -   {\frac{{\lambda}_s}{{\lambda}}} {J_1}+  c{b}^2 \\ &  +& O\left(\int \varepsilon^2 {e^{-\frac{|y|}{10}}} \right)+
O\left(|b| \left(\int \varepsilon^2 {e^{-\frac{|y|}{10}}}\right)^{\frac 12}\right)+O(|b|^3).
\eee
The orthogonality conditions \eqref{ortho1} now yield \eqref{eq:2004}.
\medskip

-- Equation of $J_2$.
We now apply \fref{dfdf} to 
$\int_{-\infty}^yf=\rho_2,$ $f=\rho_2'.$
We need some computation related to $\rho_2$.
Using $\int {\Lambda} Q  = -\frac 12 \int Q$,
\bee
(\Lambda Q,\rho_2)& =&  \frac {16}{\left( \int  Q\right)^2} \left(\frac {({\Lambda} P, Q)}{\|{\Lambda} Q\|_{L^2}^2} {\Lambda} Q + P-\frac 12 \int Q,\Lambda Q\right)
- \frac{32}{(\int Q)^2} (\Lambda Q,\int_{-\infty}^y\Lambda Q)\\
& = & \frac{16}{\left( \int  Q\right)^2}\left[(\Lambda P,Q)+(\Lambda Q,P)\right]+\frac{4}{(\int Q)^2}(\int Q)^2-\frac{16}{(\int Q)^2}(\int \Lambda Q)^2=0,
\eee
and similarly:
\bee
(\rho'_2,Q)&= &\frac {16}{\left( \int  Q\right)^2} \left(\frac {({\Lambda} P, Q)}{\|{\Lambda} Q\|_{L^2}^2} ({\Lambda} Q)' + P',Q\right)
- 8 (\rho'_1,Q)=0.
\eee
Next the algebra
$$L(P') = (LP)' + 20 Q^3 Q' P = {\Lambda} Q + 20 Q^3 Q' P,$$ and the orthogonality relations
$( \varepsilon,\Lambda Q)=0,$ $( P, Q')=0$ yield: 
\bee
(\e,(L\rho_2'))&=&\frac {16}{\left( \int  Q\right)^2}\left(\e,L\left[\frac {(\Lambda P, Q)}{\|\Lambda Q\|_{L^2}^2} ({\Lambda} Q)' + P'\right]\right)-8(\e,L\rho'_1)\\
& =&\frac {16}{(\int Q)^2} \left[ 
\frac {(\Lambda P, Q)}{\|{\Lambda} Q\|_{L^2}^2} (\varepsilon, L({\Lambda} Q)') + 20  ( \varepsilon, P Q^3 Q')\right].
\eee
Injecting these relations into \fref{2.41vv} yields:
\bea
\label{eq:p13b}
\nonumber \frac d{ds} J_2 & = & - \frac {16}{(\int Q)^2} \left[ 
\frac {(\Lambda P, Q)}{\|{\Lambda} Q\|_{L^2}^2} (\varepsilon, L({\Lambda} Q)' + 20 \int \varepsilon P Q^3 Q')\right] - \frac 12 {\frac{{\lambda}_s}{{\lambda}}} J_2 \\
& + & c  b^2   +O\left(\int \varepsilon^2 {e^{-\frac{|y|}{10}}} \right)+
O\left(|b| \left(\int \varepsilon^2 {e^{-\frac{|y|}{10}}}\right)^{\frac 12}\right)+O(|b|^3).
\eea

Combining \eqref{eq:p13} and \eqref{eq:p13b} yields \fref{eq:2006}.
\medskip

-- Equation of $J$.
We now compute from \eqref{eq:2004} and \eqref{eq:2006}:
\begin{align*}
& \frac d{ds}\left( \frac b{{\lambda}^2}\right)   = \frac {b_s}{{\lambda}^2} - 2 {\frac{{\lambda}_s}{{\lambda}}} \frac {b}{{\lambda}^2}=\frac{b_s+2b^2}{\lambda^2}- \frac{2b}{\l^2}\left(\lsl+b\right) \\
& = -\frac b{{\lambda}^2} \left[ ({J_2})_s + \frac 12 {\frac{{\lambda}_s}{{\lambda}}} {J_2}\right ]
- \frac {2 b}{{\lambda}^2} \left[2 ({J_1})_s+  {\frac{{\lambda}_s}{{\lambda}}}J_1  \right]
+  (2c_1-c_2)\frac {b^3}{\l^2}\\ &+\frac 1{{\lambda}^2} O\left( \int \varepsilon^2 {e^{-\frac{|y|}{10}}} +  b^4\right)
\\
& =
- \frac b{{\lambda}^2} \left[J_s +\frac 12 {\frac{{\lambda}_s}{{\lambda}}} J\right] + (2c_1-c_2)\frac {b^3}{\l^2}
+\frac 1{{\lambda}^2} O\left( \int \varepsilon^2 {e^{-\frac{|y|}{10}}} +  b^4\right),
\end{align*}
which is \fref{eq:bl2}. 

Finally, we check that $\rho = 4 \rho_1 + \rho_2\in \mathcal{Y}$. Indeed, $\rho_1$, $\rho_2$ are exponentially localized at $-\infty$ from \fref{eq:23}. We thus need only check that $\lim_{+\infty} \rho = 0$, but it is immediate from their definitions that $\lim_{+\infty} \rho_1 =  - \frac 2{\int Q}$ and
$\lim_{+\infty} \rho_2 = \frac 8 {\int Q}$.\\
This concludes the proof of Lemma \ref{le:3}.
\end{proof}

\subsection{Kato type identities}

We recall the following standard identities which correspond to the localization of conservation laws.

\begin{claim}[Kato localization identities]\label{cl:4}
Let $g$ be any $C^3$ function and $v(t,x)$ be a solution of \eqref{kdv}. Then
\begin{enumerate}
\item $L^2$   identity:
\begin{equation}\label{L2kato}
	\frac d{dt}\int v^2 g = - 3 \int v_x^2 g' + \int v^2 g''' + \frac 53 \int v^6 g'.
\end{equation}
\item Energy   identity:
\bea
\label{enerkato} 
	 &&\frac d{dt} \int \left( v_x^2 - \frac 13 v^6 \right) g\\
	\nonumber&    &  =
	-\int \left(v_{xx}  + v^5\right)^2 g'
	- 2 \int v_{xx}^2    g' + 10 \int v^4 v_x^2 g'
	+\int  v_x^2 g'''.
 \eea
\end{enumerate}
\end{claim}

 %%%%%%%%%%%%%%%%%%%%%%%%%%%%%%% %%%%%%%%%%%%%%%%%%%%%%%%%%%%%%%
  %%%%%%%%%%%%%%%%%%%%%%%%%%%%%%% %%%%%%%%%%%%%%%%%%%%%%%%%%%%%%%

\section{Monotonicity formulas}

 %%%%%%%%%%%%%%%%%%%%%%%%%%%%%%% %%%%%%%%%%%%%%%%%%%%%%%%%%%%%%%
  %%%%%%%%%%%%%%%%%%%%%%%%%%%%%%% %%%%%%%%%%%%%%%%%%%%%%%%%%%%%%%
  
This section is devoted to the derivation of the   monotonicity tools  for solutions near the soliton manifold which are the key technical arguments of our analysis for initial data in $\mathcal{A}$. We exhibit a Lyapounov functional based on a suitable localization of the linearized Hamiltonian, which will both control pointwise dispersion around the soliton, and display some monotonicity thanks to the coercivity of the virial quadratic form proved in \cite{MMannals}. A related strategy originated in \cite{MMT}, \cite{MRS}, \cite{RR2009}, \cite{MRR}, but is implemented here in a new optimal way. Such   dispersive estimates coupled with the modulation equation for $b$ will lead to the key rigidity property for the proof of the main results of this paper.

%%%%%%%%%%%%%%%%%%%%%%%%%%%%%%

\subsection{Pointwise monotonicity}

%%%%%%%%%%%%%%%%%%%%%%%%%%%%%%

Let  $(\varphi_i)_{i=1,2},\psi\in C^{\infty}({\mathbb R})$  be such that:
\bea
\label{defphi1bis}
\varphi_i(y) =\left\{\begin{array}{lll}e^{y}\ \ \mbox{for} \ \   y<-1,\\
 1+y  \ \ \mbox{for}\ \ -\frac 12<y<\frac 12,\\ 
 y^i\ \ \mbox{for}\ \ \mbox{for}\ \ y>2,
 \end{array}\right.  \ \ \varphi_i'(y) >0, \ \ \forall y\in \RR,
\eea
\bea
\label{defphi2}
\psi(y) =\left\{\begin{array}{ll} e^{2y}\ \ \mbox{for}   \ \ y<-1,\\
 1  \ \ \mbox{for}\ \ y>-\tfrac 12,\end{array}\right.  \ \ \psi'(y) \geq 0 \ \ \forall y\in \RR.
\eea
Let $B> 100$ be a large   universal constant to be chosen in Proposition~\ref{propasymtp}, let
 $$\psi_B(y)=\psi\left(\frac yB\right), \ \ \varphi_{i,B}=\varphi_i\left(\frac yB\right), \ \ i=1,2,$$ 
 and   define the following norms on $\e$
  \begin{equation}\label{eq:no}
{\mathcal{N}_i}  (s)  = \int  \varepsilon_y^2(s,y) \psi_{B}(y)dy + \int  \varepsilon^2(s,y) \varphi_{i,B}(y)   dy,\quad i=1,2.
\end{equation} 
We also define the following  $L^2$ weighted norm for $\e$
\be\label{eq:noloc}
{\mathcal{N}_{i,\rm loc}} (s)  =    \int\varepsilon^2 (s,y)     \varphi'_{i,B} (y) dy,\quad i=1,2.
\ee

The heart of our analysis is the following monotonicity property:

\begin{proposition}[Monotonicity formula]
\label{propasymtp}
There exist $\mu>0$, $B>100$ and  $0< \kappa^*< \kappa_0$ such that the following holds. Assume that     $u(t)$ is a solution of \eqref{kdv} which satisfies \eqref{hypeprochien} on $[0,t_0]$ and thus
admits on $[0,t_0]$ a decomposition \fref{defofeps} as in Lemma \ref{le:2}. Let $s_0=s(t_0)$, and assume the  following a priori bounds: $\forall s\in [0,s_0]$,\\
{\em (H1)   smallness:} 
\be
\label{boundnwe}
\|\e(s)\|_{L^2}+|b(s)|+ \mathcal N_2(s)\leq \kappa^*;
\ee
{\em (H2) bound related to  scaling:}
\be
\label{bootassumption}
 \frac{|b(s)|+\mathcal N_2(s)}{\lambda^2(s)}\leq \kappa^*;
\ee
{\em (H3) $L^2$ weighted bound on the right:}
\be
\label{uniformcontrol}
\int_{y>0}y^{10}\e^2(s,x)dx\leq 10\left(1+\frac{1}{\l^{10}(s)}\right).
\ee
Let   the   energy--virial Lyapounov functionals for $(i,j)\in\{1,2\}^2$
\bea
\label{feps}
 {\cal F}_{i,j}  =  \int \left[\varepsilon_y^2\psi_B + \varepsilon^2(1+\mathcal J_{i,j})\varphi_{i,B}- \frac  13 \left((\varepsilon + Q_b)^6  -  {Q_b^6}   - 6 \varepsilon Q_b^5\right)\psi_B\right] ,
\eea
with 
\be
\label{defjj}
\mathcal J_{i,j}=(1-J_1)^{-(4(j-1)+2i)}-1.
\ee
Then   the following estimates hold on $[0,s_0]$:\\
{\em (i) Scaling invariant Lyapounov control}: for $i=1,2$,
\be
\label{lyapounovconrol}
\frac{d\mathcal F_{i,1}}{ds}+  \mu  \int   \left(\varepsilon_y^2+\e^2\right)   \varphi'_{i,B}  \lesssim |b|^{4}.
\ee
{\em (ii) Scaling weighted $H^1$ Lyapounov control}: for $i=1,2$,
\be
\label{lyapounovconrolbis}
\frac{d}{ds}\left\{\frac{\mathcal F_{i,2}}{\l^2}\right\}+   \frac{\mu}{\l^2}\int   \left(\varepsilon_y^2+\e^2\right)   \varphi'_{i,B}  \lesssim \frac{|b|^{4}}{\l^2}.
\ee
{\em (iii) Coercivity of $\mathcal F_{i,j}$ and pointwise bounds}: there holds for $(i,j)\in\{1,2\}^2$,
\begin{align}
\label{lowerbound}
 &
\mathcal N_i \lesssim \mathcal F_{i,j}\lesssim \mathcal N_i , \hskip 2cm
\\ &
\label{controlj}
|J_i|+|\mathcal J_{i,j}|\lesssim \mathcal N_2^{\frac{1}{2}}.
\end{align}

\end{proposition}

\begin{remark} The $L^2$ weighted bound \fref{uniformcontrol} is fundamental for the analysis and will be further dynamically bootstrapped for an initial data in $\mathcal{A}$. Also one should think of \fref{lyapounovconrol} as a scaling invariant $L^2$ bound, which is sharpened in the singular regime $\lambda\to 0$ by the $H^1$ control \fref{lyapounovconrolbis}. Finally, an important feature of Proposition \ref{propasymtp} is that we do not assume any a priori control on the scaling parameter $\lambda(s)$.

We will use several times in the proof the fact that in the definition of $\mathcal F_{i,j}$, the weight on $\e_y$ at $-\infty$ is stronger than the weight on $\e$. It follows in particular that $\mathcal F_{i,j}$ does not control
$\int   \varepsilon_y^2   \varphi'_{i,B}$. See Remark \ref{gradient} below.
\end{remark}

\begin{proof}[Proof of Proposition \ref{propasymtp}]

 {\bf step 1} Weighted $L^2$ controls at the right.
 
We first claim the controls for all $s\in [0,s_0]$,
\be
\label{weightedone}
\int_{y>0}y\e^2(s)\lesssim \left(1+\frac{1}{\l^{\frac{10}9}(s)}\right)\mathcal{N}_{1,\rm loc}^{\frac 89}(s),
\ee
\be
\label{weightedonebis}
\int_{y>0}y^2\e^2(s)\lesssim \left(1+\frac{1}{\l^{\frac {10}9}(s)}\right)\mathcal{N}_{2,\rm loc}^{\frac 89}(s),
\ee
\be
\label{lonebound}
\int_{y>0}|\e(s)|\lesssim \mathcal N_2^{\frac{1}{2}}(s).
\ee
From \fref{uniformcontrol}: for all $A>0$, $$\int_{y>0}y\e^2\leq A\int_{0\leq y\leq A}|\e|^2+\frac{1}{A^9}\int_{y>A}y^{10}|\e|^2\lesssim A{\mathcal{N}_{1,\rm loc}}+\frac{1}{A^9}\left(1+\frac{1}{\l^{10}}\right)$$ and so the optimal choice $$A^{10}{\mathcal{N}_{1,\rm loc}}= 1+\frac{1}{\l^{10}}$$ leads to the bound using the smallness \fref{boundnwe}: $$\int_{y>0}y\e^2\lesssim \frac{(1+\l^{10})^{\frac 1{10}}}{\lambda}\mathcal{N}_{1,\rm loc}^{\frac 9{10}}\lesssim \left(1+\frac1\l\right)\mathcal{N}_{1,\rm loc}^{\frac 9{10}}\lesssim  \left(1+\frac{1}{\l^{\frac{10}9}}\right)\mathcal{N}_{1,\rm loc}^{\frac 89},$$ and \fref{weightedone} is proved. Similarily,
$$\int_{y>0}y^2\e^2\leq A\int_{0\leq y\leq A}y|\e|^2+\frac{1}{A^8}\int_{y>A}y^{10}|\e|^2\lesssim A{\mathcal{N}_{2,\rm loc}}+\frac{1}{A^8}\left(1+\frac{1}{\l^{10}}\right),$$
and thus the choice $$A^9{\mathcal{N}_{2,\rm loc}}=1+\frac{1}{\l^{10}}$$ leads to the bound: $$\int_{y>0}y^2\e^2\lesssim \mathcal{N}_{2,\rm loc}^{\frac 89}\frac{(1+\l^{10})^{\frac 19}}{\l^{\frac {10}9}}\lesssim  \left(1+\frac{1}{\l^{\frac{10}9}}\right)\mathcal{N}_{2,\rm loc}^{\frac 89},$$ and \fref{weightedonebis} is proved. 

\fref{lonebound} follows from $$\int_{y>0}|\e|\lesssim \|(1+y)\e\|_{L^2(y>0)}\lesssim \mathcal N_2^{\frac 12}.$$

Finally, we observe that \fref{lonebound} implies \fref{controlj}. In particular,   
the quantities $\mathcal J_{i,j}$ are well defined, and so are $\mathcal F_{i,j}$.

\medskip

{\bf step 2} Algebraic computations on $\mathcal F_{i,j}$. 
We compute
\begin{align*}
&\lambda^{2(j-1)} \frac d{d{s}}\left\{\frac{{\cal F_{i,j}}}{\lambda^{2(j-1)}}\right\}  \\ & =
 2\int \psi_B(\varepsilon_y)_s \varepsilon_y + 2\varepsilon_s \left[(1+\mathcal J_{i,j})\varepsilon \varphi_{i,B} - \psi_B\left[(\varepsilon+ Q_b)^5 -Q_b^5\right]\right]\\
&  +(\mathcal J_{i,j})_s\int\varphi_{i,B}\e^2-2 \int \psi_B(Q_b)_s \left[(\varepsilon+Q_b)^5 - Q_b^5 - 5 \varepsilon Q_b^4\right]\\
& - 2(j-1)\lsl\mathcal F_{i,j}
\end{align*}
which we rewrite 
\be
\label{eqfnoneoge}
 \lambda^{2(j-1)} \frac d{d{s}}\left\{\frac{{\cal F_{i,j}}}{\lambda^{2(j-1)}}\right\} =  f^{(i)}_{1}+f^{(i,j)}_2+f^{(i,j)}_3+ f^{(i)}_4,
 \ee
  where 
{\allowdisplaybreaks
\begin{align*}
& f^{(i)}_{1} =
      2 \int \left( \varepsilon_{s} - {\frac{{\lambda}_s}{{\lambda}}} {\Lambda} \varepsilon\right) \left( - (\psi_B\e_y)_{y} + \varepsilon \varphi_{i,B} -\psi_B\left[(\varepsilon +Q_b )^5 - Q_b ^5\right] \right) ,\\
      & f^{(i,j)}_{2} =2 \int \left( \varepsilon_{s} - {\frac{{\lambda}_s}{{\lambda}}} {\Lambda} \varepsilon\right)\e\mathcal J_{i,j}\varphi_{i,B},\\
      &
f^{(i,j)}_3= 
 2 {\frac{{\lambda}_s}{{\lambda}}}  \int    {\Lambda} \varepsilon  \left( - (\psi_B\varepsilon_y)_y + (1+\mathcal J_{i,j})\varepsilon\varphi_{i,B} -\psi_B\left[(\varepsilon +Q_b )^5 - Q_b ^5\right] \right)\\
 & +(\mathcal J_{i,j})_s\int\varphi_{i,B}\e^2- 2(j-1)\lsl\mathcal F_{i,j} \\
&f^{(i)}_4=
  - 2  \int \psi_B(Q_b)_{s} \left( (\varepsilon +Q_b)^5 - Q_b ^5 - 5 \varepsilon Q_b^4\right).
\end{align*}
}

We claim the following estimates on the above terms: for some $\mu_0>0$,  
\begin{align}
\label{re1}
& \frac{d }{ds} f^{(i)}_{1}\leq - \mu_0  \int   \left(\varepsilon_y^2+\e^2\right)   \varphi'_{i,B}  + C |b|^{4},
\\  
\label{re2}
& \left|\frac{d }{ds} f^{(i)}_{k}\right|\leq \frac {\mu_0}{10}    \int   \left(\varepsilon_y^2+\e^2\right)   \varphi'_{i,B}  + C |b|^{4}, \quad \hbox{for $k=2,3,4.$}
\end{align}
Note that in \eqref{re1}, we obtained a negative term $- \mu  \int   \left(\varepsilon_y^2+\e^2\right)   \varphi'_{i,B}$, related both to the smoothing effect of the (gKdV) equation and to a Virial estimate for the linearization of the (gKdV) equation close to the soliton. Inserting \eqref{re1} and \eqref{re2} into \fref{eqfnoneoge} indeed yields \fref{lyapounovconrol}, \fref{lyapounovconrolbis}.\\
In steps 3 - step 6, we prove \eqref{re1} and \eqref{re2}. Observe that the  definitions of $\varphi_i$ and $\psi$ imply the following estimates:
\begin{align} \label{defphi4}
    \forall y\in \RR, \quad  &|\varphi_i'''(y) | + | \varphi_i''(y) |+|\psi'''(y)|+|y\psi'(y)|+|\psi(y)|
   \lesssim  \varphi'_i(y)\lesssim \varphi_i(y),\\
   \forall y\in (-\infty,2], \quad 
  \label{defphi4bis}
  &    e^{ |y|} \psi(y)  + e^{  |y|} \psi'(y) +\varphi_i(y)\lesssim \varphi_i'(y),\\
  \forall y\in \RR, \quad  &
  \varphi_2'(y)\lesssim \varphi_1(y) \lesssim \varphi'_2(y). \label{defphi4tri}
\end{align}
In particular,
\be\label{locpasloc}
{\mathcal{N}_{1,\rm loc}} (s) \lesssim{\mathcal{N}_{2,\rm loc}} (s) \lesssim  {\mathcal{N}_{1}} (s)\lesssim {\mathcal{N}_2} (s),\quad
\int\varepsilon^2 (s,y)     \varphi_{1,B} (y) dy\lesssim {\mathcal{N}_{2,\rm loc}} (s) .
\ee

 {\bf step 3} Control of $f_1^{(i)}$. Proof of \eqref{re1}. 
 We compute 
$ f_{1}^{(i)} $
using the $\varepsilon$ equation \fref{eqofeps} in the following form:
\bea
\label{eqebis}
 \varepsilon_{s} -  {\frac{{\lambda}_s}{{\lambda}}} {\Lambda} \varepsilon&  = &  \left(-\varepsilon_{yy} + \varepsilon - (\varepsilon +Q_b )^5 + Q_b ^5\right)_y      \\
  \nonumber  & + & \left(\frac {{\lambda}_{s}}{{\lambda}}+{b}\right) {\Lambda} Q_b
+ \left(\frac { x_{{s}}}{\lambda} -1\right) (Q_b  + \varepsilon)_y  + \Phi_{{b}} + \Psi_{{b}},
\eea
where
$\Phi_b = - {b}_{s} \left(\chi_b   + \gamma   y (\chi_b)_y\right) P$
and
$-\Psi_b=\left(Q_b''- Q_b+ Q_b^5\right)'+b {\Lambda} Q_b$. This yields:
\begin{align*}
  \rm f_{1}^{(i)}& = 2 \int   \left(  -\varepsilon_{yy} {+} \varepsilon  {-}\left((\varepsilon +Q_b )^5 {-} Q_b ^5\right)  \right)_y \left(-(\psi_B\e_y)_y{+}\e\varphi_{i,B}{-}\psi_B[(Q_b+\e)^5-Q_b^5]\right) 
\\ & +2\left(\frac {{\lambda}_{s}}{{\lambda}}+{b}\right) \int {\Lambda} Q_b  \left( - (\psi_B\varepsilon_y)_y + \varepsilon \varphi_{i,B}-\psi_B\left((\varepsilon +Q_b )^5 - Q_b ^5\right)   \right)
\\ &   + 2 \left(\frac { x_{{s}}}{\lambda} -1\right) \int (Q_b  + \varepsilon)_y
 \left( -(\psi_B \varepsilon_y)_y + \varepsilon \varphi_{i,B}-\psi_B\left((\varepsilon +Q_b )^5 - Q_b ^5\right)   \right)
 \\ & + 2 \int  \Phi_{{b}}  \left( -(\psi_B \varepsilon_y)_y + \varepsilon \varphi_{i,B}-\psi_B\left((\varepsilon +Q_b )^5 - Q_b ^5\right)   \right)
\\ & + 2 \int  \Psi_{{b}}   \left( -(\psi_B \varepsilon_y)_y + \varepsilon \varphi_{i,B}-\psi_B\left((\varepsilon +Q_b )^5 - Q_b ^5\right)   \right) \\
& = {{\rm f}^{(i)}_{1,1}}+ {{\rm f}^{(i)}_{1,2}}+{{\rm f}^{(i)}_{1,3}}+{{\rm f}^{(i)}_{1,4}}+{{\rm f}^{(i)}_{1,5}}.
\end{align*}
\underline{{\em Term $f^{(i)}_{1,1}$}}: This term contains the leading order negative quadratic terms   thanks to our choice of orthogonality conditions and suitable repulsivity properties of the virial quadratic form\footnote{see Lemma \ref{cl:9}} on the soliton core, and intrinsic monotoninicity properties of the renormalized (KdV) flow in the moving frame at speed 1 which expulses energy to the left and leads to positive terms induced by localization of both mass and energy. 

 Let us first integrate by parts in order to obtain a more manageable formula:
\bee
f^{(i)}_{1,1}& = &  2 \int\left[  -\varepsilon_{yy} + \varepsilon  -\left((\varepsilon +Q_b )^5 - Q_b ^5\right)  \right]_y\left[-\e_{yy}+ \varepsilon  -\left((\varepsilon +Q_b )^5 - Q_b ^5\right)\right] \psi_B \\
& + & 2\int\left[  -\varepsilon_{yy} + \varepsilon  -\left((\varepsilon +Q_b )^5 - Q_b ^5\right)\right]_y\left(-\psi_B'\e_y+\e(\varphi_{i,B}-\psi_B)\right). 
\eee
We compute the various terms separately:
\bee
&& 2\int\left[  -\varepsilon_{yy} + \varepsilon  -\left((\varepsilon +Q_b )^5 - Q_b ^5\right)  \right]_y\psi_B\left[-\e_{yy}+ \varepsilon  -\left((\varepsilon +Q_b )^5 - Q_b ^5\right)\right]\\
&  =&-\int\psi_B'\left[-\e_{yy}+ \varepsilon  -\left((\varepsilon +Q_b )^5 - Q_b ^5\right)\right]^2\\
&  =&-\int\psi_B'\left[-\e_{yy}+ \varepsilon \right]^2\\ &&- \int\psi_B'\left\{\left[-\e_{yy}+ \varepsilon  -\left((\varepsilon +Q_b )^5 - Q_b ^5\right)\right]^2-\left[-\e_{yy}+\e\right]^2\right\}\\
&   =& -\left[\int\psi_B'(\e_{yy}^2+2\e_y^2)+\int\e^2(\psi_B'-\psi_B''')\right]\\
&  & -\int\psi_B'\left\{\left[-\e_{yy}+ \varepsilon  -\left((\varepsilon +Q_b )^5 - Q_b ^5\right)\right]^2-\left[-\e_{yy}+\e\right]^2\right\}.
\eee
Next after integration by parts:
\bee
& & 2 \int\left[  -\varepsilon_{yy} + \varepsilon\right]_y\left[-\psi_B'\e_y+\e(\varphi_{i,B}-\psi_B)\right]\\
&  &  =-2 \Big\{\int\psi_B'\e_{yy}^2+\int\e_y^2(\frac32\varphi_{i,B}'-\frac12\psi'_B-\frac12\psi_B''')\\&&+\int\e^2(\frac12(\varphi_{i,B}-\psi_B)'-\frac12(\varphi_{i,B}-\psi_B)''')\Big\},
\eee
similarily:
\bee
&-& 2\int\left[(Q_b+\e)^5-Q_b^5\right]_y(\varphi_{i,B}-\psi_B)\e\\
& = & -\frac13\int(\varphi_{i,B}-\psi_B)'\left\{[(Q_b+\e)^6-Q_b^6-6Q_b^5\e-6[(\e+Q_b)^5-Q_b^5]\e\right\}\\
& - & 2\int (\varphi_{i,B}-\psi_B)(Q_b)_y[(Q_b+\e)^5-Q_b^5-5Q_b^4\e],
\eee
and by  direct expansion:
$$
\int[(Q_b+\e)^5-Q_b^5]_y\psi_B'\e_y=5\int\psi_B'\e_y\left\{(Q_b)_y[(Q_b+\e)^4-Q_b^4]+(Q_b+\e)^4\e_y\right\}.
$$
We collect the above computations and obtain the following
\bee
f^{(i)}_{1,1} &=& -\int\left[3\psi_B'\e_{yy}^2+(3\varphi_{i,B}'+\psi_B'-\psi_B''')\e_y^2+(\varphi_{i,B}'-\varphi_{i,B}''')\e^2\right]\\
\nonumber & - &2\int \left[\frac {(\varepsilon + Q_b )^6} 6-\frac {Q_b ^6} 6 - Q_b ^5 \varepsilon -\left((\varepsilon +Q_b )^5 - Q_b ^5\right) \varepsilon \right](\varphi_{i,B}'-\psi_B')\\
\nonumber & + & 2  \int  \left[(\varepsilon+Q_b )^5 -Q_b ^5- 5 Q_b ^4 \varepsilon\right](Q_b )_y (\psi_B-\varphi_{i,B})\\
& + & 10\int\psi_B'\e_y\left\{(Q_b)_y[(Q_b+\e)^4-Q_b^4]+(Q_b+\e)^4\e_y\right\}\\
& - & \int\psi_B'\left\{\left[-\e_{yy}+ \varepsilon  -\left((\varepsilon +Q_b )^5 - Q_b ^5\right)\right]^2-\left[-\e_{yy}+\e\right]^2\right\}\\
& = & (f^{(i)}_{1,1})^<+(f^{(i)}_{1,1})^{\sim}+(f^{(i)}_{1,1})^>
\eee
where $(f^{(i)}_{1,1})^{<,\sim,>}$ respectively corresponds to integration on $y<-\frac B2$, $|y|\leq \frac B2$, $y>\frac B2$.

For the region $y<-B/2$, we rely on monotonicity type arguments and estimate using \eqref{defphi4}:
$$
\int_{y<-B/2} \varepsilon^2 |\varphi'''_{i,B}|\lesssim
  \frac 1 {B^2}  \int_{y<-B/2} \varepsilon^2 \varphi'_{i,B} \leq  \frac 1{100} \int_{y<-B/2} \varepsilon^2 \varphi'_{i,B},
$$
$$\int_{y<-B/2} \varepsilon_y^2|\psi'''_B|\lesssim
  \frac 1 {B^2}  \int_{y<-B/2} \varepsilon_y^2 \varphi'_{i,B} \leq  \frac 1{100} \int_{y<-B/2} \varepsilon_y^2 \varphi'_{i,B},$$
by choosing $B$ large enough. Next, we recall the Sobolev bound\footnote{see the proof of Lemma 6 in \cite{Mjams}}: $\forall B\geq 1$,
\bea
\label{nonlinearsobolev}
\nonumber   \|\e^2\sqrt{\varphi_{i,B}'}\|^2_{L^{\infty}(y<-\frac B2)}& \lesssim &  \|\e\|^2_{L^2}\left(\int_{y<-\frac B2} \e_y^2\varphi_{i,B}'+\int_{y<-\frac B2} \e^2\frac{(\varphi''_{i,B})^2}{\varphi'_{i,B}}\right)\\
& \lesssim &    {\delta(\kappa^*)} \int_{y<-B/2} (\e_y^2+\varepsilon^2)  \varphi'_{i,B} .
\eea
\begin{remark} This estimate is linked to the $L^2$ critical nature of the problem and the smallness relies only on the global $L^2$ smallness \fref{boundnwe} only, and requires no smallness of derivatives. It is the key to control the pure $\e^6$ non linear term in the functionals $\mathcal F_{i,j}$.
\end{remark}

The homogeneity of the power nonlinearity then ensures (for $B$ large and $\kappa^*$ small):
\bee
& & \left| \int_{y<-B/2} \left[\frac {(\varepsilon + Q_b )^6} 6
-\frac {Q_b ^6} 6 - Q_b ^5 \varepsilon - 
 \left((\varepsilon +Q_b )^5 - Q_b ^5\right) \varepsilon \right](\varphi_{i,B}'-\psi_B') \right|
\\ 
&&   \lesssim    
 \int_{y<-B/2} \left(\varepsilon^6  + |Q_b| ^4 \varepsilon^2\right)\varphi'_{i,B} \lesssim  \left(\delta(\kappa^*)+e^{-\frac B{10} }\right)\int_{y<-B/2}  \varphi_{i,B}'(\e^2+\e_y^2)\\
& &\leq    \frac 1{100} \int_{y<-B/2} (\e_y^2+\varepsilon^2)  \varphi'_{i,B}
\eee
  and similarily for $\kappa^*$ small depending on $B$,
\bee
&& \left| \int_{y<-\frac{B}{2}} \left[(\varepsilon+Q_b )^5 -Q_b ^5- 5 Q_b ^4 \varepsilon\right](Q_b )_y (\psi_B-\varphi_{i,B})\right|\\
&& \lesssim   B \int_{y<-\frac B2}(\e^2|Q_b|^3+|\e|^5)(|Q_y|+|b||(P\chi_b)'|)\varphi_{i,B}'\\&&\leq \frac 1{100} \int_{y<-B/2}  (\e_y^2+\varepsilon^2)  \varphi'_{i,B}+b^4.
\eee
We further estimate using \fref{nonlinearsobolev} and $(\varphi'_i)^2\lesssim \psi'\lesssim (\varphi'_i)^2$ for $y<-\frac 12$:
\bee
&&\left|\int_{y<-\frac B2}\psi_B'\e_y\left\{(Q_b)_y[(Q_b+\e)^4-Q_b^4]+(Q_b+\e)^4\e_y\right\}\right|\\
&& \lesssim   e^{- \frac 12 B} \int_{y<-\frac B2}\varphi'_{i,B}(\e_y^2+\e^2) +\int \psi_B'|\e|^4|\e_y|^2\\
&& \leq \frac 1{100}  \int  \e_{yy}^2  \psi_B'
+ \frac 1{100}  \int_{y<-B/2}  (\e_y^2+\varepsilon^2)  \varphi'_{i,B}.
\eee

Note that for the term $\int \psi_B'|\e|^4|\e_y|^2$, we have proceeded as follows:
\begin{align*}
\int  \psi_B' \e_y^2 \e^4 &\lesssim
\| \e^2 (\psi_B')^{\frac 14}\|_{L^\infty}^2  \int \e_y^2 (\psi_B')^{\frac 12}\\
& \lesssim \|\e\|_{L^2}^2  \left(\int (\e_y^2+\e^2) (\psi_B')^{\frac 12}\right)  \int \e_y^2 (\psi_B')^{\frac 12}\\
& \lesssim  \delta(\alpha^*)\left(\int \e_y^2 (\psi_B')^{\frac 12}\right)^2 + \delta(\alpha^*) \int  \e_y^2  \varphi'_{i,B} 
\end{align*}
and
\begin{align*}
\left(\int \e_y^2 (\psi_B')^{\frac 12}\right)^2 
&= \left( - \int \e \e_{yy} (\psi_B')^{\frac 12} +\frac 12  \int \e^2 ((\psi_B')^{\frac 12})'' \right)^2 \\
&\lesssim \left( \int \e^2\right)  \int  \left(\e_{yy}^2+\e^2\right) \psi_B'.
\end{align*}
Thus,
$$
\int  \psi_B' \e_y^2 \e^4 \lesssim \delta(\alpha^*) \int \left (\e_{yy}^2+\e^2\right) \psi_B'
+ \delta(\alpha^*) \int  \e_y^2  \varphi'_{i,B} .
$$

The remaining nonlinear term is estimated using the local $H^2$ control provided by localization:
\bee
& & \left|\int_{y<-\frac B2}\psi_B'\left\{\left[-\e_{yy}+ \varepsilon  -\left((\varepsilon +Q_b )^5 - Q_b ^5\right)\right]^2-\left[-\e_{yy}+\e\right]^2\right\}\right|
\\
&=&   \left|\int_{y<-\frac B2}\psi_B' \left(-2\e_{yy}+ 2\varepsilon  -\left((\varepsilon +Q_b )^5 - Q_b ^5\right)\right)  \left(  (\varepsilon +Q_b )^5 - Q_b ^5\right)  \right|\\
& \lesssim & \frac{1}{100}\int_{y<-\frac B2}\psi_{B}'(|\e_{yy}|^2+|\e|^2)+100\int_{y<-\frac B2}(\varphi'_{i,B})^2(|\e||Q_b|^4+|\e|^5)^2\\
& \lesssim &  \frac 1{100} \int_{y<-B/2} \left[ \e_{yy}^2\psi_B'+(\e_y^2+\varepsilon^2)  \varphi'_{i,B}\right].
\eee

In the region $y>\frac B2$, $\psi_B(y)=1$. We rely on \eqref{defphi4} to estimate:
$$\int_{y>B/2} \varepsilon^2 |\varphi'''_{i,B}|\lesssim
  \frac 1 {B^2}  \int_{y>B/2} \varepsilon^2 \varphi'_{i,B} \leq  \frac 1{100} \int_{y>B/2} \varepsilon^2 \varphi'_{i,B},$$ and we use the exponential localization of $Q_b$ to the right and the Sobolev bound $$\|\e\|_{L^{\infty}(y>0)}\lesssim \|\e\|_{H^1(y>0)}\lesssim  \mathcal{N}_{2}^{\frac 12}\lesssim
  \delta(\kappa^*)$$
   to control:  \bee
& & \left| \int_{y>B/2} \left(\frac {(\varepsilon + Q_b )^6} 6
-\frac {Q_b ^6} 6 - Q_b ^5 \varepsilon - 
 \left((\varepsilon +Q_b )^5 - Q_b ^5\right) \varepsilon \right)\varphi_{i,B}' \right|
\\ 
&   \lesssim &  
 \int_{y>B/2} \left(\varepsilon^6  + |Q_b| ^4 \varepsilon^2\right)\varphi_{i,B}' \lesssim  (\delta(\kappa^*)+e^{-\frac B{10}})\int_{y>B/2} \varphi_{i,B}'(\e^2+\e_y^2)\\
& \leq &  \frac 1{100} \int_{y>B/2} (\e_y^2+\varepsilon^2)  \varphi'_{i,B},
\eee
\bee
&& \left| \int _{y>B/2} \left[(\varepsilon+Q_b )^5 -Q_b ^5- 5 Q_b ^4 \varepsilon\right](Q_b )_y (\psi_B-\varphi_{i,B})\right|\\
& \lesssim & \int_{y>B/2}(\e^2|Q_b|^3+|\e|^5)(|Q_y|+|b|e^{-|y|})\leq \frac 1{100} \int_{y>B/2}  (\e_y^2+\varepsilon^2)  \varphi'_{i,B}.
\eee

In the   region $|y|<B/2$, $\varphi_{i,B}(s,y) = 1+ y/B$ and $\psi_B(y)=1$. In particular, $\varphi_{i,B}'''=\psi_B'=0$ in this region, and we obtain:
$$
\begin{array}{lll}
(f^{(i)}_{1,1})^{\sim}  
& \displaystyle  = - \frac {1}{B} \int_{|y|<B/2} &\bigg\{ 3 \varepsilon_y^2 + \varepsilon^2  \\
  & \quad & \displaystyle + 2 \left(\frac {(\varepsilon + Q_b )^6} 6
-\frac {Q_b ^6} 6 - Q_b ^5 \varepsilon -
 \left((\varepsilon +Q_b )^5 - Q_b ^5\right) \varepsilon \right) 
  \\
  &\quad & \displaystyle + 2\left((\varepsilon+Q_b )^5 -Q_b ^5- 5 Q_b ^4 \varepsilon\right) y (Q_b )_y  
\bigg\}\\
& \displaystyle  = - \frac {1}   { B} \int_{|y|<B/2}  &\bigg\{ 3 \varepsilon_y^2 + \varepsilon^2 - 5 Q^4 \varepsilon^2 + 20 y Q' Q^3 \varepsilon^2\bigg\} + R_{\rm Vir}(\varepsilon),
\end{array}
$$
where
\begin{align*}
R_{\rm Vir}(\varepsilon) = - \frac {1}   { B} & \int_{|y|<B/2} \bigg\{   -5 (Q_b^4-Q^4) \varepsilon^2 
+ 20 y ( (Q_b)_y Q_b^3 - Q'Q^3) \varepsilon^2     \\
& -\frac {40} 3 Q_b^3 \varepsilon^3 - 15 Q_b^2 \varepsilon^4 - 8 Q_b \varepsilon^5 - \frac 53 \varepsilon^6 \\
&  + 20 y (Q_b)_y Q_b^2 \varepsilon^3  + 10 y (Q_b)_y Q_b\varepsilon^4
+ 2 y(Q_b)_y \varepsilon^5 \bigg\}.
\end{align*}
We now claim the following coercivity result which is the main tool to measure dispersion (related to the Viriel estimate, see Section \ref{s:A:2}).

\begin{lemma}[Localized viriel estimate]\label{cl:9}
There exists $B_{0}>100$ and $\mu_{3}>0$ such that
if $B\geq B_{0}$, then
$$
\int_{|y|<B/2}   \left( 3 \varepsilon_y^2 + \varepsilon^2 - 5 Q^4 \varepsilon^2 + 20 y Q' Q^3 \varepsilon^2\right) \geq \mu_{3} \int_{|y|<B/2}  \left(  \varepsilon_y^2 + \varepsilon^2 \right)
-  \frac 1 B \int \e^2 e^{-\frac{|y|}{2}}.
$$
\end{lemma}
We further estimate by Sobolev's inequality,
$$
|R_{\rm Vir}(\varepsilon)|\lesssim \frac 1B (|b| + \|\e\|_{L^\infty(|y|< B/2)}) \int_{|y|<B/2} (\e_y^2+\varepsilon^2)\lesssim  \frac 1B \delta(\kappa^*)\int_{|y|<B/2} (\e_y^2+\varepsilon^2),
$$
and thus for $\kappa^*$ small enough:
$$(f^{(i)}_{1,1})^\sim   \leq  - \frac { \mu_{3}}   {2 B} \int_{|y|<B/2}  \left(  \varepsilon_y^2 + \varepsilon^2 \right) + \frac 1 {B^2} \int \e^2 e^{-\frac{|y|}{2}}.$$ 
The collection of above estimates yields the bound:
\begin{equation}\label{eq:cf11}
   f^{(i)}_{1,1}  \leq  - \frac  {\mu_4} B    \int\left[\psi'_B\e_{yy}^2+\varphi'_{i,B}(\e_y^2+\e^2)\right] + Cb^4
    \end{equation}
    for some universal $\mu_4>0$ independent of $B$.\\
  \underline{{\em Term $f^{(i)}_{1,2}$}}:  We integrate by parts to express $f_{1,2}$:
\begin{align*}
f^{(i)}_{1,2} & =  2 \left(\frac {{\lambda}_{s}}{{\lambda}}+{b}\right)  \int {\Lambda} Q  (L \varepsilon) -2\left(\frac {{\lambda}_{s}}{{\lambda}}+{b}\right)\int\varepsilon (1-\varphi_{i,B})\Lambda Q
\\
&+ 2b \left(\frac {{\lambda}_{s}}{{\lambda}}+{b}\right)  \int {\Lambda} ( \chi_b P)  \left(-(\psi_B\e_y)_{y}  + \e\varphi_{i,B} - \psi_B[(Q_b+\e)^5-Q_b^5)]\right)\\
& + 2 \left(\frac {{\lambda}_{s}}{{\lambda}}+{b}\right)  \int {\Lambda} Q      \left(-(\psi_B)_y\e_y
-(1-\psi_B)\e_{yy}+(1-\psi_B)[(Q_b+\e)^5-Q_b^5]\right) \\&
+ 2\left(\frac {{\lambda}_{s}}{{\lambda}}+{b}\right)  \int {\Lambda} Q  [(Q_b+\e)^5-Q_b^5-5Q^4\e]
\end{align*}
Observe from \fref{ortho1}: $$\int \Lambda Q (L\e)=(\e,L\Lambda Q)=-2(\e,Q)=0.$$  We now use the orthogonality conditions $( \varepsilon, y {\Lambda} Q)=0$ and the definition of $\varphi_{i,B}$ to estimate:
$$
	\left| \int {\Lambda} Q \varepsilon (1-\varphi_{i,B}) \right| = \left| \int {\Lambda} Q \varepsilon \left(1-\varphi_{i,B}+ \frac yB\right)\right|
	\lesssim e^{- \frac B 8 } \mathcal{N}_{i,\rm loc}^{\frac 12}  ,
$$
so that by \eqref{eq:2002} and for $B$ large enough:
\begin{align*}
  \left|  \left(\frac {{\lambda}_{s}}{{\lambda}}+{b}\right)  \int {\Lambda} Q\varepsilon (1-\varphi_{i,B}) \right| 
&\lesssim \left( \mathcal{N}_{i,\rm loc}^{\frac 12}  + b^2 \right) e^{- \frac B 8}  \mathcal{N}_{i,\rm loc}^{\frac 12}  
\\ &
\leq 
 \frac 1 {500}  \frac {\mu_4} {B} \mathcal{N}_{i,\rm loc} 
+C b^4.
\end{align*}
For the next term in $f_{1,2}^{(i)}$, we first integrate by parts to remove all derivatives on $\e$.
Then, by  \eqref{eq:2002}, the weighted Sobolev bound \fref{nonlinearsobolev} and the properties of $\varphi_{i,B}$, $\psi_B$, $P$  and $\chi_b$ \eqref{eq:210}, we obtain
for $\kappa^*$ small,
 \bee
&&\left| 2b \left(\frac {{\lambda}_{s}}{{\lambda}}+{b}\right)  \int {\Lambda} ( \chi_b P)  \left(-(\psi_B\e_y)_{y}  + \e\varphi_{i,B} - \psi_B[(Q_b+\e)^5-Q_b^5)]\right)\right|\\
& \lesssim & |b| \left( \mathcal{N}_{i,\rm loc}^{\frac 12}  + b^2 \right) \left(\int_{y<0} e^{\frac {y}B}+1\right)^{\frac 12} \mathcal{N}_{i, \rm loc}^{\frac 12} \\
& \lesssim & |b| \left( \mathcal{N}_{i,\rm loc}^{\frac 12}  + b^2 \right)B^{\frac 12}\mathcal{N}_{i, \rm loc}^{\frac 12} \leq \frac 1 {500}  \frac {\mu_4} {B } \mathcal{N}_{i,\rm loc} (s)
+Cb^4.
\eee
Next, integrating by parts,   using   the exponential decay of $Q$ and since $\psi_B(y)\equiv 1$ on $[-\frac B2,\infty)$:
\bee
&&\left|  \left(\frac {{\lambda}_{s}}{{\lambda}}+{b}\right)  \int {\Lambda} Q      \left(-(\psi_B)_y\e_y
-(1-\psi_B)\e_{yy}+(1-\psi_B)[(Q_b+\e)^5-Q_b^5]\right) \right|\\
& \lesssim & \left( \mathcal{N}_{i,\rm loc}^{\frac 12}  + b^2 \right) (e^{-\frac B{10} }+\delta(\kappa^*))\mathcal{N}_{i,\rm loc}^{\frac 12} \leq \frac 1 {500}  \frac {\mu_4} {B} \mathcal{N}_{i,\rm loc} ,
\eee
and finally:
\bee
&& \left| \left(\frac {{\lambda}_{s}}{{\lambda}}+{b}\right)  \int {\Lambda} Q  \left[(Q_b+\e)^5-Q_b^5-5Q^4\e\right]
\right|\\
& \lesssim & \left( \mathcal{N}_{i,\rm loc}^{\frac 12}  + b^2 \right) \delta(\kappa^*)\mathcal{N}_{i,\rm loc}^{\frac 12}  \leq \frac 1 {500}  \frac {\mu_4} {B} \mathcal{N}_{i,\rm loc}.
\eee
The collection of above estimates yields the bound:
 $$
 |f^{(i)}_{1,2}|\leq \frac 1 {100}  \frac {\mu_4} {B} \mathcal{N}_{i,\rm loc} + Cb^4.
 $$
\underline{{\em Term $f^{(i)}_{1,3}$}}: We use the identity
\bee
&& \int \psi_B(Q_b )_y \left((\varepsilon+Q_b )^5 - Q_b ^5 - 5 Q_b ^4 \varepsilon \right) + \int \psi_B\varepsilon_y \left( (\varepsilon + Q_b )^5 - Q_b ^5 \right) \nonumber \\
& = &\frac 16 \int \psi_B\partial_y\left[  {(Q_b +\varepsilon)^6}  -   {Q_b ^6}   - 6 Q_b ^5 \varepsilon\right] =-\frac16\int\psi_B'\left[  {(Q_b +\varepsilon)^6}  -   {Q_b ^6}    - 6 Q_b ^5 \varepsilon\right] 
\eee
to compute:
\bee
f^{(i)}_{1,3}& = & 2\left(\xsl-1\right)\int\frac 16\psi_B'\left[(Q_b+\e)^6-Q_b^6-6Q_b^5\e\right]\\
& + & 2\left(\xsl-1\right)\int(b\chi_bP+\e)_y\left[-\psi_B'\e_y-\psi_B\e_{yy}+\e\varphi_{i,B}\right]\\
& + &  2\left(\xsl-1\right) \int Q'\left[L\e-\psi_B' \e_y + (1-\psi_B)\e_{yy}-\e(1-\varphi_{i,B})\right]\\
& + &  10 \left(\xsl-1\right)  \int\e\psi_B(Q_b^4(Q_b)_y-Q^4Q_y) .
\eee
Since $|(Q_b+\e)^6-Q_b^6-6Q_b^5\e|\lesssim |\e|^2+  |\e|^6$, by \eqref{nonlinearsobolev} and
$|\xsl-1|\leq \delta(\kappa^*)$, we have
\bee
&& \left|2 \left(\xsl-1\right)\int\frac 16\psi_B'\left[(Q_b+\e)^6-Q_b^6-6Q_b^5\e\right]\right|
\\&& \lesssim\delta(\kappa^*) \int \psi_B'(|\e|^2 + |\e|^6) \leq \frac1{500} \frac  {\mu_4}{B}\int   \left(\varepsilon_y^2+\e^2\right)   \varphi'_{i,B}.
\eee
Then, as before, integrating by parts, and using Cauchy-Schwarz inequality,
\bee
&&\left|2b \left(\xsl-1\right)\int (\chi_bP)_y\left[-\psi_B'\e_y-\psi_B\e_{yy}+\e\varphi_{i,B}\right]\right|\\\
&&\lesssim
|b| \left(\mathcal N_{i,\rm loc}^{\frac 12} + b^2\right)
B^{\frac 12} \mathcal N_{i,\rm loc}\leq \frac 1 {500}  \frac {\mu_4} {B} \mathcal{N}_{i,\rm loc}+b^4.
\eee
\bee
&&\left|2\left(\xsl-1\right)\int  \e_y\left[-\psi_B'\e_y-\psi_B\e_{yy}+\e\varphi_{i,B}\right]\right|
\\ &&\lesssim \delta(\kappa^*) \int   \left(\varepsilon_y^2+\e^2\right)   \varphi'_{i,B}\leq \frac1{500} \frac  {\mu_4}{B}\int   \left(\varepsilon_y^2+\e^2\right)   \varphi'_{i,B}.
\eee
The next term is treated using the cancellation $L Q'=0$ and the orthogonality conditions $(\varepsilon, \Lambda Q)=(\e,Q)=0$, so that $(yQ',\e)=0$. Thus, by the definitions of $\varphi_{i,B}$ and $\psi_B$,
\bee
&&\left|2\left(\xsl-1\right) \int Q'\left[L\e-\psi_B' \e_y + (1-\psi_B)\e_{yy}-\e(1-\varphi_{i,B})\right]\right| 
\\
&&=\left|2\left(\xsl-1\right) \int Q'\left[-\psi_B' \e_y + (1-\psi_B)\e_{yy}-\e\left(1+\frac yB-\varphi_{i,B}\right)\right]\right| 
\\&& \lesssim   \left(\mathcal{N}_{i,\rm loc}^{\frac 12}+ b^2\right)e^{-\frac B {10}}\mathcal{N}_{i,\rm loc}^{\frac 12} \leq   \frac 1 {500}  \frac {\mu_4} {B} \mathcal{N}_{i,\rm loc}+b^4.
\eee
Finally,
\bee
&& \left|10 \left(\xsl-1\right)  \int\e\psi_B(Q_b^4(Q_b)_y-Q^4Q_y) \right| \\
&&\lesssim |b|  \left(\mathcal{N}_{i,\rm loc}^{\frac 12}+ b^2\right) B^{\frac 12} \mathcal{N}_{i,\rm loc}^{\frac 12} \leq   \frac 1 {500}  \frac {\mu_4} {B} \mathcal{N}_{i,\rm loc}+Cb^4.
\eee
In conclusion for $f^{(i)}_{1,3}$,
$$
|f^{(i)}_{1,3}| \leq \frac1{100} \frac  {\mu_4}{B}\int   \left(\varepsilon_y^2+\e^2\right)   \varphi'_{i,B}+Cb^4,
$$
for $B$ large enough and $\kappa^*$ small enough.\\
\underline{{\em Term $f^{(i)}_{1,4}$}}:  We compute explicitely:
$$
f^{(i)}_{1,4} = - 2{b}_{s}  \int   \left(\chi_b   + \gamma   y (\chi_b)_y\right) P  \left(   - \psi_B\varepsilon_{yy} -\psi_B'\e_y+ \varepsilon \varphi_{i,B}-\psi_B\left((\varepsilon +Q_b )^5 - Q_b ^5\right)   \right).$$
 We estimate after   integrations by parts
 \bee
 \left| \int   \left(\chi_b   + \gamma   y (\chi_b)_y\right) P  \left(   - \psi_B\varepsilon_{y } \right)_y\right|&\lesssim& \int|\e|\left|(\psi_B (  (\chi_b   + \gamma   y (\chi_b)_y) P)_y)_y\right| \\
 & \lesssim & B^{\frac 12}\mathcal{N}_{i,\rm loc}^{\frac 12}.
 \eee
 $$\left|\int   \left(\chi_b   + \gamma   y (\chi_b)_y\right) P \varepsilon \varphi_{i,B}\right|\lesssim B^{\frac 12}\mathcal{N}_{i,\rm loc}^{\frac 12}.$$ The estimate of the nonlinear term follows from the weighted Sobolev estimate \fref{nonlinearsobolev} with $\psi\leq (\varphi'_i)^2$ for $y<-\frac 12$:
 \bee
 && \left|\int   \left(\chi_b   + \gamma   y (\chi_b)_y\right) P \psi_B\left[(\varepsilon +Q_b )^5 - Q_b ^5\right]\right|  \lesssim  \int\psi_B(|Q_b|^4|\e|+|\e|^5)\\
 && \lesssim    B^{\frac 12}\left( \int (|\e|^2+ |\e|^6) \psi_B\right)^{\frac 12} \lesssim B^{\frac 12}
\left(\int   \left(\varepsilon_y^2+\e^2\right)   \varphi'_{i,B} \right)^{\frac 12} .
 \eee 
Together with \eqref{eq:2003}, these estimates yield the bound:
$$|f_{1,4}| 
\leq  \frac 1 {500}  \frac {\mu_4} {B } 
 \int   \left(\varepsilon_y^2+\e^2\right)   \varphi'_{i,B}   +C|b|^4.$$
\underline{{\em Term $f^{(i)}_{1,5}$}}: This term generates the leading order term in $b$ through the error term $\Psi_b$ in the construction of the approximate $Q_b$ profile. Recall:
$$
f^{(i)}_{1,5}=    
 2  \int  \Psi_b    \left(   -(\psi_B\varepsilon_{y})_y  + \varepsilon \varphi_{i,B}-\psi_B\left((\varepsilon +Q_b )^5 - Q_b ^5\right)   \right).
$$
We now rely on   \eqref{eq:203} to estimate by integration by parts and Cauchy-Schwarz's inequality,
$$
\left|  \int  (\Psi_b)_y    \psi_B\varepsilon_{y}\right| \lesssim    B^{\frac 12}  b^2 \mathcal{N}_{i,\rm loc}^{\frac12} \leq  \frac 1 {500}  \frac {\mu_4} {B} \mathcal{N}_{i,\rm loc} +  C |b|^{4}.
$$
By \eqref{eq:202}, $|\Psi_b|\leq b^2 + |b|^{1+\gamma} \mathbf{1}_{[-2,-1]}(|b|^\gamma y)$ and so
by the exponential decay of $\varphi_{i,B}$ in the left,
$$\left|  \int  \Psi_b \varphi_{i,B}\e\right|\lesssim \left( b^2 B^{\frac 12}+ e^{-\frac 1{2 |b|^{\gamma}}}\right)|b|^{1+\gamma}\mathcal{N}_{i,\rm loc}^{\frac12}\leq  \frac 1 {500}  \frac {\mu_4} {B } \mathcal{N}_{i,\rm loc} +C|b|^{4}.$$ For the nonlinear term, similarly and using \eqref{nonlinearsobolev},
\bee
 \left| \int  \Psi_b   \psi_B\left[(\varepsilon +Q_b )^5 - Q_b ^5\right]   \right|  \leq  \frac 1 {500}  \frac {\mu_4} {B} \int   \left(\varepsilon_y^2+\e^2\right)   \varphi'_{i,B}  +C|b|^{4}.
\eee
The collection of above estimates yields the bound:
$$
| f^{(i)}_{1,5}| \leq \frac 1 {100}  \frac {\mu_4} {B} \int   \left(\varepsilon_y^2+\e^2\right)   \varphi'_{i,B}+|b|^{4}.$$

{\bf step 4} $f_2^{(i,j)}$ term. 

We integrate by parts using \fref{eqebis}: 
\bee
f_{2}^{(i,j)} & = & 2 \mathcal J_{i,j}\int \e\varphi_{i,B}\left[ \left(-\varepsilon_{yy} + \varepsilon - (\varepsilon +Q_b )^5 + Q_b ^5\right)_y \right. \\
  & + &\left. \left(\frac {{\lambda}_{s}}{{\lambda}}+{b}\right) {\Lambda} Q_b
+ \left(\frac { x_{{s}}}{\lambda} -1\right) (Q_b  + \varepsilon)_y  + \Phi_{{b}} + \Psi_{{b}}\right]
\eee
We integrate by parts, estimate all terms like for $f_1^{(i)}$ and use \fref{controlj} which implies $$|\mathcal J_{i,j}|\lesssim \delta(\kappa^*)$$ to conclude: $$|f_2^{(i,j)}|\lesssim \delta(\kappa^*)\left[  \int   \left(\varepsilon_y^2+\e^2\right)   \varphi'_{i,B}   +|b|^{4}\right].$$
 
{\bf step 5} $f^{(i,j)}_3$ term.

Recall:
\bee
f^{(i,j)}_3 &= & 2 {\frac{{\lambda}_s}{{\lambda}}}  \int    {\Lambda} \varepsilon  \left( - (\psi_B\varepsilon_y)_y + (1+\mathcal J_{i,j})\varepsilon\varphi_{i,B} -\psi_B\left[(\varepsilon +Q_b )^5 - Q_b ^5\right] \right)\\
& + & (\mathcal J_{i,j})_s\int\varphi_{i,B}\e^2 -2(j-1)\lsl\mathcal F_{i,j}.
\eee
We integrate by parts to compute:
\begin{align*}
 & \int {\Lambda} \varepsilon (\psi_B\varepsilon_{y})_{y} =  -\int \e_y^2 \psi_B + \frac 12 \int \e^2_y y \psi_B',\\
 & \int  ({\Lambda} \varepsilon )    \varepsilon \varphi_{i,B} =- \frac 12 \int     \varepsilon^2    y\varphi_{i,B}',\\
 & \int {\Lambda} \varepsilon\psi_B\left[ (\varepsilon + Q_b )^5 - Q_b ^5\right]
 = \frac 16 \int (2\psi_B-y\psi_B')\left[(\varepsilon+Q_b )^6 - Q_b ^6 - 6 Q_b ^5 \varepsilon \right]\\ &
-\int \psi_B{\Lambda} Q_b  \left( (\varepsilon+Q_b )^5 - Q_b ^5 - 5 Q_b ^4 \varepsilon\right).
\end{align*}
Thus,\begin{align*}
f^{(i,j)}_3 
& =   \lsl\int[(2-2(j-1))\psi_B-y\psi_B']\e_y^2 \\
&-\frac 13\lsl\int[(2-2(j-1))\psi_B-y\psi_B']\left[(\varepsilon+Q_b )^6 - Q_b ^6 - 6 Q_b ^5 \varepsilon \right]\\
& +2 {\frac{{\lambda}_s}{{\lambda}}} \int \psi_B {\Lambda} Q_b  \left( (\varepsilon +Q_b )^5 - Q_b ^5 - 5 Q_b ^4 \varepsilon\right)
  \\
& + (\mathcal J_{i,j})_s\int\varphi_{i,B}\e^2-\lsl (1+\mathcal J_{i,j})\int y\varphi'_{i,B}\e^2-2(j-1)\lsl(1+\mathcal J_{i,j})\int\varphi_{i,B}\e^2\\
& =  \lsl\int[2(2-j)\psi_B-y\psi_B']\e_y^2 \\ &-\frac 13\lsl\int[2(2-j)\psi_B-y\psi_B']\left[(\varepsilon+Q_b )^6 - Q_b ^6 - 6 Q_b ^5 \varepsilon \right]\\
& +2 {\frac{{\lambda}_s}{{\lambda}}} \int \psi_B {\Lambda} Q_b  \left( (\varepsilon +Q_b )^5 - Q_b ^5 - 5 Q_b ^4 \varepsilon\right) 
  \\
& + \frac{1}{i}\left[(\mathcal J_{i,j})_s-2(j-1)(1+\mathcal J_{i,j})\lsl\right]\int(i\varphi_{i,B}-y\varphi_{i,B}')\e^2\\
& +\frac{1}{i}\left[(\mathcal J_{i,j})_s-(2(j-1)+i)(1+\mathcal J_{i,j})\lsl\right] \int y\varphi'_{i,B}\e^2\\
& =  f^{(i,j)}_{3,1} +f_{3,2}^{(i,j)},
\end{align*}
where
$$f_{3,2}^{(i,j)}=
\frac{1}{i}\left[(\mathcal J_{i,j})_s-(2(j-1)+i)(1+\mathcal J_{i,j})\lsl\right] \int y\varphi'_{i,B}\e^2.
$$
We estimate all terms in the above expression using again the notation $(f_{3,k}^{(i,j)})^{<,\sim,>}$ corresponding to integration on $y<-\frac B 2$, $|y|<\frac B 2$, $y>\frac B2$. The middle term is easily estimated in brute force using \fref{controlj}, \fref{eq:2004}, \fref{eq:2002} and the a priori bound \fref{boundnwe}, we get
$$|(f^{(i,j)}_3)^\sim|\lesssim \delta(\kappa^*) \int   \left(\varepsilon_y^2+\e^2\right)   \varphi'_{i,B} .$$ For $y<-B$, we use the exponential decay of $\psi_B, \varphi_{i,B}$ and   \fref{defphi4}
to estimate:
\bee
&&\int_{y<-\frac B2}(\psi_B+|y|\psi'_B+\varphi_{i,B})(\e_y^2+\e^2) +|y|\varphi'_{i,B} \e^2\\
&& \lesssim  \int_{y<-\frac B2} \e_y^2 \varphi'_{i,B}+\int_{y<-\frac B2}|y|\varphi'_{i,B}\e^2\\
&&\lesssim  \int   \varepsilon_y^2    \varphi'_{i,B}  + \left(\int_{y<-\frac B2}|y|^{100}e^{\frac yB}\e^2\right)^{\frac{1}{100}}\left(\int_{y<-\frac B2}e^{\frac yB}\e^2\right)^{\frac{99}{100}} \\
&& \lesssim \int   \varepsilon_y^2    \varphi'_{i,B} +\mathcal{N}_{i,\rm loc}^{\frac{9}{10}},
\eee
where we have used $\int_{y<-\frac B2}|y|^{100}e^{\frac yB}\e^2\leq \|\e\|_{L^2}^2\leq \delta(\kappa^*)$.

\begin{remark}\label{gradient}
We see in the above estimate why we need to impose a stronger exponential  weight on $\e_y$ than on 
$\e$ at $-\infty$ in the definition of $\mathcal F_{i,j}$. Indeed, since the global $L^2$ norm of $\e_y$ is not controlled\footnote{because $\l$ becomes large in the (Exit) regime.}, we cannot estimate $\int_{y<0} |y|\psi_B' \e_y^2$ as we did for
$\int_{y<0} |y| \varphi_{i,B}' \e^2$. \end{remark}
 
Together with \fref{eq:2002} and the weighted Sobolev bound \fref{nonlinearsobolev}, this yields the bound:
$$|(f_{3}^{(i,j)})^<|\lesssim (b+\mathcal{N}_{i,\rm loc}^{\frac 12})\left(\int   \varepsilon_y^2    \varphi'_{i,B} +\mathcal{N}_{i,\rm loc}^{\frac{9}{10}}\right)\lesssim \delta(\kappa^*) \int  ( \varepsilon_y^2 +\e^2)   \varphi'_{i,B}  +b^4.$$ For $y>B$, we estimate in brute force using \fref{defphi4} 
$$i\varphi_{i,B}-y\varphi_{i,B}'=0\ \ \mbox{for}\ \ y>B,$$ 
and \eqref{nonlinearsobolev},
$$| (f^{(i,j)}_{3,1})^>| \lesssim (b+\mathcal{N}_{i,\rm loc}^{\frac 12}) \int  ( \varepsilon_y^2 +\e^2)   \varphi'_{i,B} \lesssim \delta(\kappa^*) \int  ( \varepsilon_y^2 +\e^2)   \varphi'_{i,B} .$$

It only remains to estimate $(f_{3,2}^{(i,j)})^>$.
It is a dangerous term which requires:

- the weighted bound \fref{uniformcontrol} and in particular its consequences \fref{weightedone}, \fref{weightedonebis} which are  additionnal information necessary to close the estimates;

- the following  cancellation manufactured in the definition \fref{defjj} from \fref{eq:2004}, \fref{controlj}:
\bea
\label{est:1}
\nonumber
&& \left|(\mathcal{J}_{i,j})_s - (2(j-1)+i) (1+{\mathcal J_{i,j}}){\frac{{\lambda}_s}{{\lambda}}}\right|\\ && =  \frac{4(j-1)+2i}{(1-J_1)^{4(j-1)+2i+1}}\left|(J_1)_s-\frac12\lsl(1-J_1)\right|\lesssim   |b| +  \mathcal{N}_{i,\rm loc}
\eea
\begin{remark}
\label{remarkj}
Note that   the gain in \fref{est:1} with respect to \eqref{eq:2002}   motivates the presence of the term $(1+\mathcal J_{i,j})$ in \fref{feps}.
\end{remark}
The estimates \fref{est:1}, \fref{weightedone}, \fref{weightedonebis} together with the bootstrap bounds \fref{boundnwe}, \fref{bootassumption} and the control \fref{locpasloc} imply:
\bee
|(f_{3,2}^{(i,j)})^>| & \lesssim & ( |b| +  \mathcal{N}_{i,\rm loc})  \left(1+\frac{1}{\l^{\frac {10}9}}\right)\mathcal{N}_{i,\rm loc}^{\frac 89} \\
& \lesssim &   |b| \left(1+\delta(\kappa^*) |b|^{-\frac 5 9} \right)\mathcal{N}_{i,\rm loc}^{\frac 89} +  \mathcal{N}_{i,\rm loc}  \left(1+\delta(\kappa^*) \mathcal{N}_{i,\rm loc}^{-\frac 5 9} \right)\mathcal{N}_{i,\rm loc}^{\frac 89}\\
& \lesssim & \delta(\kappa^*)(\mathcal{N}_{i,\rm loc}+|b|^{4}).
\eee
The collection of above estimates yields the bound:
$$|f_3^{(i,j)}|\lesssim \delta(\kappa^*)\left(\int  ( \varepsilon_y^2 +\e^2)   \varphi'_{i,B} +|b|^4\right).$$

{\bf step 6} $f_4^{(i)}$ term.\\

First,
$$
|(Q_b)_s|=\left|b_s P \left( \chi(|b|^{\gamma} y) + \gamma |b|^\gamma y \chi'(|b|^{\gamma} y)\right)\right|\lesssim |b_s| .
$$
We use the following Sobolev bound:
\be\label{sobopsi}
  \|\e^2\sqrt{\psi_{B} }\|^2_{L^{\infty} }  \lesssim    {\delta(\kappa^*)} \int (\e_y^2+\varepsilon^2)  \psi_{B} 
\ee
to obtain
$$\int \psi_B|\e|^5\lesssim 
\|\psi_B^{\frac 12} \e^2\|_{L^\infty}^{\frac 32} \int \psi_B^{\frac 14} \e^2 
\lesssim \left(\int \e^2\right)^{^{\frac 34}} \int   (\e_y^2+\varepsilon^2)  \psi_{B} 
\lesssim \delta(\kappa^*)     \int (\e_y^2+\varepsilon^2)  \psi_{B},
$$
and thus from \fref{eq:2003}, $|Q_b|\leq C$ and \eqref{defphi4},
\bee
|f^{(i)}_4|& \lesssim& |b_s|\int\psi_B(\e^2|Q_b|^3+|\e|^5)\lesssim    \left(b^2 +  \mathcal{N}_{i,\rm loc}\right)  \int (\e_y^2+\varepsilon^2)  \psi_{B} \\
& \lesssim &  \delta(\kappa^*)\int (\e_y^2+\varepsilon^2)  \varphi_{i,B}'.
\eee

\medskip

{\bf step 7} Proof of \fref{lowerbound}.

First, we estimate from the homogeneity of the nonlinearity and  the  Sobolev bound \eqref{sobopsi}
$$
 \int\psi_B \left|(\varepsilon + Q_b)^6  - Q_b^6   - 6 \varepsilon Q_b^5\right|dy   \lesssim  \int\psi_B(|Q_b|^4\e^2+|\e^6|)\lesssim \delta(\kappa^*) \int (\e_y^2+\varepsilon^2)  \psi_{B} .  $$
The upper bound follows immediately.

The lower bound follows from the structure \fref{feps} of $\mathcal F_{i,j}$ which is a localization of the linearized Hamiltonian close to $Q$. Indeed, we rewrite:
\bee
\mathcal F_{i,j}& = &  \int \psi_B\varepsilon_y^2 + \varphi_{i,B}\varepsilon^2- 5\int Q^4\e^2 +  \mathcal J_{i,j}\int\varphi_{i,B}\e^2\\
& - & \frac  13\int\psi_B \left[(\varepsilon + Q_b)^6  -  Q_b^6  - 6 \varepsilon Q_b^5-15Q_b^4\e^2\right]dy-5\int\psi_B(Q_b^4-Q^4)\e^2dy.
\eee
The small $L^2$ term is estimated from \fref{defjj}, \fref{controlj}: $$|\mathcal J_{i,j}|\int\varphi_{i,B}\e^2\lesssim \delta(\kappa^*)\int\varphi_{i,B}\e^2,$$
The non linear term is estimated using the homogeneity of the nonlinearity and the Sobolev bound \fref{sobopsi}:
\bee
&& \int\psi_B \left|(\varepsilon + Q_b)^6  -  Q_b^6  - 6 \varepsilon Q_b^5-15Q_b^4\e^2\right|dy 
\\
&& \lesssim \int \psi_B(|Q_b|^3|\e|^3+|\e|^6) \lesssim \delta(\kappa^*)\int (\e_y^2+\varepsilon^2)  \psi_{B} .
\eee
The coercivity of the linearized energy \fref{eq:A6} together with the choice of orthogonality conditions \fref{ortho1} and a standard localization argument\footnote{see for example the Appendix of \cite{MMannals} for more details} now ensure the coercivity for $B$ large enough: 
$$\int \psi_B\varepsilon_y^2 + \varphi_{i,B}\varepsilon^2- 5 \psi_B Q^4\e^2 \geq \mu \mathcal N_i,$$ and the lower bound \fref{lowerbound} follows.

This concludes the proof of Proposition \ref{propasymtp}.
\end{proof}

%%%%%%%%%%%%%%%%%%%%%%%%%%%%%%%%%%

\subsection{Dynamical control of the tail}
   
%%%%%%%%%%%%%%%%%%%%%%%%%%%%%%%%%%%%%%%%%%%%%%%%

We now provide an elementary dynamical control of the $L^2$ tail on the right of the soliton which will allow us to close the bootstrap bound (H3) of Proposition \ref{propasymtp} in the setting of Theorem \ref{th:2}. Let a smooth function $$\varphi_{10}(y)=\left\{\begin{array}{ll} 0\ \ \mbox{for}\ \ y\leq 0, \\ y^{10}\ \  \mbox{for}\ \ y\geq 1.\end{array}\right., \ \ \varphi_{10}'\geq 0.$$

\begin{lemma}[Dynamical control of the tail on the right]
\label{lemmatail}
Under the assumptions of Proposition \ref{propasymtp}, there holds: 
\be
\label{keyestimate}
\frac{1}{\lambda^{10}}\frac{d}{ds}\left\{\l^{10}\int\varphi_{10}\e^2\right\}\lesssim \mathcal N_{1,\rm loc}+b^2.
\ee
\end{lemma}

\begin{proof}[Proof of Lemma \ref{lemmatail}] We compute from \fref{eqebis}:
\bee
\frac 12\frac{d}{ds}\int\varphi_{10}\e^2& = & \int \e_s \e \varphi_{10}=\int\varphi_{10}\e\left[\lsl\Lambda \e+\left(-\varepsilon_{yy} + \varepsilon - (\varepsilon +Q_b )^5 + Q_b ^5\right)_y   \right.   \\
   & + & \left .\left(\frac {{\lambda}_{s}}{{\lambda}}+{b}\right) {\Lambda} Q_b
+ \left(\frac { x_{{s}}}{\lambda} -1\right) (Q_b  + \varepsilon)_y  + \Phi_{{b}} + \Psi_{{b}}\right].
\eee
We integrate by parts the linear term and use $y\varphi_{10}'={10}\varphi_{10}$ for $y\geq 1$ and $\varphi_{10}'''\ll \varphi_{10}'$ for $y$ large enough to derive the bound
\bee
&&\int\varphi_{10}\e\left[\lsl\Lambda \e+\left(-\varepsilon_{yy} + \varepsilon\right)_y\right] 
\\& = & -\frac12\lsl\int y\varphi_{10}'\e^2-\frac 32\int\varphi_{10}'\e_y^2-\frac12\int\varphi'_{10}\e^2+\frac 12\int\varphi_{10}'''\e^2\\
& \leq & -\frac{10}2\lsl\int \varphi_{10}\e^2-\frac 14\int\varphi_{10}'(\e_y^2+\e^2)+C\mathcal N_{1,\rm loc}.
\eee

The terms involving the geometrical parameters are controlled from the exponential localization of $Q_b$ on the right and \fref{eq:2002}, \fref{eq:2003}:
$$\left|\lsl+b\right|\left|\int\varphi_{10}\e(\Lambda Q_b)\right|\lesssim (b+\mathcal N_{1,\rm loc}^{\frac12})\mathcal N_{i,\rm loc}^{\frac12}\lesssim \mathcal N_{1,\rm loc}+b^2,$$
\begin{align*}\left|\xsl-1\right|\left|\int\varphi_{10}\e(Q_b+\e)_y\right|&\lesssim (b+\mathcal N_{1,\rm loc}^{\frac12})\left[\mathcal N_{1,\rm loc}^{\frac12}+\int\varphi_{10}'\e^2\right]\\ &\lesssim \mathcal N_{1,\rm loc}+b^2+\delta(\kappa^*)\int\varphi_{10}'\e^2,\end{align*}
$$\int\left|\varphi_{10}\e\Phi_b\right|\lesssim |b_s| \mathcal N_{1,\rm loc}^{\frac 12}\lesssim b^2+\mathcal N_{1,\rm loc}.$$
We control similarily the interaction with the error from \fref{eq:201}: $$\int\left|\varphi_{10}\e\Psi_b\right|\lesssim b^2 \mathcal N_{1,\rm loc}^{\frac 12}\lesssim b^2+\mathcal N_{1,\rm loc}.$$
 By integration by parts in the nonlinear term, we can remove all derivatives on $\e$ to obtain
 (using $|Q_b|+|(Q_b)_y|\leq Ce^{-\frac 12 {y}}$ for $y>0$) 
\bee
 \left|\int\varphi_{10}\e \left[(\varepsilon +Q_b )^5 -Q_b ^5\right]_y\right| 
&\lesssim & \int_{y>0} \varphi_{10}e^{-\frac 12 {y}}\e^2(|\e|^3 +1) +\int\varphi_{10}'\e^6\\
&  \lesssim &  \int_{y>0} e^{-\frac 14 {y}} \e^2(|\e|^3 +1)+\int\varphi_{10}'\e^6\eee
Thus, by standard Sobolev estimates,
\bee
 \left|\int\varphi_{10}\e \left[(\varepsilon +Q_b )^5 -Q_b ^5\right]_y\right|  
 \lesssim \mathcal N_{1,\rm loc} +\delta(\kappa^*)\int\varphi'_{10}(\e_y^2+\e^2).
\eee
The collection of above estimates yields the bound: 
$$\frac{d}{ds}\int\varphi_{10}\e^2+{10}\lsl\int\varphi_{10}\e^2\lesssim \mathcal N_{1,\rm loc}+b^2,$$ and \fref{keyestimate} is proved.
\end{proof}

%%%%%%%%%%%%%%%%%%%%%%%%%%%%%%%%%%

\section{Rigidity near the soliton. Proof of Theorem \ref{th:2}}
\label{sectionfour}
%%%%%%%%%%%%%%%%%%%%%%%%%%%%%%%%%%%%%%%%%%%%%%%%

This section is devoted to the proof of the following proposition which classifies the behavior of any solution close to $Q$
and directly implies Theorem \ref{th:2}. Let 
 $u_0\in H^1$ with\be
\label{estinitial}
 u_0= Q+\e_0, \quad \|\e_0\|_{H^1}< \alpha_0, \ \ \int_{y>0}y^{10}\e_0^2(y)dy<1,
\ee 
and let $u(t)$ be the corresponding solution of \eqref{kdv} on $[0,T)$. Let  $\mathcal T_{\alpha^*}$ be the $L^2$ modulated tube around the manifold of solitary waves given by \eqref{tube} and define the exit time:
$$
t^* = \sup\{0<t<T, \hbox{ such that } \forall t'\in [0,t], \ u(t)\in \mathcal{T}_{\alpha^*}\}
$$
which satisfies $t^*>0$ by assumption on the data. We claim:

\begin{proposition}[Rigidity-Dynamical version]\label{PR:4.1}
There exist universal constants $0<\alpha_0^*\ll \alpha^*\ll \kappa^*$ and $C^*>1$ such that the following holds. Let $u_0$ satisfy \eqref{estinitial} with $0<\alpha_0< \alpha_0^*$, then $u(t)$ satisfies the  assumptions {\rm (H1)-(H2)-(H3)} of Proposition \ref{propasymtp}  on $[0,t^*)$. 

Moreover, let $t_1^*$ be the separation time defined as:
\begin{align}
&\nonumber 
\hbox{$t_1^*=0$, if $|b(0)| \geq  C^* \mathcal{N}_1(0)$},\\
\label{deftonestar}
&  t_1^* = \sup\{0<t<t^* \hbox{ such that } \forall t'\in [0,t], \  |b(t)| < C^* \mathcal{N}_1(t)\},\hbox{ otherwise.}
\end{align}
Then the following dichotomy holds:
\medskip
 
 \noindent
{\bf   (Soliton)}  If $t_1^*=t^*$  then $t_1^*=t^*=T=+\infty$. In addition,
\bea\label{sol1}
&& \mathcal N_2(t) \to 0, \quad b(t) \to 0, \quad \hbox{as $t\to +\infty$},
\\
\label{sol2}
&& \l(t) = \l_\infty(1+o(1)), \quad x(t) = \frac t{\lambda_\infty^2} (1+o(1)),\quad
\hbox{as $t\to +\infty$,}
\eea 
for some $\l_\infty$ satisfying $|\lambda_\infty-1|\leq \delta(\alpha_0)$.

 \medskip
 
 \noindent
{\bf (Exit)} If  $t_1^*<t^*$ with $b(t_1^*)\leq - C^* \mathcal{N}_1(t_1^*)$, then $t^*<T$. In particular,
\be\label{exi1}\inf_{\l_0>0, \ x_0\in \RR}\Big\|u(t^*)-\frac 1{\l_0^{\frac 12}} Q\left(\frac {.-x_0}{\l_0}\right)\Big\|_{L^2} = \alpha^*.\ee In addition:
\be\label{exi2}\lambda(t^*)\geq    \frac{C(\alpha^*)}{\delta(\alpha_0)}.
\ee
\medskip
 
 \noindent
{\bf (Blow up)} If $t_1^*<t^*$ with $b(t_1^*)\geq C^* \mathcal{N}_1(t_1^*)$, then
$t^*=T$. In addition $T<+\infty$ and there exists $0<\ell_0<\delta(\alpha_0)$ such that
\be
\label{det0}
 \lim_{t\to T} \frac {{\lambda}(t)}{(T-t)}= \ell_0, \quad \lim_{t\to T} \frac {b(t)}{(T-t)^2} = \ell_0^3,\quad   \lim_{t\to T}  {(T-t)} x(t)= \frac 1{\ell_0^2},
 \ee
 and there holds the bounds:
 \be
\label{boundhoneglobal}
\|\e_x(t)\|_{L^2}\lesssim \l^2(t)\left[|E_0|+ \delta(\alpha_0)\right],  \ \ \|\e(t)\|_{L^2}\lesssim \delta(\alpha_0).
\ee
 \end{proposition}

\begin{remark}\label{re:PR41}
Note   that $u(t)$ belongs to the tube $\mathcal{T}_{\alpha^*}$ as long as $\frac 13 \le \l(t) \leq 3$ and that the three cases are   equivalently characterized by:
\smallskip

(Soliton) For all $t$, $\l(t) \in [\frac 12, 2]$.
\smallskip

(Exit) There exists $t_0>0$ such that $\l(t_0)>2$.
\smallskip

(Blow up) There exists $t_0>0$ such that $\l(t_0)<\frac 12$.
\smallskip

\noindent A  continuity argument thus ensures that the cases (Exit) and (Blow up) are open in $\mathcal A$.

\smallskip

  Also, note that on $(t_1^*,t^*)$, $\l(t)$ is almost monotonic known for $t>t_1^*$ and the separation time $t_1^*$ defines a trapped regime i.e $$|b(t)| \gtrsim C^* \mathcal{N}_1(t) \ \ \mbox{for}\ \ t\ge t_1^*,$$ and hence the scenario is chosen at this point. 
\end{remark}

The rest of this section is devoted to the proof of Proposition \ref{PR:4.1}. First, note that by Lemma \ref{le:2}, $u$ admits a decomposition on $[0,t^*]$:
$$u(t,x)=\frac{1}{\lambda^{\frac 12}(t)}(Q_{b(t)}+\e)\left(t,\frac{x-x(t)}{\l(t)}\right)$$ with thanks to \fref{estinitial}:
\be
\label{smalnenoneoiebis}
\|\e(0)\|_{H^1}+|b(0)|+\left|1-\lambda(0)\right|\lesssim \delta(\alpha_0), \ \ \int_{y>0}y^{10}\e^2(0)dy\leq 2.
\ee
In particular, arguing as in the proof of \eqref{weightedone}, we have
\be
\label{smalnenoneoiebisbis}
\mathcal N_2(0)\lesssim \delta(\alpha_0).
\ee

For $\kappa^*$ as in Proposition \ref{propasymtp}, define
$$
t^{**}=\sup \{0<t<t^*   \hbox{ such that  $u$ satisfies   (H1)--(H2)--(H3)  on $[0,t]$}\}
.$$
Note that $t^{**}>0$ is well-defined from \eqref{smalnenoneoiebis}, \eqref{smalnenoneoiebisbis} and  a straightforward continuity argument.
Recall that  $s=s(t)$ is the rescaled time \fref{rescaledtime}, and we let $ s^{**}=s(t^{**})$ and $s^*=s(t^*)$. One important step of the proof is to obtain $t^{**}=t^*$ by improving (H1)--(H2)--(H3) on $[0,t^{**}]$.

\subsection{Consequence of the monotonicity formula}

We start with {\it coupling} the dispersive bounds \fref{lyapounovconrol}, \fref{lyapounovconrolbis} with the modulation equation for $b$ given by \fref{eq:bl2} to derive the key rigidity property at the heart of our analysis.

 \begin{lemma}\label{le:4.1}
The following holds:\\
{\rm 1. Dispersive bounds.}
For $i=1,2$, for all $  0\leq s_1\leq s_2< s^{**}$,
\be
\label{estfondamentale}
\mathcal N_i(s_2)+\int_{s_1}^{s_2} \int \left(\e_y^2+\e^2\right) (s) \varphi'_{i,B} ds \lesssim  \mathcal N_i(s_1)+ |b^3(s_2)|+|b^3(s_1)|,
\ee
\be
\label{estfondamentalebis}
\frac{\mathcal N_i(s_2)}{\lambda^2(s_2)}+\int_{s_1}^{s_2}\frac{\int \left(\e_y^2+\e^2\right) (s) \varphi'_{i,B}+|b|^{4}}{\lambda^2(s)}ds\lesssim \frac{\mathcal N_{i}(s_1)}{\lambda^2(s_1)}+ \left[\frac{|b^3(s_1)|}{\lambda^2(s_1)}+\frac{|b^3(s_2)|}{\lambda^2(s_2)}\right].
\ee
\\
{\rm 2. Control of the dynamics for $b$.}
For all $  0\leq s_1\leq s_2< s^{**}$,
\be
\label{contorlbonehoeh}
\int_{s_1}^{s_2}b^2(s)ds\lesssim \mathcal N_1(s_1)+|b(s_2)|+|b(s_1)|,
\ee
and  for a universal constant $K_0>1$,
\be
\label{conrolbintegre}
\left|\frac{b(s_2)}{\lambda^2(s_2)}-\frac{b(s_1)}{\lambda^2(s_1)}\right|\leq
K_0 \left[\frac{b^2(s_1)}{\lambda^2(s_1)}+\frac{b^2(s_2)}{\l^2(s_2)}+ \frac{\mathcal N_{1}(s_1)}{\lambda^2(s_1)}\right].
\ee
\\
{\rm 3. Control of the scaling dynamics.}
Let $\lambda_0(s) = \lambda(s) (1-J_1(s))^2$. Then on $[0,s^{**})$,
\be\label{dvnkoenneoneor}
 \left|\frac{(\l_0)_s}{\l_0}+b\right| \lesssim  \int \e^2 e^{-\frac {|y|}{10}}
 +|b|\left(\mathcal N_2^{\frac 12}+|b|\right).
\ee
\end{lemma}
 
 \begin{proof}
 \emph{Proof of \eqref{estfondamentale} and \eqref{estfondamentalebis}.}
We first observe from \fref{eq:p13} the bound: 
\be
\label{borneintegree}
b^2\leq -b_s+C\mathcal N_{1,\rm loc}.
\ee
By the monotonicity formula \fref{lyapounovconrol} with \fref{lowerbound}: 
\bee
\mathcal N_i(s_2)+\int_{s_1}^{s_2}\int \left(\e_y^2+\e^2\right) (s) \varphi'_{i,B}ds& \lesssim &\mathcal F_{i,1}(s_2) +\mu \int_{s_1}^{s_2}\int \left(\e_y^2+\e^2\right) (s) \varphi'_{i,B}ds\\
& \leq & \mathcal F_{i,1}(s_1) + \int_{s_1}^{s_2} b^4(s)ds\\
& \lesssim & \mathcal N_{i}(s_1)  + \int_{s_1}^{s_2} b^4(s)ds
\eee
and thus using \fref{borneintegree}, \fref{eq:noloc} and $|b|$ small,
$$\mathcal N_i(s_2)+\int_{s_1}^{s_2}\int \left(\e_y^2+\e^2\right) (s) \varphi'_{i,B}ds \lesssim  \mathcal N_i(s_1)+ |b^3(s_2)|+|b^3(s_1)| .
$$

Similarily, from \fref{lyapounovconrolbis}, \fref{lowerbound}:
\bea
\label{nekonveoheo}
\nonumber &&\frac{\mathcal N_i(s_2)}{\lambda^2(s_2)}+\int_{s_1}^{s_2} \frac{1}{\lambda^2(s)}{\int \left(\e_y^2+\e^2\right) (s) \varphi'_{i,B}}ds \\
\nonumber &&\lesssim  \frac{\mathcal F_{i,2}(s_2)}{\lambda^2(s_2)}+\mu \int_{s_1}^{s_2} \frac {1} {\lambda^2(s)}{\int \left(\e_y^2+\e^2\right) (s) \varphi'_{i,B}}ds\\
 & &\lesssim   \frac{\mathcal F_{i,2}(s_1)}{\lambda^2(s_1)}+\int_{s_1}^{s_2}\frac{b^4(s)}{\lambda^2(s)}ds\lesssim  \frac{\mathcal N_{i}(s_1)}{\lambda^2(s_1)}+\int_{s_1}^{s_2}\frac{b^4(s)}{\lambda^2(s)}ds.
\eea
We now integrate by parts in time using \fref{borneintegree}, \fref{eq:2002} to estimate:
\bee
&&\int_{s_1}^{s_2}\frac{b^{4}(s)}{\lambda^2(s)}ds \leq \int_{s_1}^{s_2}\frac{-b^2b_s}{ \l^2}+ {\delta(\kappa^*)} \int_{s_1}^{s_2}\frac{\mathcal N_{1,\rm loc}(s)}{\lambda^2(s)}ds\\
& = & -\frac{1}{3}\left[\frac{b^3}{\l^2}\right]_{s_1}^{s_2}-\frac{ 2}{3}\int_{s_1}^{s_2} b^3\frac{\l_s}{\l^3}ds+ {\delta(\kappa^*)} \int_{s_1}^{s_2}\frac{\mathcal N_{1,\rm loc}(s)}{\lambda^2(s)}ds\\
& \leq & \left[\frac{|b^3(s_1)|}{\lambda^2(s_1)}+\frac{|b^3(s_2)|}{\lambda^2(s_2)}+ {\delta(\kappa^*)} \int_{s_1}^{s_2}\frac{\mathcal N_{1,\rm loc}(s)}{\lambda^2(s)}ds\right]+\frac{2}{3}\int_{s_1}^{s_2}\frac{b^4(s)}{\lambda^2(s)}ds\\
&& +C   \int_{s_1}^{s_2}\frac{|b|^3}{\l^2}\left[b^2+\mathcal N_{1,\rm loc}^{\frac 12}\right]ds\\
& \leq& \left[\frac{|b^3(s_1)|}{\lambda^2(s_1)}+\frac{|b^3(s_2)|}{\lambda^2(s_2)}\right]+ {\delta(\kappa^*)} \int_{s_1}^{s_2}\frac{\mathcal N_{1,\rm loc}(s)}{\lambda^2(s)}ds+\left[\frac{2}{3}+\delta(\kappa^*)\right]\int_{s_1}^{s_2}\frac{b^4(s)}{\lambda^2(s)}ds
\eee
and thus for $\kappa^*$ small,
\be
\label{cnkenoen}
\int_{s_1}^{s_2}\frac{b^4(s)}{\lambda^2(s)}ds \lesssim\left[\frac{|b^3(s_1)|}{\lambda^2(s_1)}+\frac{|b^3(s_2)|}{\lambda^2(s_2)}\right]+\delta(\kappa^*)\int_{s_1}^{s_2}\frac{\mathcal N_{1, \rm loc}(s)}{\lambda^2(s)}ds.
\ee
Injecting this bound into \fref{nekonveoheo} concludes the proof of \fref{estfondamentalebis}.\\

The virtue of \fref{estfondamentale}, \fref{estfondamentalebis} is to reduce the control of the full problem to the sole control of the parameter $b$ which is driven by the sharp ODE \fref{eq:bl2}.
\\

\emph{Proof of \eqref{contorlbonehoeh} and \eqref{conrolbintegre}.}
The estimate \fref{contorlbonehoeh} is derived by integrating \fref{borneintegree} in time using \fref{estfondamentale}. We then compute from \fref{eq:bl2}, \fref{eq:2002} and the a priori bound\footnote{recall that $J$ given by \fref{defjgg} is a well localized $L^2$ scalar product.} $|J|\lesssim \mathcal{N}^{\frac 12}_{1,\rm loc}$:
\bea
\label{cneoneoenoe}
\nonumber \left|\frac d{ds}\left\{\frac b{{\lambda}^2} e^{J} \right\}\right|&  = & \left|\frac d{ds}\left\{\frac b{{\lambda}^2} \right\}+\frac{b}{\l^2}J_s\right|e^J\lesssim \left|\frac {{\lambda}_s}{{\lambda}}  \frac b{{\lambda}^{2}}   J\right| +\frac 1{{\lambda}^2} \left( \int \varepsilon^2 {e^{-\frac {|y|}{10}}} +  |b|^3\right)\\
\nonumber & \lesssim & \frac {b^2}{\lambda^2}  |J|+ \frac{1}{\l^2}\left( \mathcal N_{1,\rm loc} +  |b|^3\right)\lesssim  \frac {b^2}{\lambda^2}  \mathcal{N}^{\frac 12}_{1,\rm loc}+\frac{1}{\l^2}\left( \mathcal N_{1,\rm loc} +  |b|^3\right)\\
& \lesssim & \frac{1}{\l^2}\left( \mathcal N_{1,\rm loc} +  |b|^3\right).
\eea
We integrate this estimate in time and use \eqref{borneintegree}, \fref{estfondamentalebis} to get
\bee
\left|\left[\frac{b}{\lambda^2}e^{J}\right]_{s_1}^{s_2}\right|\lesssim \frac{\mathcal N_{1}(s_1)}{\lambda^2(s_1)}+\ \left[\frac{b^2(s_1)}{\lambda^2(s_1)}+\frac{b^2(s_2)}{\lambda^2(s_2)}\right]
\eee
and \fref{conrolbintegre} follows from $|e^J-1|\leq |J|\lesssim \mathcal{N}^{\frac 12}_{1}$ and \eqref{estfondamentalebis}.\\
 
\emph{Proof of \eqref{dvnkoenneoneor}.}
We integrate the scaling law using the sharp modulation equation \fref{eq:2004}. From \fref{controlj}: 
\be
\label{comparisonlamnda}
\left|\frac{\l}{\l_0}-1\right|\lesssim |J_1|\lesssim \delta(\kappa^*),
\ee
and  thus from \fref{eq:2004}, we get
\bee
\nonumber \left|\frac{(\l_0)_s}{\l_0}+b-c_1 b^2 \right|& = & \left|\frac{1}{1-J_1}\left[(1-J_1)\lsl+b-2(J_1)_s\right]-\frac{J_1}{1-J_1}b\right|\\
& \lesssim &  \int \e^2 e^{-\frac {|y|}{10}}+|b|(\mathcal N_2+|b|^2).
\eee
This concludes the proof of Lemma \ref{le:4.1}.
\end{proof}

We are now in position to prove the dichotomy of Proposition \ref{PR:4.1}. Let 
\be
\label{cnincneone}
C^*=10 K_0
\ee
where $K_0$ is the universal constant  in \eqref{conrolbintegre} and let the separation time $t_1^*$ be given by \fref{deftonestar}.

\subsection{The soliton case.}
Assume that
\be\label{S1}t_1^*=t^* \quad \hbox{ i.e.  for all $t\in [0,t^*]$, $|b(t)| \leq C^* \mathcal N_1(t)$}.\ee

We first prove that in this case $t^{**}=t^*$ which means that the bootstrap estimates  (H1)--(H2)--(H3) of Proposition \ref{propasymtp} hold on $[0,t^*]$. Indeed, we claim: $\forall s \in [0,s^{**})$, 
\be
\label{H1mieux}
 |b(s)|+\mathcal N_2(s)+\|\e(s)\|_{L^2}+|1-\l(s)|\lesssim \delta(\alpha),
\ee
\be
\label{H2mieux}
\frac{|b(s)|+\mathcal N_2(s)}{\lambda^2(s)}\lesssim \delta(\alpha),
\ee
\be
\label{H3mieux}
\int_{y>0}y^{10}\e^2(s)dy\leq 5.
\ee
Taking $\alpha^*>0$ small enough (compared to $\kappa^*$), this garantees by a standard continuity argument that
$t^{**}=t^*$.
\medskip

\emph{Proof of \eqref{H1mieux}--\eqref{H3mieux}.}
First, observe that by \eqref{defphi4tri}, \eqref{defphi} and the definition of
$t^{**}$, on $[0,s^{**}]$,
\be
\label{comparisonnormesloc}
\matchal N_{1}\lesssim \int (\e_y^2 + \e^2) \varphi_{2,B}',\quad
\mathcal N_{1}\lesssim \mathcal N_2\lesssim \delta(\kappa^*).
\ee
Therefore, from \eqref{S1}, \fref{estfondamentale} and \fref{eq:2003}: $\forall s\in [0,s^{**})$,
\bee
|b(s)-b(0)|& \leq &\int_0^{s}|b_s|ds\lesssim \int_0^{s}(b^2+\mathcal N_{1,\rm loc})ds\lesssim  \int_0^{s}( \delta(\kappa^*)(C^*)^2+1)\mathcal N_{1}(s)ds\\
& \lesssim & \int_0^{s}\int (\e_y^2 + \e^2)(s) \varphi_{2,B}'ds\lesssim \mathcal N_2(0)+\delta(\kappa^*)(|b(s)|+|b(0)|).
\eee
We thus conclude from \fref{smalnenoneoiebis} : $\forall s\in [0,s^{**})$, $$|b(s)|\lesssim |b(0)|+\mathcal N_2(0)\lesssim \delta(\alpha_0).$$ 
Then, from \fref{estfondamentale} and \fref{contorlbonehoeh}, 
\be
\label{neononone}
\matchal N_2(s)+\int_0^{s}\left(b^2+\int (\e_y^2 + \e^2)(s) \varphi_{2,B}'\right)ds\lesssim \delta(\alpha_0).
\ee Injecting this into the conservation of the $L^2$ norm \fref{twobound} using \fref{eq:204} ensures $$\int|\e|^2\lesssim \delta(\alpha_0),$$ and \fref{H1mieux} is proved. Note that we also have
from \eqref{controlj}
\be
\label{bornesjonejtwo}
|J_1|+|J_2|\leq  \delta(\alpha_0).
\ee 
We now compute the variation of scaling from \fref{dvnkoenneoneor} which together with \eqref{S1} implies:
$$\left|\frac{(\lambda_0)_s}{\lambda_0}\right|\lesssim |b|+\mathcal N_{1,\rm loc}\lesssim \mathcal N_1\lesssim \int (\e_y^2 + \e^2)(s) \varphi_{2,B}'$$ and thus from  \fref{neononone}: $\forall 0\leq s<s^{**}$, 
$$\left|\log\left(\frac{\lambda_0(s)}{\lambda_0(0)}\right)\right|\lesssim \matchal N_2(0)+\delta(\alpha_0)\lesssim \delta(\alpha_0).
$$
Hence from \fref{comparisonlamnda}, \fref{bornesjonejtwo}:\\
$$\left| \left(\frac{\lambda(s)}{\lambda(0)}-1 \right)\right|\lesssim   \delta(\alpha_0)$$ which with \fref{smalnenoneoiebis} implies:
\be
\label{uniformboundscaling}
\forall s\in [0,s^{**}), \ \ |1-\lambda(s)|\lesssim \delta(\alpha_0).
\ee
Together with \fref{H1mieux}, this  implies \fref{H2mieux}. We now integrate \fref{keyestimate} using \fref{smalnenoneoiebis}, \fref{uniformboundscaling}, \fref{neononone} and obtain:
\bee
\int y^{10}\e^2(s)dy& \leq & \frac{\lambda^{10}(0)}{\lambda^{10}(s)}\int y^{10}\e^2(0)dy+\frac{C}{\lambda^{10}(s)}\int_{0}^{s}\lambda^{10}(s)(\mathcal N_{1,\rm loc}+b^2(s))ds\\
& \leq & 2+\delta(\alpha_0)\leq 3
\eee
and \fref{H3mieux} is proved. 

\medskip

We therefore conclude that $t^*=T$ and $u(t)$ remains in the tube $\mathcal T_{\alpha^*}$ for all $t\in [0,T)$ from \fref{H1mieux}. Moreover, inserting \eqref{H1mieux} in the conservation of the energy \fref{energbound}, we get
  $$\forall t\in [0,T), \ \ \| \e_y(t)\|_{H^1}\lesssim C.$$ Hence the solution $u(t)$
  is uniformly bounded in $H^1$ and thus global: $T=+\infty$. 
  
  It remains to show the convergence \fref{sol1}--\fref{sol2}. From \fref{eq:2003}, \fref{neononone}, \fref{uniformboundscaling}: \be
\label{cnoenoneoneo}
\int_0^{+\infty}|b_t|dt\lesssim \int_0^{+\infty}|b_s|ds\lesssim \int_0^{+\infty}\left(b^2+\int (\e_y^2 + \e^2)(s) \varphi_{2,B}'\right)ds\lesssim \delta(\alpha_0)
\ee which implies 
\be
\label{behvaourb}
\lim_{t\to+\infty}b(t)=0
\ee and the existence of a sequence $t_n\to \infty$ such that 
$$ \int (\e_y^2 + \e^2)(t_n) \varphi_{2,B}'  \to 0\ \ \mbox{as}\ \ t_n\to +\infty.$$ 
By \eqref{comparisonnormesloc}, $ \mathcal N_1(t_n)\to 0$ as $n\to +\infty$ and thus using the monotonicity \fref{estfondamentale}: $$\mathcal N_1(t)\to 0\ \ \mbox{as}\ \ t\to \infty.$$ Together with the uniform bound \fref{H3mieux}, we also obtain
\be
\label{behvaourbbis}
\mathcal N_2(t)\to 0\ \ \mbox{as}\ \ t\to \infty.
\ee
Finally, from \fref{dvnkoenneoneor}, \fref{cnoenoneoneo}: $$\int_0^{+\infty}\left|\log\left(\frac{(\l_0)_t}{\l_0}\right)\right|dt\lesssim \int_0^{+\infty}\left|\log\left(\frac{(\l_0)_s}{\l_0}\right)\right|ds\lesssim \delta(\alpha_0)$$ and thus $$\lim_{t\to+\infty}\lambda_0(t)=\lambda^{\infty}_0\ \ \mbox{with}\ \ \left|\lambda_0^{\infty}-1\right|\lesssim \delta(\alpha_0).$$ Now from \fref{behvaourbbis}: $$|J_1|\lesssim \mathcal N_2^{\frac12}\to 0\ \ \mbox{as}\ \ t\to +\infty$$ and thus from \fref{contorlbonehoeh}: 
\be
\label{convscaling}
\lim_{t\to+\infty}\lambda(t)=\lambda^{\infty}\ \ \mbox{with}\ \ \left|\lambda^{\infty}-1\right|\lesssim \delta(\alpha_0).
\ee 
The translation parameter is controlled using \fref{eq:2002} and \fref{behvaourbbis}, \fref{convscaling} which imply: $$x_t=\frac{1}{\lambda^2}\xsl=\frac{1+o(1)}{\lambda_{\infty}^2}\  \mbox{as}\  \ t\to+\infty.$$
 
This concludes the proof of  \fref{sol1}, \fref{sol2}. 

\subsection{Exit case}

Now, we assume $t_1^*<t^*$ and 
\be
\label{cneioniehovhe}
b(s_1^*)\leq -C^*\mathcal N_1(s_1^*).
\ee
Observe first that arguing on $[0,s_1^*]$  as in the soliton case, where the parameter $b$ is controlled by $\mathcal N_1$, we get
 $\forall s\in[0,s_1^*],$
\be
\label{estcrucialscaling}
|\lambda(s)-1|+|b(s)|+\mathcal N_2(s)+\int_{0}^{s}  \int (\e_y^2+ \e^2) \varphi_{2,B}' ds\lesssim \delta(\alpha_0),
\ee
\be
\label{ncieoneoneoino}
  \int_{y>0}y^{10}\e^2(s)dy\leq 5.
\ee
In particular, $t_1^*<t^{**}\leq t^*$.
Now, we claim $$t^{**}=t^*\quad  \hbox{and} \quad t^*<T,$$ which means that the solution leaves the tube $\mathcal{T}_{\alpha^*/2}$
in finite time.
\medskip

\emph{Proof of $t^{**}=t^*$.} We improve (H1)--(H2)--(H3) on $[t_1^*,t^{**}]$ to obtain $t^{**}=t^*$.
The proof is different than the one for the soliton case since now $b$ is not controlled by $\mathcal N_1$.
{\it The} fundamental observation is that \fref{conrolbintegre}, \fref{cnincneone}, \fref{cneioniehovhe} immediately imply the rigidity: 
\be
\label{controlbrigidityone}
\forall s\in[s_1^*,s^{**}), \ \ -2\ell^*\leq \frac{b(s)}{\l^2(s)}\leq -\frac{\ell^*}{2}
\ee
where we have set from \fref{estcrucialscaling}:
\be
\label{deflstart}
 \ell^*=\frac{b(s^*_1)}{\lambda^2(s^*_1)} \leq  -C^* \frac {\mathcal{N}_1(s_1^*)}{\lambda^2(s^*_1)}<0, \ \ |\ell^*|\lesssim \delta(\alpha_0).
\ee
Together with \fref{estfondamentalebis} and \fref{estcrucialscaling}, this implies the bound: $$\forall s\in [0,s^*], \ \ \frac{|b(s)|+\mathcal N_2(s)}{\l^2(s)}\lesssim \delta(\alpha_0)$$ and (H2) is improved for $\alpha^*$
small compared to $\kappa^*$. We now observe using $b<0$ from \fref{controlbrigidityone} and \fref{dvnkoenneoneor}: $\forall s\in [s_1^*,s^{**})$, $$\frac{(\lambda_0)_s(s)}{\lambda_0(s)}\gtrsim -\mathcal N_{1,\rm loc}.$$ Together with \fref{estfondamentale} and the definition of $\lambda_0$, this yields the almost monotonicity
property of $\l$:
\be
\label{almostmofm}
\forall s_1^*\leq \sigma_1\leq \sigma_2<s^{**}, \ \ \lambda(\sigma_2)\geq \frac 12 \lambda(\sigma_1).
\ee
We now integrate \fref{keyestimate} using \fref{estfondamentale}, \fref{almostmofm}, \fref{estcrucialscaling}, \fref{ncieoneoneoino} and \fref{contorlbonehoeh} to obtain: $\forall s_1^*\leq s<s^{**}$,
\bee&&
\int \varphi_{10} \e^2(s)dy\\ & \leq &  \frac{\lambda^{10}(s_1^*)}{\lambda^{10}(s)}\int \varphi_{10}\e^2(s_1^*)dy+\frac{C}{\lambda^{10}(s)}\int_{s_1^*}^{s}\lambda^{10}(s')(\mathcal N_{1,\rm loc}(s')+b^2(s'))ds'\\
& \leq & 3+C\int_{s_1^*}^{s}(\mathcal N_{1,\rm loc}(s')+b^2(s'))ds'\leq 3+C(|b(s_1^*)|+|b(s)|+\mathcal N_1(s_1^*))\\
& \leq&  3+\delta(\kappa^*)
\eee 
and (H3) is improved. We now improve (H1). Since 
$u(t) \in \mathcal{T}_{\alpha^*}$ on $[0,t^{*})$, we have by \eqref{controle}, 
$\forall s \in [0,s^{*})$,  $|b(s)|\leq \delta(\alpha^*)\ll \kappa^* $. 
By \eqref{estfondamentale}, it follows that for all $s\in [0,s^{**})$, $N_2(s)\ll
\kappa^*$. By \eqref{twobound}, for all $s\in [0,s^{**})$ $\|\e(s)\|_{L^2} \ll \kappa^*$, and
(H1) is improved. In conclusion, we have proved $t^{**}=t^*$ again in this case.\\

\emph{Proof of $t^*<T$.}  Let us now show that (Exit) occurs in finite time. We divide \fref{dvnkoenneoneor} by $\lambda_0^2$ and use \fref{controlbrigidityone}, \fref{comparisonlamnda} to estimate on $[t_1^*,t^*)$: $$\frac{|\ell^*|}{3}-C \frac{\mathcal N_{1,\rm loc}}{\l^2}\leq  (\l_0)_t\leq 3|\ell^*|+C \frac{\mathcal N_{1,\rm loc}}{\l^2}.$$ 
Integrating in time, for all $t\in [t_1^*,t^*)$, we get
$$
\frac{|\ell^*|(t-t_1^*)}{3}-C_1\int_{t_1^*}^{t}\frac{\mathcal N_{1,\rm loc}}{\l^2} \leq \lambda_0(t)-\lambda_0(t_1^*)\leq 3|\ell^*|(t-t_1^*)+C_2\int_{t_1^*}^{t_2}\frac{\mathcal N_{1,\rm loc}}{\l^2} .$$
From the monotonicity \fref{almostmofm} and then \fref{estfondamentale}:
$$\int_{t_1^*}^{t}\frac{\mathcal N_{1,\rm loc}}{\l^2} =\int_{s_1^*}^{s}\l \mathcal N_{1,\rm loc}\lesssim \lambda(s)\int_{s_1^*}^{s}\mathcal N_{1,\rm loc} \lesssim \delta(\kappa^*)\lambda(t),$$ and we therefore obtain the bound: $\forall t\in[t_1^*,T^*)$, $$\frac14 \left(|\ell^*|(t-t_1^*) +  \lambda_0(t_1^*)\right)\leq \lambda(t)\leq 4\left(|\ell^*|(t-t_1^*)+    \lambda_0(t_1^*)\right).$$ 
This yields the following estimates on $b$ from \fref{controlbrigidityone}: $\forall t\in[t_1^*,t^*)$, 
\be
\label{cenoenoneo}
-40|\ell^*|\left(|\ell^*|(t-t_1^*) +\lambda _0(t_1^*)\right)^2 \leq b(t)\leq -\frac{|\ell^*|}{40}\left(|\ell^*|(t-t_1^*)+\lambda_0(t_1^*)\right)^2.
\ee 
Injecting this bound into \fref{estfondamentale} yields the control $$\mathcal N_2(t)\lesssim C(t)$$ which injected into the energy and mass conservation laws \fref{twobound}, \fref{energbound} yields the $H^1$ bound $$\|\e(t)\|_{H^1}\lesssim C(t).$$ 

It follows that $t^*=T<+\infty$ is not possible. On the other hand, $t^*=T=+\infty$ is also impossible since then by \eqref{cenoenoneo}, $b(t)\to -\infty$ as $t\to +\infty$, which contradicts the definition of $t^*$. Thus, $t^*<T\leq +\infty$.

Finally, we   observe that the scaling parameter is large at the exit time for $\alpha$ small compared to $\alpha^*$. Indeed, $|b(t^*)| \gtrsim (\alpha^*)^4$ from \eqref{twobound} and thus from \fref{controlbrigidityone}, \fref{deflstart}:
$$\lambda^2(t^*)\geq \frac 12  \frac{|b(t^*)|}{|\ell^*|}\geq  \frac{C(\alpha^*)}{\delta(\alpha_0)}.$$

%%%%%%%%%%%%%%%%%%%%%%%%%%%%%%%%%%%%%%%%%%%%%%%%%%%%%%%%%%%%%%%%%%%% 
%%%%%%%%%%%%%%%%%%%%%%%%%%%%%%%%%%%%%%%%%%%%%%%%%%%%%%%%%%%%%%%%%%%% 

\subsection{Blow up case}

%%%%%%%%%%%%%%%%%%%%%%%%%%%%%%%%%%%%%%%%%%%%%%%%%%%%%%%%%%%%%%%%%%%% 
%%%%%%%%%%%%%%%%%%%%%%%%%%%%%%%%%%%%%%%%%%%%%%%%%%%%%%%%%%%%%%%%%%%% 

We now assume $t_1^*<t^*$ and
\be
\label{bootsunnegatifbis}
b(s_1^*)\geq C^*\mathcal N_1(s_1^*)>0.
\ee
As before,  we have
 $\forall s\in[0,s_1^*],$
\be
\label{estcrucialscalingbis}
|\lambda(s)-1|+|b(s)|+\mathcal N_2(s)+\int_{0}^{s} \int (\e_y^2+ \e^2) \varphi_{2,B}' ds\lesssim \delta(\alpha),
\ee
\be
\label{vhiohvihoieohe}
 \int_{y>0}y^{10}\e^2(s)dy\leq 5.
\ee
In particular, $t_1^*<t^{**}\leq t^*$. In this case, we claim that $t^{**}=t^*=T$ and $T<\infty$.  

\medskip

\emph{Proof of $t^{**}=t^*=T$.}
First, we  improve   the bounds (H1)--(H2)--(H3) of Proposition \ref{propasymtp}. From \fref{conrolbintegre}, \fref{cnincneone}, \fref{cneioniehovhe} , we recover the rigidity:
\be
\label{controlbrigidityonebis}
\forall s\in[s_1^*,s^{**}), \ \ \frac{\ell^*}{2}\leq \frac{b(s)}{\l^2(s)}\leq 2\ell^*
\ee
where we set from \fref{estcrucialscalingbis}:
\be
\label{deflstartbis}
 \ell^*=\frac{b(s^*_1)}{\lambda^2(s^*_1)}>0, \ \ |\ell^*|\lesssim \delta(\alpha_0).
\ee
Together with \fref{estfondamentalebis} and \fref{estcrucialscalingbis}, this implies the bound: $$\forall s\in [0,s^*], \ \ \frac{|b(s)|+\mathcal N_2(s)}{\l^2(s)}\lesssim \delta(\alpha_0)$$ and (H2) is improved provided $\alpha^*$ is small compared to $\kappa^*$. We now observe from $b>0$ and \fref{dvnkoenneoneor}: on $[s_1^*,s^{**})$, 
$$\frac{(\lambda_0)_s}{\lambda_0(s)}\gtrsim - \mathcal N_{1,\rm loc}$$ which together with \fref{estfondamentale} and the definition of $\lambda_0$, yields the almost monotonicity: 
\be
\label{almostmofmbis}
\forall s_1^*\leq \sigma_1\leq \sigma_2<s^{**}, \ \ \lambda(\sigma_2)\leq \frac 32 \lambda(\sigma_1).
\ee
In particular, from \fref{estcrucialscalingbis}: 
\be
\label{cnoconeono}
\forall s\in[0,s^{**}), \ \ \lambda(s)\leq 2.
\ee
This yields with \fref{estcrucialscalingbis}, \fref{controlbrigidityonebis}, \fref{deflstart}, \fref{estfondamentale}: for all $ 0\leq s\leq s^{**}$,  $$|b(s)|\lesssim \lambda^2(s)\ell^*\lesssim \delta(\alpha_0), \ \ \mathcal N_2(s)+\int_{0}^{s}  \int (\e_y^2+ \e^2) \varphi_{2,B}' ds\lesssim \mathcal \delta(\alpha_0).$$ The conservation of the $L^2$ norm \fref{twobound} implies 
\be
\label{ltwoboums;smsnd}
\|\e\|_{L^2}^2\lesssim \delta(\alpha_0)
\ee
and (H1) is improved. We now integrate \fref{keyestimate} using \fref{estfondamentale}, \fref{cnoconeono}, \fref{estcrucialscalingbis}, \fref{contorlbonehoeh} and obtain: $\forall 0\leq s<s^{**}$,
\bee
\int \varphi_{10}\e^2(s)dy & \leq &  \frac{\lambda^{10}(0)}{\lambda^{10}(s)}\int \varphi_{10}\e^2(0)dy+\frac{C}{\lambda^{10}(s)}\int_{0}^{s}\lambda^{10} (\mathcal N_{1,\rm loc}+b^2)ds\\
& \leq & \frac{1}{\l^{10}(s)}\left[5+C\int_{0}^{s}(\mathcal N_{1,\rm loc}+b^2)ds\right]\leq \frac{5+\delta(\kappa^*)}{\lambda^{10}(s)}.
\eee 
and (H3) is improved. We conclude that $t^{**}=t^*$. Moreover, by \eqref{ltwoboums;smsnd}, for $\alpha_0$ small enough compared to $\alpha^*$, we get $t^*=T$ since the condition in the definition of $\mathcal{T}_{\alpha^*}$ is also improved by this estimate.

\medskip

\emph{Blow up in finite time.} We now divide \fref{dvnkoenneoneor} by $\lambda_0^2$ and use \fref{controlbrigidityonebis}, \fref{comparisonlamnda} to estimate on $[t_1^*,T)$: $$\frac{|\ell^*|}{3}-C \frac{\mathcal N_{1, \rm loc}}{\l^2}\leq  -(\l_0)_t\leq 3|\ell^*|+C \frac{\mathcal N_{1, \rm loc}}{\l^2}.$$ 
We integrate in time and obtain in particular: for all $ t\in [t_1^*,T)$,
\be
\label{cevnonone}
0\leq \lambda_0(t)\leq \lambda_0(t_1^*) 
 - \frac{|\ell^*|(t -t_1^*)}{3}+C_1\int_{t_1^*}^{t }\frac{\mathcal N_{1,\rm loc}}{\l^2}.
\ee
Now from the bound \fref{cnoconeono} again and \fref{estfondamentale}:
$$\int_{t_1^*}^{t }\frac{\mathcal N_{1,\rm loc}}{\l^2}d\tau=\int_{s_1^*}^{s }\l(\sigma)\mathcal N_{1,\rm loc}d\sigma\lesssim 2 \int_{s_1^*}^{s }\mathcal N_{1,\rm loc}d\sigma\lesssim 1,$$ and thus \fref{cevnonone} implies: $$T<+\infty\quad \hbox{and in particular}\quad \lambda(t)\to 0 \hbox { as $t\to T$}.$$ The conservation of energy \fref{energbound} implies \be
\label{energyconservation}
\|  \e_y(t)\|^2_{L^2}\lesssim \lambda^2(t)|E_0|+\mathcal N_2(t)\ee
 and thus from (H2):
 \be
 \label{cnocnoeno}
 \|  \e_y(t)\|_{L^2}+b(t)+\matchal N_2(t)\to 0\ \ \mbox{as}\ \ t\to T.
 \ee
 
 \medskip
 
\emph{Proof of \eqref{det0}--\eqref{boundhoneglobal}.} We estimate from \fref{controlbrigidityonebis}, \fref{cneoneoenoe} and \fref{estfondamentalebis} using $T<+\infty$: $$\int_0^{+\infty}\left|\frac d{ds}\left\{\frac b{{\lambda}^2} e^{J} \right\}\right|ds\lesssim \int_0^{\infty} \frac{1}{\l^2}\left( \mathcal N_{1,\rm loc} +  |b|^3\right)ds<+\infty,$$ and thus $\frac{be^J}{\l^2}$ has a limit as $t\to T$. Moreover, $$|J(t)|\lesssim \mathcal N^{\frac 12}_2(t)\to 0\ \ \mbox{as}\ \ t\to T$$ from (H2), and thus from \fref{controlbrigidityonebis}, \fref{deflstartbis}: 
\be
\label{vjdbkbdjdbkdj}
\frac{b(t)}{\l^2(t)}\to \ell_0>0, \ \ t\to T, \ \ \hbox{with } |\ell_0|\lesssim \delta(\alpha_0).
\ee The time integration of \fref{dvnkoenneoneor} using \fref{vjdbkbdjdbkdj}, \fref{almostmofmbis}, \fref{estfondamentale} yields: 
\bee
\left|\lambda_0(t)-\int_t^T\frac{b}{\lambda^2}dt'\right|& \lesssim & \int_t^T\frac{\lambda(\mathcal N_{1,\rm loc}+o(b)}{\l^2}dt'\lesssim \int_s^{+\infty}\l\mathcal N_{1,\rm loc}ds'+ o(T-t)\\
& \lesssim & o(T-t)+\lambda(s)\int_s^{+\infty}\mathcal N_{1,\rm loc}ds'=o(|T-t|+\lambda(t))
\eee
 and thus using \fref{vjdbkbdjdbkdj} again:$$\lim_{t\to T} \frac {{\lambda_0}(t)}{(T-t)}= \ell_0.$$ Moreover from \fref{comparisonlamnda}:
$$\left|\frac{\l(t)}{\l_0(t)}-1\right|\lesssim |J_1(t)|\to 0\ \ \mbox{as}\ \ t\to T.$$ 
The control of the translation parameter follows from \fref{eq:2002} and (H2) which yield: $$x_t=\frac{1}{\l^2}\xsl=\frac{1}{\l^2}(1+o(1))$$ and \fref{det0} follows. Finally, the $L^2$ bound in \fref{boundhoneglobal} follows from \fref{ltwoboums;smsnd}, and the rest of \fref{boundhoneglobal}  follows from (H2) and the conservation of energy \fref{energbound}:  $$\|\e_y(t)\|_{L^2}^2\lesssim \l^2(t)|E_0|+|b(t)|+\mathcal N_2(t)\lesssim( |E_0|+\delta(\alpha_0))\lambda^2(t).$$  

\medskip

This concludes the proof of Proposition \ref{PR:4.1}.

%%%%%%%%%%%%%%%%%%%%%%%%%%%%%%%%%%

\section{Blow up for $E_0\leq 0$}

%%%%%%%%%%%%%%%%%%%%%%%%%%%%%%%%%%%%%%%%%%%%%%%%

In this section, we let an initial data  $$u_0 \in \mathcal{A}\ \ \mbox{with}\ \ E_0\leq 0.$$ We moreover assume that $u_0$ is not a solitary wave up to symmetries. We claim that the corresponding solution $u(t)$ to gKdV blows up in finite time in the (Blow up) regime described by Proposition \ref{PR:4.1}.
 
 \medskip
 
Let us first recall the following standard orbital stability statement which follows from the variational characterization of the ground state and a standard concentration compactness argument:

\begin{lemma}[Orbital stability]
\label{orbital}
Let $\alpha>0$ small enough and a function $v\in H^1$ such that  $$\left|\int v^2- \int Q^2\right|\leq \alpha, \ \ E(v)\leq \alpha \int v_x^2,$$ then there exist $(\lambda_v,x_v)\in \RR^*_+\times \RR$  such that $$\|Q-\epsilon_0\lambda_v^{\frac12} v(\lambda_v x+x_v )\|_{H^1}\leq \delta(\alpha), \ \ \epsilon_0\in \{-1,1\}.$$
\end{lemma}

For $\alpha>0$ small enough compared to $\alpha^*$, it follows from 
the conservation of mass and energy  that $u$ remains in the tube $\mathcal T_{\alpha^*}$ on $[0,T)$.
Therefore, only the case (Blowup) and (Soliton) can occur in Proposition \ref{PR:4.1}. We argue by contradiction and assume that (Soliton) occurs.\\

{\it Case $E_0<0$}: This case is particularly simple to treat using the estimates of Proposition \ref{PR:4.1}. Indeed, the conservation of energy \fref{energbound} with $E_0< 0$ together with the asymptotic stability statements \fref{sol1}, \fref{sol2} imply: $$\lambda^2(t)|E_0|+\int|   \e_y|^2\lesssim |b(t)|+\mathcal N_1(t)\to 0\ \ \mbox{as}\ \ t\to +\infty,$$ and thus $$\lambda(t)\to 0 \ \mbox{as}\ \ t\to+\infty$$ hence contradicts the dynamics of $\lambda$ \fref{sol2}. 

\medskip

{\it Case $E_0=0$}: This case is substantially more subtle and in particular there is no obvious obstruction to the (Soliton) dynamics.  In fact, the conservation of energy \fref{energbound} yields with \fref{sol1}, \fref{sol2}:
\be\label{Ezero}\int|   \e_y|^2\lesssim |b(t)|+\mathcal N_1(t)\to 0\ \ \mbox{as}\ \ t\to +\infty,
\ee
but there is no further simple information on $\lambda(t)$. Our aim is to show that this $\dot{H}^1$ implies global $L^2$ dispersion, and hence the solution has minimal mass which for $E_0=0$ is possible only for the solitary wave itself.\\ 
By rescaling, we may without loss of generality assume that $\l_\infty=1$ in \eqref{sol2}. We claim

 \begin{lemma}[$L^2$ compactness]
\label{cl:5}
Assume $E_0=0$ and $u(t)$ satisfies the (Soliton) case. Then
\begin{equation}\label{eq:decaywu}
\forall t\geq 0,\ \forall x_0>1,\quad
 \int_{ {x-x(t)} <-x_0} u_x^2(t,x) dx \lesssim  \frac 1 {x_0^{3}},
\end{equation}
 \be
\label{nkovneonone}
\forall t\geq 0,\ \forall x_0>1,\quad
 \int_{x-x(t)<-x_0}u^2(t,x)dx\lesssim \frac{1}{\sqrt{x_0}}.
\ee 
\end{lemma} 

Assume Lemma \ref{cl:5}, then from \eqref{sol1}, $$\forall x_0>1, \ \ |b(t)| + \int_{y>-x_0}|\e(t,y)|^2 dy\to 0\ \ \mbox{as}\ \ t\to +\infty$$ and thus from \eqref{nkovneonone}, \fref{sol2}: $$\int|u_0|^2=\int u^2(t)= \int|Q_{b(t)}+\e(t)|^2\to \int Q^2 \ \ \mbox{as}\  \ t\to +\infty.$$ Hence $u_0$ has critical mass and a contradiction follows.

\begin{proof}[Proof of Lemma \ref{cl:5}]
Without loss of generality, by translation invariance, we assume that for all $t\geq 0$,  
\begin{equation}\label{rappels}
|\l(t)-1|\leq \frac 1{100},\quad |x_t(t)-1|\leq \frac 1{100},  \quad \hbox{and}\quad
\|\e(t)\|_{H^1}+|b(t)| \leq 1/{100}.
\end{equation}
From the decomposition of $u(t)$, there exists  $a_0>1$  such that, for   $\alpha$ small enough,  for all   $t\in [0,T)$, 
\begin{equation}\label{azero}
 \int_{x<-\frac 12 a_0} u^2(t,x+x(t)) dx 
 \leq  \int_{y<-\frac 18 a_0} \left(\varepsilon(t) + Q_{b(t)}\right)^2(y) dy 
 \leq   \frac 1{100}.
\end{equation}
Such $a_0>1$ is now fixed.

\medskip

{\bf step 1} First decay property of $u_x$ using almost  monotonicity of a localized energy. 

We claim that there exists $C>0$ such that,  
\begin{equation}\label{firstcl}
\forall t_0\geq 0,\ \forall x_0>a_0,  \quad \int_{x-x(t_0)<-x_0} u_x^2(t_0,x) dx \leq  \frac C {x_0^{2}}.
\end{equation}

\emph{Proof of \eqref{firstcl}.}
Let $\psi$ be a $C^3$ function such that for $c>0$,
\begin{align}\label{lephi} 
&\psi\equiv 1   \text{ on $(-\infty,-3]$, }  
\psi \equiv 0   \text{ on $[-\frac 12,+\infty)$, }\\
&\psi' = -\frac 12   \text{ on $[-2,-1]$, } \psi'\leq 0 \text{ on ${\mathbb R}$, } 
(\psi'')^2 \leq -c  \,\psi', \ 
(\psi')^2 \leq c \psi \text{ on ${\mathbb R}$.}\nonumber
 \end{align}
Let $x_0>a_0$. Define, for all $t>0$,
\be\label{eq:cl100}
E_{x_0}(t) = \int \left(u_x^2 - \frac 13 u^6 \right)(t,x) \psi\left(  {\widetilde x}\right) dx,
\end{equation}
where
$$
{\widetilde x} =  \frac {x-x(t)} {\xi(t)}, \quad
\xi(t) = x_0 + \frac 14 ( x(t) -x(t_0)) .
$$
First,  observe that $\lim_{t \to +\infty} E_{x_0}(t)=0$ by \eqref{Ezero}, \fref{sol2} and the Gagliardo-Nirenberg inequality.
Then, we control  the variation of $E_{x_0}(t)$ on $[t_0,+\infty)$.
By \eqref{enerkato},
\begin{align}
\frac d{dt} E_{x_0} (t) & = -   \frac 1 {\xi (t)} \int ( u_{xx}+u^5)^2 \psi'({\widetilde x})- 
\frac 2 {\xi (t)} \int u_{xx}^2 \psi'({\widetilde x})  \nonumber \\
& +   \frac {10} {\xi (t)}\int u^4 u_x^2 \psi'({\widetilde x})+ \frac 1 {\xi^{3 }(t')}\int u_x^2 \psi'''({\widetilde x})
\label{eq:cl100b}\\
& - \frac {{x}_t(t)} {\xi (t)} \int \left(u_x^2 - \frac 13 u^6\right) (1+\frac 14 \tilde x) \psi'({\widetilde x}).
\nonumber\end{align}
All the integrals above are restricted to ${\widetilde x} \in [-3,-\frac 12 ]$ since  $\psi'({\widetilde x})=0$ for $\widetilde x \not \in [-3,-\frac 12]$.
In particular, we have
$$
-\frac {({x})_t(t)} {\xi (t)} \int u_x^2 (1+\frac 14 \tilde x) \psi'({\widetilde x})
\geq - \frac 14  \frac  1 {\xi (t)} \int u_x^2  \psi'({\widetilde x}).
$$
By \eqref{rappels} and $\|u\|_{L^\infty}^4 \lesssim \|u_x\|_{L^2}^2 \|u\|_{L^2}^2
\lesssim 1$,
$$ 
\frac {10} {\xi (t)}\int u^4 u_x^2 |\psi'({\widetilde x})|\lesssim
\frac 1{\xi(t)} \|u\|_{L^{\infty}}^4 \int u_x^2 |\psi'({\widetilde x})|
\leq \frac 1{100} \frac 1 {\xi(t)} \int u_x^2 |\psi'({\widetilde x})|.
$$
Moreover, 
$$
\left|\frac 1 {\xi^{3 }(t)}\int u_x^2 \psi'''({\widetilde x})\right| \lesssim 
 \frac 1 {\xi^{3 }(t)} \int u_x^2(t') \lesssim \frac 1 {\xi^{3}(t)}.
$$
Now, we treat the $u^6$ term.
Recall   
the following standard computation (see e.g.  the proof of Lemma 6 in \cite{Mjams}),
for a $C^1$ positive function $\phi$ such that $\frac {\phi'}{\sqrt{\phi}}\lesssim 1$, for all 
$v\in H^1({\mathbb R})$,
\begin{align}
\|v^2 \sqrt{\phi}\|_{L^\infty} & \leq \sup_{x\in {\mathbb R}}\left|
\int_{-\infty}^x \left(2 v' v \sqrt{\phi} + \frac 12 v^2 \frac  {\phi'}{\sqrt{\phi}}\right) \right| \nonumber\\
& \lesssim \left(\int v^2\right)^{\frac 12} \left( \int (v')^2\phi + \int v^2 \frac {(\phi')^2}{\phi}\right)^{\frac 12}.
\label{gnpoids}
\end{align}
Using this estimate, and the fact  that $\frac {(\psi''({\widetilde x}))^2}{|\psi'({\widetilde x})|} \lesssim 1$,  we obtain:  \begin{align}\label{GNpoids} 
	\|u^2\sqrt{- \psi'({\widetilde x})}\|_{L^\infty}^2 & \lesssim \left(\int_{{\rm supp}\, \phi} u^2\right)\left( \int  u_x^2 |\psi'({\widetilde x})| +   \frac 1 {\xi^{2}(t)} \int u^2 \frac {(\psi''({\widetilde x}))^2}{|\psi'({\widetilde x})|}\right)\\
	 & \lesssim  \left(\int_{{\rm supp}\, \phi} u^2\right)\left( \int  u_x^2 |\psi'({\widetilde x})| +   \frac C {\xi^{2}(t)} \int u^2\right).\nonumber 
 \end{align}
Since $x_0>a_0$, by \eqref{azero}, we have
\begin{align*}
\int_{{\widetilde x} \in [-3,-\frac 12]} u^2(t) &\leq \int_{x<-\frac 12 x_0} u^2(t,x+{x}(t)) dx 
\leq \frac 1{100}.
\end{align*}
Thus, we get
\begin{align}
\left|\int_{{\widetilde x}\in [-3,-1/2]} u^6 \psi'({\widetilde x})\right| & \lesssim  \left( \int_{{\widetilde x}\in [-3,-1/2]} u^2\right)^2  \left( \int u_x^2 |\psi' ({\widetilde x})| + 
\frac C  {\xi^{2  } } \int u^2  \right) \nonumber \\
& \leq \frac 1{100}     \int u_x^2 |\psi' ({\widetilde x})| + 
\frac C  {\xi^{2  } } \int u^2  . \label{dim}
\end{align}
 Combining these estimates, we get
\begin{align}
&\frac d{dt} E_{x_0} (t) \gtrsim   \frac 1 {\xi(t)} \int  u_{xx}^2 (t) |\psi'({\widetilde x})|
+ \frac   {1} {\xi (t)} \int u_x^2  |\psi'({\widetilde x})|
 - C   \xi^{-3  }(t) . 
\label{estdE}\end{align}
Integrating between $t_0$ and $+\infty$, using $\lim_{t \to +\infty} E_{x_0}(t)=0$, and \eqref{rappels}, we get
\begin{align} &E_{x_0} (t_0) = \int (u_x^2 - \frac 13 u^6) (t_0) \psi\left(\frac {x-x(t_0)}{x_0}\right) dx \lesssim \frac 1{x_0^2},\label{100}\\
& 
\int_{0}^{+\infty} \left[
 \int  u_{xx}^2 (t) |\psi'({\widetilde x})|  
+     \int u_x^2  |\psi'({\widetilde x})| \right] \frac {dt} {\xi(t)}  \lesssim \frac 1{x_0^2}.\label{101}
\end{align}
 Using \eqref{gnpoids} and \eqref{azero}, we have
 \begin{align*}
& \int u^6 (t_0) \psi\left(\frac {x-x(t_0)}{x_0}\right) dx \\
& \lesssim \left( \int_{{\widetilde x}\leq -\frac 12} u^2(t_0)\right)^2  \left( \int u_x^2(t_0) \psi \left(\frac {x-x(t_0)}{x_0}\right) + 
\frac 1 {x_0^2 } \int u^2 (t_0) \right)\\ &\leq
\frac 1{100}   \int u_x^2 (t_0) \psi \left(\frac {x-x(t_0)}{x_0}\right) + 
\frac C {x_0^2 } \int u^2(t_0).
 \end{align*}
Therefore, for all $t_0\in [0,T)$, $x_0>a_0$, we have obtained
\begin{equation}\label{step1}
\int_{x-x(t_0)<-x_0} u_x^2(t_0,x) + u^6(t_0,x) dx \lesssim \frac 1 {x_0^2}.
\end{equation}
Since $\psi'(\tilde x)=0$ for $\tilde x<-3$ and $\tilde x>-\frac 12$, using \eqref{101}, we have 
$$\int_{0}^{+ \infty} \int u_{xx}^2(t') \psi({\widetilde x}) dt'<\infty.$$
Moreover,
\begin{align*}
 \frac {d}{dx_0} \left(\int_{0}^{+\infty} \int u_{xx}^2(t ) \psi({\widetilde x}) dt \right)
& = \int_{0}^{+\infty} \int u_{xx}^2 (t ) \frac {-\tilde x}{\xi(t )} \psi'({\widetilde x}) dt \\
& \lesssim  \int_{0}^{+\infty}   \frac 1{\xi(t )} \int  u_{xx}^2 (t ) \psi'({\widetilde x}) dt \lesssim \frac 1 {x_0^2}.
\end{align*}
Integrating in $x_0$, we get $\int_{0}^{+\infty} \int u_{xx}^2 \psi(\tilde x) dt' \leq \frac C{x_0}$ and arguing in a similar way for $u_x$, we obtain the following
\begin{equation}\label{step1b}
\int_{0}^{+\infty} \int \left[  u_{xx}^2(t ) \psi(\tilde x)   
+     u_x^2(t ) \psi(\tilde x) \right] dt \leq \frac 1{x_0}.
\end{equation}
\medskip

{\bf step 2} Refined decay property of $u_x$.

We claim the improved decay:
\begin{equation}\label{secondcl}
\forall x_0>2 a_0,\quad \int_{x<-x_0+x(t_0)} u_x^2(t_0,x ) dx \lesssim \frac 1 {x_0^{3}}.
\end{equation}
To obtain this improved estimate, we introduce
$$
G_{x_0}(t) = \int u_x^2 (t)   \psi\left({\widetilde x} \right).
$$
By direct computations
\begin{align*}
\frac d{dt} G_{x_0}(t) & =  - \frac 3{\xi(t)} \int u_{xx}^2  \psi' ({\widetilde x}) -
 \frac{ {x}_t(t)}{\xi(t)} \int u_x^2 \left(1+\frac 14 {\widetilde x}\right) \psi' ({\widetilde x}) + \frac{1}{\xi^3(t)}\int u_x^2 \psi'''({\widetilde x})\\
& - 20 \int u_x^3 u^3  \psi({\widetilde x}) + \frac 5{\xi(t)} \int u_x^2 u^4 \psi' ({\widetilde x}).
\end{align*}
The second and the last terms  in the right hand side are treated as before. 
For the third term, we use \eqref{step1} and $\psi'''=0$ for ${\widetilde x} \geq -\frac 12$, which gives
$$
\frac 1{\xi^3(t)} \int u_x^2(t) \psi'''({\widetilde x}) \lesssim \frac 1{\xi^3(t)} \int_{x \leq - \frac 12 \xi(t)}
u_x^2 \lesssim \frac 1{\xi^5(t)} \lesssim  \frac {\xi_t(t)} {\xi^5(t)}.
$$

Finally, the term $\int u_x^3 u^3 \psi({\widetilde x})$ is controlled as follows, using \eqref{gnpoids} with $\phi=\psi({\widetilde x})$
\begin{align*}
&\left| \int u_x^3 u^3 \psi(\tilde x)\right|   \leq
\|u_x^2 \sqrt{\psi(\tilde x)}\|_{L^\infty} \int |u_x u^3 \sqrt{\psi({\widetilde x})}|
\\
& \leq \|u_x^2 \sqrt{\psi(\tilde x)}\|_{L^\infty} \left( \int u_x^2 \psi({\widetilde x})\right)^{\frac 12}
\left( \int_{x <-\frac 12 x_0+{x}} u^6\right)^{\frac 12} \\
&\lesssim   \left(\int_{x<-\frac 12x_0+{x}} u_x^2\right)^{\frac 12} \left( \left(\int u_{xx}^2 \psi({\widetilde x})\right)^{\frac 12}
+\left(  \frac 1{\xi^2} \int u_{x}^2 \psi({\widetilde x})\right)^{\frac 12}\right) \\ & \qquad \times \left( \int u_x^2 \psi({\widetilde x})\right)^{\frac 12} 
\left( \int_{x <-\frac 12x_0+{x}} u^6\right)^{\frac 12}\\
&\lesssim  \left(\int_{x<-\frac 12x_0+{x}} u_x^2\right)^{\frac 12} \left( \int_{x <-\frac 12 x_0+{x}} u^6\right)^{\frac 12}
\left( \int u_{xx}^2 \psi({\widetilde x})+ \int u_{x}^2 \psi({\widetilde x})\right) \\
& \lesssim \frac 1 {x_0^2} \left( \int u_{xx}^2 \psi({\widetilde x})+ \int u_{x}^2 \psi({\widetilde x})\right).
\end{align*}

In conclusion of these estimates, we have obtained
$$
\frac d{dt} G_{x_0}(t) \gtrsim -  \frac {\xi_t}{\xi^5} - \frac 1{x_0^2} \left(\int (u_{xx}^2 +u_x^2)
\psi(\tilde x)\right).
$$
Therefore, by integration on $[t_0,+\infty)$, using \eqref{step1b} and $\lim_{t\to +\infty} G_{x_0}(t)=0$, we obtain
$G_{x_0}(t_0) \lesssim 1/x_0^3 $, which proves \eqref{eq:decaywu}.

 \medskip
 {\bf step 3.} $L^2$ estimate.
 
We deduce from \fref{eq:decaywu} some $L^2$ tightness for $u$. Indeed,  for $x_0> 1$:
\bee
\|u(t,\cdot)\|^2_{L^{\infty}(x-x(t)<-x_0)} & \lesssim &   \int_{x-x(t)<-x_0} |u_x u| dx\lesssim   \left(\int_{x-x(t)<-x_0} u_x^2\right)^{\frac 12} \left(\int u^2\right)^{\frac 12}\\
&  \lesssim  &  \frac 1 { x_0^{3/2}}
\eee
from which:
\bee
\int_{ {x-x(t)} \leq -x_0}|u(t,x)|^2dx  
& \lesssim & \int_{y>x_0}\frac{dy}{|y|^{\frac32}}\lesssim \frac{1}{\sqrt{x_0}},
\eee
and \eqref{nkovneonone} follows.
\end{proof}

%%%%%%%%%%%%%%%%%%%%%%%%%%%%%%%%%%

\section{Sharp description of the blow up regime}

%%%%%%%%%%%%%%%%%%%%%%%%%%%%%%%%%%%%%%%%%%%%%%%%

We now finish the proof of Theorem \ref{th:1} by proving  \eqref{th1.1} and \eqref{th:1:4} in the framework of a blow up solution in $\mathcal{T}_{\alpha^*}$. 
We   further use $L^2$ and $H^1$ monotonicity properties {\it away} from the soliton to propagate the dispersive information in larger regions to the left than the  norm $\mathcal N_i$ controlled by Proposition \ref{PR:4.1}, and this will yield the sharp behavior \fref{th:1:4}.\\

We let:
$$
\widetilde u(t,x) = u(t,x)-\frac{1}{\lambda^{\frac12}(t)}Q_{b(t)}\left(\frac{x-x(t)}{\lambda(t)}\right).
$$

\begin{proposition}[Improved dispersive bounds away from the soliton]
\label{propaway}
Let $u_0 \in \mathcal{A}$ such that $u(t)$ blows up in finite time $T$ and: $$\forall t\in [0,T), \ \ u(t)\in \matchal T_{\alpha^*}.$$ Then, the following holds : 

\noindent{\em (i) $H^1$ estimates around the soliton}:
  \begin{align}
& \sup_{R>1}\sup_{[T- \frac 1{\ell_0^2 R}, T)} R^2 \int_{x>R} {\widetilde u}^2(t,x) dx < \infty,\label{l2R}\\
& \lim_{R\to +\infty} \sup_{[T- \frac 1{\ell_0^2 R}, T)}   \int_{x>R} {\widetilde u}_x^2(t,x) dx=0,\label{h1R}\\
& \lim_{t\to T} \frac  1 {(T-t)^2} \int_{x - x(t) \geq -\frac {x(t)}{\log(T-t)}} {\widetilde u}^2 (t,x) dx =0.
\label{eq:108}
\end{align}
\noindent{\em (ii) Existence and asymptotic of the dispersed remainder}: there exists $u^*\in H^1$ such that 
\be
\label{cneocneone}
\widetilde u\to u^*\ \ \mbox{in}\ \ L^2\ \ \mbox{as}\ \ t\to T,
\ee
and
\begin{equation}
\label{th:1:4bis}
\int_{x>R} ({u^\star})^2(x) dx \sim \frac {\|Q\|_{L^1}^2}{8 \ell_0 R^2} \ \ \mbox{as} \ \ R\to +\infty.
\end{equation}
\end{proposition}

The rest of this section is devoted to the proof of Proposition \ref{propaway}.

%%%%%%%%%%%%%%%%%%%%%%%%%%%%%%%%%%

\subsection{$H^1$ monotonicity away from the soliton}

%%%%%%%%%%%%%%%%%%%%%%%%%%%%%%%%%%%%%%%%%%%%%%%%

We aim at refining the dispersive estimate \fref{estfondamentalebis} by propagating it to the left of the solitary wave, since $\mathcal N_2$ involves an exponentially well localized norm at the left of the soliton. For this, we use $H^1$ monotonicity tools in the spirit of \cite{Mjams}, \cite{MMannals}. 

\begin{lemma}[Monotonicity away from the soliton core]
\label{lemmaimprovedaway}
There exist $a_0\ll 1$, $0<\delta_0\ll 1$  universal constants such that the following holds. Let $0\leq t_0<T$ close enough to $T$ and $0<\nu<\frac{1}{10}$ satisfying:
\be
\label{conditionnubis}
\frac {\lambda^2(t_0)}{\nu}  <\delta_0.
\ee
Let 
$$  \phi(x)= \frac 2\pi \arctan\left(\exp\left(\frac {\sqrt{\nu}}{5} x\right)\right)
$$ 
so that 
\be
\label{defphi}
\lim_{+\infty} \phi=1,\
 \lim_{-\infty} \phi=0,\ \  \phi'''(x) \leq \frac {\nu}{25} \phi'(x), \quad
|\phi''(x)|\lesssim \sqrt{\nu} \phi'(x),\ \ \forall x\in \RR.\ee
Then: $\forall y_0>a_0$, $\forall t_0\le t<T$, there holds the $L^2$ monotonicity bound:
\bea
\label{l2wuOK}
\nonumber &&\int {\widetilde u}^2(t,x) \phi\left( \frac  {x-x({t_0})}{{\lambda}({t_0})}  -\nu \frac {t-{t_0} } {{\lambda}^3({t_0})} +{y_0} \right) dx+ 2 (b(t)-b(t_0))(P,Q)\\
& \lesssim &   \int \widetilde u^2 (t_0,x) \phi\left( \frac {x-x(t_0)}{\lambda(t_0)} + y_0\right) dx 
+   \frac 1{\sqrt{\nu}} e^{-\frac {\sqrt{\nu}}{10} {y_0}}  +  \lambda^{2+\frac 14}(t_0)
\eea
and the $H^1$ monotonicity bound:
\bea
\label{h1wuOK}
& &\int \left(\widetilde u_x^2 - \frac 13{\widetilde u}^6\right)(t,x) \phi\left( \frac 54 \left(\frac  {x-x({t_0})}{{\lambda}({t_0})}  -\nu \frac {t-{t_0} } {{\lambda}^3({t_0})} +{y_0} \right) \right)dx\\
\nonumber &  -&  2  \left( \frac {b(t)}{{\lambda}^2(t)} - \frac {b(t_0)}{{\lambda}^2(t_0)}\right) (P,Q)\\
\nonumber  & \lesssim &  \int \left(\widetilde u_x^2(t_0,x) +  \frac {\widetilde u^2 (t_0,x)}{\l^2(t_0)}  \right) \phi\left( \frac {x-x(t_0)}{\lambda(t_0)} + y_0\right) dx 
+ \frac 1{\sqrt{\nu}} \frac {e^{-\frac {\sqrt{\nu}}{10} {y_0}}}{\lambda^2(t_0)}  +  \lambda^{\frac 14 }(t_0).
\eea
\end{lemma}

The proof of Lemma \ref{lemmaimprovedaway} is postponed to Appendix \ref{appendixa}.

 %%%%%%%%%%%%%%%%%%%%%%%%%%%%%%%% %%%%%%%%%%%%%%%%%%%%%%%%%%%%%%%%
 %%%%%%%%%%%%%%%%%%%%%%%%%%%%%%%% %%%%%%%%%%%%%%%%%%%%%%%%%%%%%%%%
 
\subsection{Proof of Proposition \ref{propaway}}\label{s:5}

 %%%%%%%%%%%%%%%%%%%%%%%%%%%%%%%% %%%%%%%%%%%%%%%%%%%%%%%%%%%%%%%%
 %%%%%%%%%%%%%%%%%%%%%%%%%%%%%%%% %%%%%%%%%%%%%%%%%%%%%%%%%%%%%%%%
  \quad 
\medskip

{\bf step 1} Proof of \eqref{l2R}.  
The estimate \eqref{l2R} is a direct consequence of \fref{l2wuOK} and the space time control of local terms \eqref{estfondamentalebis} which implies:
\be\label{vneonveovnoneo}
\frac{\mathcal N_i(t_2)}{\lambda^2(t_2)}+\int_{t_1}^{t_2}\frac{\int (\e_y^2 + \e^2)(s)\varphi_{i,B}'}{\lambda^5(\tau)}d\tau\lesssim\delta(\alpha_0).
\ee
Indeed, fix $\nu=\frac{1}{16}$ in Lemma \ref{lemmaimprovedaway} (note that $B>40=10/\sqrt{\sigma}$), then \fref{conditionnubis} is satisfied from the blow up assumption for $t$ close enough to $T$, and we estimate the RHS of \eqref{l2wuOK}:
\bea
\label{versloc}
\nonumber&&\int \widetilde u^2 (t_0,x) \phi\left( \frac {x-x(t_0)}{\lambda(t_0)} + y_0\right) dx  =  \int \varepsilon^2({t_0}) \phi(y+{y_0})\\
 \nonumber & \lesssim & 
\int_{y<-{y_0}} \varepsilon^2({t_0}) e^{\frac {\sqrt{\nu}}{10} (y+{y_0})}
+\int_{y>-{y_0}} \varepsilon^2({t_0}) \nonumber\\
\nonumber & \lesssim & \int_{y<-{y_0}} \varepsilon^2({t_0}) e^{\frac 1B (y+{y_0})} +e^{\frac {{y_0}}B} \int_{ -{y_0}<y<0} \varepsilon^2({t_0}) e^{\frac yB}+ \int_{y>0} \varepsilon^2({t_0})\nonumber \\
& \lesssim&  e^{\frac {{y_0}}B} {\mathcal{N}_{1,\rm loc}}({t_0}).
\eea
Let then $R\gg1 $ large enough and ${t_R}$ be such that $x({t_R})= R$, so that 
\be
\label{cnoeoebusbisb}
T-{t_R}  = \frac 1{ \ell_0^2 R}(1+o_R(1))=\frac{\lambda(t_R)}{\ell_0}(1+o_R(1)) \ \ \hbox{as $R\to +\infty$}.
\ee
We now make an essential use of the fact that the space time estimate \eqref{vneonveovnoneo} is better for local $L^2$ terms  than the pointwise bound given by (H2). Indeed, the law \fref{det0} and \fref{cnoeoebusbisb} ensure, for $R$ large: $$\forall \tau \in [t_R-{(R\ell_0)^{-\frac 52}},t_R], \ \ \lambda(\tau)=\ell_0\left[T-t_R+o_R(1)\right]\geq \frac12 \lambda(t_R)$$ 
and thus \eqref{vneonveovnoneo} implies:
$$(R\ell_0)^5 \int_{{t_R} -  {(R\ell_0)^{-\frac 52}}}^{{t_R}} 
\int (\e_y^2 + \e^2)(t)\varphi_{1,B}'   dt \lesssim \int_0^T \frac{\int (\e_y^2 + \e^2)(t)\varphi_{1,B}' }{\lambda^5(t)}dt\lesssim \delta(\alpha_0).
$$
Thus, there exists ${\bar t_R}\in [{t_R} - {(R\ell_0)^{-\frac 52}}, {t_R} ]$ such that
\be
\label{meilleuteemps}
\int (\e_y^2 + \e^2)({\bar t_R})\varphi_{1,B}' \lesssim \delta(\alpha_0)(\ell_0 R )^{-\frac 52 }\sim \delta(\alpha_0)(\lambda(\bar t_R))^{\frac52}
\ee
which is a strict gain on the pointwise bound (H2).
Note also the relations: 
\be
\label{cneoohoe}
b({\bar t_R}) = \ell_0^3 (T-{\bar t_R})^2 (1+o_R(1))= \frac 1{\ell_0 R^2} (1+o_R(1)), \ \ x({\bar t_R}) = R + o_R(1).
\ee
We now apply  \eqref{l2wuOK} to $u(t)$ with $$\nu=\frac{1}{16},  \ \ t_0={\bar t_R}, \ \ y_0=y_R= 40 \log(\ell_0R^3).$$ 
 We obtain from \eqref{versloc}, \fref{meilleuteemps}, \fref{cneoohoe} and $B\gg1$: $\forall t\in [{\bar t_R},T)$,
\bea
\label{inieieuh}
\nonumber & & \int {\widetilde u}^2(t,x) \phi\left( \frac {x-x({\bar t_R}) }{{\lambda}({\bar t_R})}- \frac 1{16}\frac {t-{\bar t_R}}{{\lambda}^3({\bar t_R})}+y_R\right)-2 b({\bar t_R})(P,Q)\\
& \lesssim &  e^{\frac {y_R}{B}} {\mathcal{N}_{1,\rm loc}}({\bar t_R}) + e^{-\frac 1{40} y_R} +   (T-t_R)^{2 +\frac 14}=o_R\left(\frac{1}{R^2}\right).
\eea
and $$2 b({\bar t_R})(P,Q)=\frac {\|Q\|_{L^2}^2}{8 \ell_0}\frac 1 {R^2} (1+o_R(1)).$$ Moreover, we estimate using \fref{cneoohoe}: $\forall x>2R$, $\forall t\geq \bar t_R$: 
\bee
\frac {x-x({\bar t_R}) }{{\lambda}({\bar t_R})}- \frac 1{16}\frac {t-{\bar t_R}}{{\lambda}^3({\bar t_R})}+y_R \geq  \frac{2R-R}{\lambda(t_R)}- \frac 1{16}\frac {t-{\bar t_R}}{{\lambda}^3({\bar t_R})}\geq \frac{1}{\ell_0\lambda^2(t_R)}>0.
\eee
Thus, from \eqref{cneoohoe}, \eqref{inieieuh}, and also using
$\phi(y)\geq \frac12$ for $y>0$, we obtain
$$
\forall t\in \left[ T- \frac 1{2 \ell^2_0 R}, T\right),\quad
\int_{x>2R} {\widetilde u}^2(t,x) dx \lesssim \frac 1{\ell_0R^2},
$$
and \eqref{l2R} follows.\\

{\bf step 2} Proof of \eqref{h1R}.
We now apply \eqref{h1wuOK} to $u(t)$ with the same choice as before $$\nu=\frac{1}{16}, \ \ t_0={\bar t_R}, \ \ y_0=y_R=  40 \log(\ell_0 R^3).$$
%  as indeed again \fref{conditionnubis} is satisfied for $\alpha<\alpha^*$ small enough.

We estimate like for the proof of \fref{versloc} and using \eqref{meilleuteemps}
\bee
\int \widetilde u_x^2(\bar t_R,x)\phi\left( \frac {x-x(\bar t_R)}{\lambda(\bar t_R)} + y_R\right) dx \lesssim e^{\frac{y_R}{B}}\frac{\int \e_y^2({\bar t_R})\varphi_{1,B}'}{\lambda^2(\bar t_R)}=o_R(1).
\eee
Using \eqref{versloc}, we obtain for all $t\in [{\bar t_R},T)$,
\bee
& & \int \left({\widetilde u}_x^2(t,x)-\frac 13 {\widetilde u}^6(t,x)\right) \phi\left(\frac 54 \left(\frac {x-x({\bar t_R}) }{{\lambda}({\bar t_R})}- \frac 1{16} \frac {t-{\bar t_R}}{{\lambda}^3({\bar t_R})}+y_R\right) \right)dx\nonumber \\
& \lesssim  & \left|\frac {b(t)}{{\lambda}^2(t)} - \frac {b({\bar t_R})}{{\lambda}^2({\bar t_R})}\right| +  R^2 e^{-\frac 1{40} y_R} +o_R(1)\\
& = & o_R(1)
\eee
where we used $\lim_{t\to T} \frac {b(t)}{{\lambda}^2(t)}=\ell_0$ in the last step. Observe now the bound from Sobolev, \fref{boundhoneglobal} and \eqref{inieieuh}:
\begin{align*}
&\int  {\widetilde u}^6(t,x)  \phi\left(\frac 54 \left(\frac {x-x({\bar t_R}) }{{\lambda}({\bar t_R})}- \frac 1{16} \frac {t-{\bar t_R}}{{\lambda}^3({\bar t_R})}+y_R\right) \right)dx\\
&\leq C \|{\widetilde u}\|_{L^\infty}^4 \int  {\widetilde u}^2(t,x)  \phi\left( \frac {x-x({\bar t_R}) }{{\lambda}({\bar t_R})}- \frac 1{16} \frac {t-{\bar t_R}}{{\lambda}^3({\bar t_R})}+y_R\right) dx\lesssim \frac{1}{R^2},
\end{align*}
and \eqref{h1R} follows.\\

{\bf step 3} Proof of \eqref{eq:108}.\\

Let $t$ be close to $T$. The space time estimate \eqref{vneonveovnoneo} and \fref{det0} ensure:
$$\frac{1}{\lambda^5(t)}\int_{ t -\frac{20(T-t)}{|\log(T-t)|}}^{ t- \frac{10 (T-t)}{|\log(T-t)|}}{\mathcal N}_{1,\rm loc}(\tau)d\tau\lesssim \int_0^T\frac{{\mathcal N}_{1,\rm loc}(\tau)}{\lambda^5(\tau)}d\tau\lesssim \delta(\alpha),$$ and thus there exists $$\bar t\in  
\left[ t -\frac{20 (T-t) }{ |\log(T-t)|}, t- \frac{10 (T-t)}{ |\log(T-t)|}\right]$$ such that
$${\mathcal N}_{1,\rm loc}(\bar t) \leq \ell_0^5 (T-t)^4 |\log(T-t)|.$$
Moreover, from \eqref{det0}, 
\be
\label{cenoenooe}
x(t)- x(\bar t) \geq (t-\bar t) \min_{[t,\bar t]} x_t 
\geq \frac {9}{ \ell_0^2  (T-t)| \log(T-t)|} \geq   \frac {8 x(t)}{|\log(T-t)|},
\ee
$$b(t)-b(\bar t)= o\left[(T-t)^2\right]\ \ \mbox{as}\ \ t\to T.$$ 
We now apply \eqref{l2wuOK} with: $$\nu=\frac{1}{16}, \ \ y_0= \bar y = 40 |\log (T-t)|, \ \ t_0=\bar t.$$ The RHS of\eqref{l2wuOK} is estimated using \eqref{versloc} and we obtain:
$$\int \widetilde u^2(t,x) \phi\left( \frac {x-x(\bar t)}{\lambda(\bar t)} - \frac 1{10} \frac {t-\bar t}{\lambda^3(\bar t)} + \bar y\right) dx
= o\left[(T-t)^2\right]\quad \hbox{as $t\to T$.}
$$
Moreover, let $x$ such that $$x-x(t)\geq -\frac{x(t)}{|\log (T-t)|},$$ then from \fref{cenoenooe}, \eqref{det0},
\bee
&& \frac {x-x(\bar t)}{\lambda(\bar t)} - \frac 1{10} \frac {t-\bar t}{\lambda^3(\bar t)}\geq \frac{1}{\lambda(\bar t)|\log (T-t)|}\left[8x(t)-\frac{1}{10}\frac{10(T-t)}{\lambda^2(t)}\right]>0,
\eee 
 and then $\phi(y)\geq  \frac 12 $ for $y>0$ yields \eqref{eq:108}.

\begin{remark} Observe that \eqref{l2R} and \eqref{h1R} imply:
\begin{align*}
\forall R>1,\quad   &\int_{x>R} {\widetilde u}^2\left(T- \frac 1{200 \ell_0^2 R},x\right) dx \lesssim \frac 1{R^2}, \\&
\lim_{R\to +\infty}    \int_{x>R} {\widetilde u}_x^2\left(T- \frac 1{200\ell_0^2 R},x\right) dx=0.
\end{align*}
In particular, given $t$ close enough to $T$, we chose $R=(200 \ell_0^2 (T-t))^{-1}< \frac 1{100} x(t)$ and conclude:
\be
\label{det1}
\int_{x>\frac{x(t)}{100}} \ut^2(t,x)dx\lesssim (T-t)^2, \ \ \lim_{t\to T}\int_{x>\frac{x(t)}{100}} \ut_x^2(t,x)dx=0.
\ee
\end{remark}

{\bf step 4} $L^2$ tightness.\\

First observe from direct check using \fref{det0} that $$\frac{1}{\lambda^{\frac12}(t)}Q_{b(t)}\left(\frac{x-x(t)}{\lambda(t)}\right)-\frac{1}{\lambda^{\frac12}(t)}Q\left(\frac{x-x(t)}{\lambda(t)}\right)\to 0\ \ \mbox{in}\ \ L^2\ \ \mbox{as}\ \ t\to T,$$ and hence \fref{cneocneone} is equivalent to showing the existence of a strong limit 
\be
\label{cneocnone}
\ut(t)\to u^*\ \ \mbox{in} \ \ L^2 \ \ \mbox{as}\ \ t\to T.
\ee
We first claim that the sequence is tight: $\forall \epsilon>0$,  $\exists A_{\epsilon}>1$ such that for all $t\in [0,T)$, 
\begin{equation}\label{cd0}
\int_{|x|>A(\e)} {\widetilde u}^2(t,x) dx<\epsilon.
\end{equation}
On the right $x>A$ where non linear interactions take place, the claim directly follows from \eqref{l2R}.  On the left, this is a simple linear claim which follows from the finitness of the time interval $[0,T)$, the $H^1$ bound \fref{boundhoneglobal} and a Kato $L^2$ localization argument. Indeed, let $t_\epsilon$ be close enough to $T$ such that
\begin{equation}\label{cd2}
\int_{t_\e}^T \int_{x<0}  ( u_x^2 +   u^2 ) dx dt < \epsilon.
\end{equation}
Let $\psi$ be a $C^3$ function such that  
\begin{equation}\label{lephi3} 
 \psi\equiv 1   \text{ on $(-\infty,-2]$}, \quad   
\psi \equiv 0   \text{ on $[-1+\infty)$,}
\quad \psi' \leq 0 \text{ on $\mathbb{R}$}.   \end{equation}
Pick $A_\epsilon>1$ large enough so that 
$ 
\int  u^2(t_\epsilon) \psi(x+A) \leq \epsilon,
$ 
then by \eqref{L2kato}, 
\begin{align*}
&\frac d{dt} \int u^2(t) \psi(x+A) \\& = -  3  \int u_x^2(t) \psi'(x+A)  +  \int u^2(t) \psi'''(x+A)+ \frac 5{3} \int u^6(t) \psi'(x+A),
\end{align*}
and thus from \eqref{cd2}: $\forall t \in [t_\epsilon,T)$,
$$
\left| \int u^2(t) \psi\left(x+A\right) - \int u^2(t_\epsilon) \psi\left(x+A\right) \right| \leq C \epsilon.
$$
and \fref{cd0} follows. Now the uniform $H^1$ bound \fref{boundhoneglobal} ensures that for all sequence $t_n\to T$, there exists a subsequence $t_{\phi(n)}\to T$ and ${u^\star}\in H^1$ such that ${\widetilde u}(t_{\phi(n)}) \rightharpoonup {u^\star}$ in $H^1$ weak and ${\widetilde u}(t_{\phi(n)}) \to {u^\star}$ $L^2$ strong from \fref{cd0} and the local compactness of the Sobolev embedding. By a weak convergence argument,  the limit ${u^\star}$ does not depend on the sequence $(t_n)$. Indeed, let $\theta$ be a $C^\infty$ function with support in $[-K,K]$, then
$$
\left|\frac d{dt} \int u \theta \right|= \left|\int u^5 \theta_x + \int u \theta_{xxx}\right|
\leq C_{\theta} \int_{-K}^{K} \left(|u|^5  + |u|\right) \leq C_{\theta,K},
$$ 
and thus $\int u(t)\theta$ has a limit as $t\to T$, and \fref{cneocnone} follows. Note that the regularity $u^*\in H^1$ follows from \fref{cneocnone}, \fref{boundhoneglobal}.\\

{\bf step 5}  Universal behavior of $u^*$ on the singularity.\\

We now turn to the proof of the universal behavior of $u^*$ \eqref{th:1:4bis} on the singularity which follows from lower and upper bounds.

{\em (i) Upper bound}: Let $R\gg1$ large enough.
Let ${t_R}$ be such that $$x({t_R})= R,$$ so that from \fref{det0}: $$  \frac{\lambda(t_R)}{\ell_0}=(T-{t_R})(1+o_R(1))  = \frac 1{ \ell_0^2 R}(1+o_R(1)),$$ $$ b({ t_R}) = \ell_0^3 (T-{ t_R})^2 (1+o_R(1))= \frac 1{\ell_0 R^2} (1+o_R(1)).$$
We apply  \eqref{l2wuOK} to $u(t)$ with 
$$\nu = \nu_R=\frac {1}{\log^2 R}, \ \ y_0=y_R= 10 \log^2(R^3),\ \ t_0={  t_R}$$ which satisfy the condition \fref{conditionnubis} for $R$ large enough, and obtain: $\forall t\in [{  t_R},T)$,
\begin{align*}
& \int {\widetilde u}^2(t,x) \phi\left(\frac {x-x({  t_R}) }{{\lambda}({  t_R})}- \nu_R \frac {t-{  t_R}}{{\lambda}^3({  t_R})}+y_R\right) dx-2 b({  t_R}) \int PQ\\
& \lesssim  \int {\widetilde u}^2(t_R,x) \phi\left( \frac {x-x({  t_R})}{\lambda({  t_R})} +y_R\right) dx + \frac{1}{\nu_R}e^{-\frac {\sqrt{\nu_R}}{10} y_R} +(T-t_R)^{2+\frac 14}\\ 
& \lesssim \int {\widetilde u}^2(   t_R,x) \phi\left( \frac {x-x({  t_R})}{\lambda({  t_R})} +y_R\right) dx +o\left(\frac{1}{R^2}\right).
\end{align*}
Note that $$\frac{-x(T_R)}{|\log (T-t_R)|}=-\frac{R}{\log R}(1+o_R(1))\ll \lambda(t_R)y_R$$
so that by \eqref{eq:108} :
\begin{align*}
& \int {\widetilde u}^2(   t_R) \phi\left( \frac {x-x({  t_R})}{\lambda({  t_R} )}+y_R\right) dx\\
& \lesssim e^{ - \frac {\sqrt{\nu_R}}{10} y_R} \int \widetilde u^2 (t_R,x)dx+ \int_{x-x(t_R) > - 2\lambda({  t_R}) y_R} {\widetilde u}^2(   t_R,x) dx
=\frac 1 {R^2} o_R(1).
\end{align*}
We thus conclude from \eqref{PQ}:
\begin{align*}
& \int {\widetilde u}^2(t,x) \phi\left(\frac {x-x({  t_R}) }{{\lambda}({  t_R})}- \nu_R \frac {t-{  t_R}}{{\lambda}^3({  t_R})}+y_R\right) dx\\
& \leq \frac {2\int PQ}{\ell_0 R^2} (1+o_R(1))= \frac {\|Q\|_{L^1}^2}{8 \ell_0R^2} (1+o_R(1)).
\end{align*}
Passing to the limit $t\to T$, we find
$$
R^2 \int ({u^\star})^2(x) \phi\left(\frac {x-x({  t_R})}{{\lambda}({  t_R})} - \nu_R\frac {T-{  t_R}}{{\lambda}^3({  t_R})}+y_R\right) dx\leq
\frac {\|Q\|_{L^1}^2}{8 \ell_0}  (1+o_R(1)).
$$
Using $x({  t_R})= R$ and ${\lambda}({  t_R}) = \frac 1{R\ell_0} (1+o_R(1))$, and passing to the limit $R\to +\infty$ yields:
$$
\limsup_{R\to +\infty} R^2 \int_{x>(1+\nu_R) R} ({u^\star})^2(x) dx \leq \frac {\|Q\|_{L^1}^2}{8 \ell_0},
$$
which now easily implies:\begin{equation}\label{upper}
\limsup_{R\to +\infty} R^2 \int_{x>  R} ({u^\star})^2(x) dx \leq   \frac {\|Q\|_{L^1}^2}{8 \ell_0} .
\end{equation}

{\em (ii) Lower bound}: Let a smooth cut off function:
$$
\omega\equiv 0 \hbox{ on $(-\infty,-1]$,} \quad
\omega\equiv 1 \hbox{ on $[0, +\infty)$,}\quad
\omega' \geq 0 \hbox{ on ${\mathbb R}$}.
$$
Let $0<\nu< \frac 1{10}$ be arbitrary and
let $\omega_\nu$ be defined by  $\omega_\nu(x) = \omega  ( \frac x{\nu} )$.
For $R>1$ large, we define ${t_R}$ such as $x({t_R})=R$ as before. Using the identity \eqref{L2kato}, we have, for all ${  t_R}\leq t<T$, 
\begin{align*}	& \frac d{dt} \int u^2\omega_\nu\left( \frac {x-R+4 \log R}{R}\right) \\& \geq 
- \frac 3 R \int u_x^2 \omega_\nu'\left( \frac {x-R+4 \log R}{R}\right)
+ \frac 1 R \int u^2 \omega_\nu'''\left( \frac {x-R+4 \log R}{R}\right)
\\& \geq 
- \frac {C_\nu}{R} \int_{(1-\nu) R<x+4\log R<  R} u_x^2  - \frac {C_\nu}{R^3} \int_{(1-\nu) R<x+4\log R<   R} u^2.  
\end{align*}
By \eqref{h1R} and the properties of $Q_b$, (see in particular \eqref{eq:210} and \eqref{eq:002}), we have
$$\sup_{t\in [t_R,T)} \int_{(1-\nu)R<x+4\log R< R} {  u}^2_x(t,x)dx =  o_R(1) \quad \hbox{as $R\to +\infty$.}
$$
Since $T-t_R \lesssim\frac 1{\ell_0^2 R}$, 
we obtain by integrating on $[t_R,t]$: $\forall t\in[ t_R,T)$,
\be
\label{cmeepeoje}
\int u^2(t)\omega_\nu\left( \frac {x-R+4 \log R}{R}\right) \geq \int u^2({  t_R}) \omega_\nu\left( \frac {x-R+4 \log R}{R}\right)  +o_R\left(\frac{1}{R^2}\right).
\ee
We now develop $u$ in terms of $Q_b$ and ${\widetilde u}$. On the one hand, a simple computation ensures:
\bee
&&\int u^2(t)\omega_\nu\left( \frac {x-R+4 \log R}{R}\right)\\ &&=  \int Q^2+\int \ut^2(t)\omega_\nu\left( \frac {x-R+4 \log R}{R}\right)+o_{t\to T}(1)\\
&  &  \to \int Q^2+\int (u^*)^2(t)\omega_\nu\left( \frac {x-R+4 \log R}{R}\right)\ \ \mbox{as}\ \ t\to T.
\eee
Next,
\bee
&&\int u^2({  t_R}) \omega_\nu\left( \frac {x-R+4 \log R}{R}\right)\\
& = & \int Q^2+2( P,Q) b(t_R)+\int \ut^2(t_R)\omega_\nu\left( \frac {x-R+4 \log R}{R}\right)+o_R\left(\frac{1}{R^2}\right)\\
& \geq& \int Q^2+\frac 1{R^2}\left(\frac {\|Q\|_{L^1}^2}{8 \ell_0}+o_R(1)\right)
\eee
where we used \fref{th:1:4bis} to treat the crossed term. We therefore conclude from \fref{cmeepeoje}:
$$\liminf_{R\to +\infty} R^2 \int_{x>  (1-\nu)R-4\log R} ({u^\star})^2(x) dx \geq \frac {\|Q\|_{L^1}^2}{8 \ell_0}$$ and since $\nu$ is arbitrary,
$$\liminf_{R\to +\infty} R^2 \int_{x> R} ({u^\star})^2(x) dx \geq  \frac {\|Q\|_{L^1}^2}{8 \ell_0} .
$$
This concludes the proof of Proposition \ref{propaway}.
 
%%%%%%%%%%%%%%%%%%%%%%%%%%%%%%%%%%%%%%%%%%%%%%%%%%%%%%
%%%%%%%%%%%%%%%%%%%%%%%%%%%%%%%%%%%%%%%%%%%%%%%%%%%%%%

\appendix

%%%%%%%%%%%%%%%%%%%%%%%%%%%%%%%%%%%%%%%%%%%%%%%%%%%%%%
%%%%%%%%%%%%%%%%%%%%%%%%%%%%%%%%%%%%%%%%%%%%%%%%%%%%%%

%%%%%%%%%%%%%%%%%%%%%%%%%%%%%%%%%%%%%%%%%%%%%%%%%%%%%%

\section{}
\label{appendixa}
%%%%%%%%%%%%%%%%%%%%%%%%%%%%%%%%%%%%
%%%%%%%%%%%%%%%%%%%%%%%%%%%%%%%%%%%%

\subsection{Proof of Lemma \ref{lemmaimprovedaway}}

Let $a_0\gg 1$, $0<\delta_0\ll 1$  two constants to be chosen.

For $t_0\in [0,T)$, we  consider the renormalized solution 
\be
\label{cnekocneoneo}
{z}(t',x') = {\lambda}^{\frac 12}({t_0}) u({\lambda}^3({t_0})  t' + {t_0} , {\lambda}({t_0})  x'+x({t_0})), \ \ t'\in[0,T_z), \ \ T_z = \frac{   T-t_0}{\lambda^3(t_0)}.
\ee
The function ${z}$ admits a decomposition 
\bea
\label{neoneononeono}
z(t',x') & = & \frac{1}{\l^{\frac 12}_z(t')}(Q_{b_z}+\e_z)\left(t',\frac{x-x_z(t')}{\l_z(t')}\right)\\
\nonumber & = & \frac{1}{\l^{\frac 12}_z(t')}Q_{b_z(t')}\left(\frac{x-x_z(t')}{\l_z(t')}\right)+\tilde{z}(t',x'),
\eea
with explicitely:
\begin{align*}
& {\varepsilon_z}(t')=\varepsilon( {\lambda}^3({t_0}) t'+{t_0}), \quad {\lambda}_z(t')  =  {\lambda}(  {\lambda}^3({t_0}) t'+{t_0}) / {\lambda}({t_0}), \\
& {x_z}(t') = ( x({\lambda}^3({t_0})  t' +{t_0}) - x({t_0})) /{\lambda}({t_0}), \quad {b_z}(t')  = b({\lambda}^3({t_0}) t'+{t_0}). \end{align*}
In particular:
\be
\label{intientif}
\lambda_z(0)=1, \ \ x_z(0)=0, \ \ b_z(0)=b(t_0).
\ee
The monotonicity bound \fref{almostmofmbis} and \fref{boundhoneglobal} ensure: 
\be
\label{assumptions}
\forall t'\in [0,T_z), \ \
\|(\e_z)_x(t')\|_{L^2}^2\lesssim  \l^2(t') (|E_0| + \delta(\alpha)) ,
\ \ 
\|\e_z(t')\|_{L^2}^2 \lesssim \delta(\alpha),
\ee
  \be
\label{cnecnonoeoe} 
\ \ \lambda_z(t')\leq \frac 32, \quad 
 \|\tilde z(t')\|_{H^1}\lesssim  \lambda^2(t_0)|E_0| + \delta(\alpha)  \leq \delta_0
\ee provided $t_0$ is close enough to $T$ and $\alpha$ is small enough.

We denote by $\mathcal{N}_2(t')$ the quantity defined in \eqref{eq:no} for $z(t')$.
From (H2), and then \fref{conditionnubis}, we have
\bea
\nonumber \theta_z& = & \sup_{t'\in [0,T_z]} \ \ \left|\frac{b_z(t')+\mathcal N_{2,z}(t')}{\lambda^2_z(t')}\right|=\sup_{t\in[t_0,T)}\lambda^2(t_0) \left|\frac{b(t)+\mathcal N_2(t)}{\lambda^2(t)}\right|\\
& \lesssim &\lambda^2(t_0)\delta(\alpha) \lesssim \delta_0.\label{aswell}
\eea

Lemma \ref{lemmaimprovedaway} follows directly from the following   monotonicity result on $\widetilde z$ and
\eqref{aswell}.

\begin{lemma}[Monotonicity in renormalized variables]
\label{cl:6}
Assume \fref{intientif}, \fref{assumptions}, \fref{cnecnonoeoe}, \fref{aswell},  then $\forall y_0>a_0$, $\forall t'\in [0,T_z)$, there holds:\\
{\em (i) $L^2$ monotonicity}:  
\bea
\label{tout}
&& \int \widetilde z^2(t') \phi(x'-\nu t' + y_0) dx' +2(P,Q)(b_z(t')-b_z(0))\\
\nonumber &  & +
\frac 14\int_{0}^{{t'}} \int ({z}_x^2 + \nu {z}^2) (t'') \phi'(x'-\nu t' + {y_0}) dx'dt'' \nonumber\\
\nonumber && \lesssim (\theta_z)^{\frac 98} +  \int \widetilde z^2(0) \phi(x'+y_0) dx' + \frac 1{\sqrt{\nu}} e^{- \frac {\sqrt{\nu}} {10} {y_0}}.
\eea
{\em (ii) $H^1$ monotonicity}: 
\bea
\label{monoh1bis}
\nonumber & &  \int\left[ \tz _x^2-\frac13\tz^6\right](t')\phi\left( \frac 54 (x' - \nu t'  + {y_0})\right) dx'-2(P,Q)\left[\frac{b_z(t')}{\lambda_z^2(t')}-\frac{b_z(0)}{\lambda_z^2(0)}\right]\\
\nonumber   &  & +\frac 14  \int_0^{t'}\int ( {z}_{xx}^2   +\nu   {z}_x^2 ) (t'')\phi\left(  \frac 54 ({x' - \nu t''  + {y_0}}   )\right) dx'dt''\\ 
&&  \lesssim  (\theta_z)^{\frac{1} 8}+\int [ \tz _x^2(t_0)+\tz^2(t_0)] \phi(x'+y_0) dx'+\frac{1}{\sqrt{\nu}}e^{-\frac{\sqrt{\nu}}{10}y_0}.
\eea
\end{lemma}

 Undoing the transformation \fref{cnekocneoneo} and applying Lemma \ref{cl:6} yields Lemma \ref{lemmaimprovedaway}.
%%%%%%%%%%%%%%%%%%%%%%%%%%%%%%%%%%%%
%%%%%%%%%%%%%%%%%%%%%%%%%%%%%%%%%%%%

\begin{proof}[Proof of Lemma \ref{cl:6}]

The proof is closely related to the argument in \cite{Mjams}, \cite{MMannals}.

\medskip

We  define for ${y_0}>1$ and $0<\nu <\frac 1{10}$ the following localized mass and energy quantities:
\begin{align*}
& M_0 (t' ) = \int {z}^2(t',x') \phi(x' - \nu t'  + {y_0}) dx',\\
& E_0 (t' ) = \frac 12 \int \left(   {z}_x^2  - \frac 13 {z}^6 \right)(t',x') \phi\left(  \frac 54 ({x' - \nu t'  + {y_0}}   )\right) dx'.
\end{align*}

{\bf step 1} Monotonicity in $L^2$ for $z$.\\

We claim :
\begin{equation}\label{monol2}
M_0 (t' )-M_0(0)+ \frac 14 \int_0^{t'}\int ( {z}_x^2   +  \nu {z}^2 ) (t'')\phi'(x' - \nu t''  + {y_0}) dx'dt''\lesssim \frac 1{\sqrt{\nu}} e^{- \frac {\sqrt{\nu}} {10} {y_0}}.
\end{equation}
Indeed, we use   formula \eqref{L2kato} and \eqref{defphi} to estimate:
\begin{equation*}
\frac d{dt'}   M_0 (t' )  \leq    \int \left(- 3 {z}_x^2 - \frac {24}{25} \nu    {z}^2    + \frac 53   {z}^6
\right) \phi'(x' - \nu t'  + {y_0}).
\end{equation*}
We claim that the nonlinear term\footnote{which has the wrong sign}  is controllable up to an exponentially small term after integration in time. Indeed, first recall from Lemma 6 in \cite{Mjams} and \eqref{defphi} that for all $v\in H^1$, $a>0$, $b\in {\mathbb R}$,
\begin{align}
\|v^2 (\phi')^{\frac 12}\|_{L^\infty(|x-b|>a)}^2 & \lesssim \|v\|_{L^2(|x-b|>a)}^2 \left(\int v_x^2 \phi' +  \int v^2  \frac {(\phi'')^2}{\phi'}\right)
\nonumber\\
& \lesssim \|v\|_{L^2(|x-b|>a)}^2 \left(\int v_x^2 \phi' + \nu \int v^2  \phi'\right).\label{gnlocal}
\end{align}
Fix $a_0\gg 1$ such that 
$$
  \left(\int_{2 |y| >a_0} Q^2 \right)^2\leq  \delta_0.
$$
On the one hand, by \eqref{gnlocal},
\begin{align*}
& \int_{|x'-{x_z}(t')|>a_0} {z}^6 \phi'(x' - \nu t'  + {y_0})   \\
& \leq  \|{z}\|_{L^2(|x'-{x_z}(t')|>a_0)}^2 \|{z}^2 (\phi'(x' - \nu t'  + {y_0}))^{\frac 12}\|_{L^\infty(|x'-{x_z}(t')|>a_0)}^2   \\
& \lesssim \|{z}\|_{L^2(|x'-{x_z}(t')|>a_0)}^4  \int ({z}_x^2 +  \nu {z}^2)  \phi'(x' - \nu t'  + {y_0}).
\end{align*}
Since
\begin{align*}
\|{z}\|_{L^2(|x'-{x_z}(t')|>a_0)}^2 & \lesssim 
 \int_{{\lambda}_z (t') |y|> a_0} Q_b^2(y) dy + 
 \int  \varepsilon_z^2  \lesssim \delta_0+\delta(\alpha),
\end{align*}
we obtain, for $\delta_0 $ small enough and $\alpha$ small enough,
\bee
\int_{|x'-{x_z}(t')|>a_0} {z}^6 \phi'(x' - \nu t'  + {y_0})  &\lesssim &  (\delta_0+\delta(\alpha)) \int ( {z}_x^2   +  \nu {z}^2 ) \phi'(x' - \nu t'  + {y_0}) \\
 & 
\leq & \frac 1{4}  \int  ({z}_x^2   +  \nu {z}^2 ) \phi'(x' - \nu t'  + {y_0}).
\eee

On the other hand, the modulation equation \fref{eq:2002} and the upper bound on scaling \fref{assumptions} ensure: 
\be
\label{lowerboundseppd}
(x_z)_t=\frac{1}{\lambda_z^2}\frac{(x_z)_s}{\lambda_z}\geq \frac{1 +\delta(\alpha_0)}{\lambda^2_z}\geq 0.2
\ee
and thus in particular:
\be
\label{lowevound}
{x_z}(t') \geq {x_z}(0) + 0.2 t' \geq 0.1 t' + \nu t',
\ee 
We then estimate from Sobolev: $$ \|{z}\|_{L^6}^6 \lesssim  \|{z}\|_{H^1}^2 \|{z}\|_{L^2}^4\lesssim \frac 1{{\lambda}_z^2}\lesssim  ({x_z})_t(t'),$$ and obtain: $\forall {y_0}>a_0$,
 \begin{align*}
 & \int_{|x'-{x_z}(t')|<a_0} {z}^6 \phi'(x' - \nu t'  + {y_0})   \lesssim  ({x_z})_t(t')  \|\phi'(x' - \nu t'  + {y_0})\|_{L^\infty(|x'-{x_z}(t')|< a_0)}\\
 & \lesssim ({x_z})_t  (t') e^{-\frac {\sqrt{\nu}} {10} ({x_z}(t') -a_0 - \nu t' + {y_0}) } 
 \lesssim ({x_z})_t (t')  e^{-\frac {\sqrt{\nu}} {100} {x_z}(t')  - \frac {\sqrt{\nu}} {10} {y_0}}.
 \end{align*}
 
In conclusion, we have the $L^2$ motonicity formula: for all $t'\in [0,{t_0})$,
$$\frac d{dt'}   M_0 (t' )+ \frac 1{4} \int ( {z}_x^2   + \nu  {z}^2 )(t') \phi'(x' - \nu t'  + {y_0}) dx'  \lesssim  ({x_z})_t(t')   e^{-\frac {\sqrt{\nu}} {100} {x_z}(t') } e^{- \frac {\sqrt{\nu}} {10} {y_0}},$$ and by integration between $0$ and $t'$ using ${x_z}(0)=0$: $\forall t'\in [0,{T_z})$,
$$M_0 (t' )+ \frac 1{4} \int_0^{t'}\int ( {z}_x^2   + \nu  {z}^2 ) (t'')\phi'(x' - \nu t'  + {y_0}) dx'dt''\leq M_0(0) +   \frac C{\sqrt{\nu}} e^{- \frac {\sqrt{\nu}} {10} {y_0}}.
$$

{\bf step 2} Monotonicity in $L^2$ for $\tilde z$. Proof of \eqref{tout}. 

We now rewrite the monotonicity \fref{monol2} using the decomposition \fref{neoneononeono}.
We compute:
\bee
M_0(t')& = & \int {z}^2(t',x') \phi(x' - \nu t'  + {y_0}) dx'\\
& = & \int (Q_{b_z(t')}(y)+\e_z(y,t'))^2\phi(\lambda_z(t')y+x_z(t')- \nu t'  + {y_0}) dydt'\\
& = & \int Q_{b_z(t')}^2\phi(\lambda_z(t')y+x_z(t')- \nu t'  + {y_0})\\
& + & \int {\tz}^2(t',x') \phi(x' - \nu t'  + {y_0}) dx'\\&+&2\int Q_{b_z(t')}\e_z(t')\phi(\lambda_z(t')y+x_z(t')- \nu t'  + {y_0}) dydt
\eee
We estimate using the lower bound \fref{lowevound}:
\bee
&&\int Q_{b_z(t')}^2\phi(\lambda_z(t')y+x_z(t')- \nu t'  + {y_0})\\
&  =&   \int Q^2+ 2 b_z({t'}) (P,Q) + 2 b_z({t'}) \int \chi_{b_x(t')} P \phi(\lambda_z(t')y+x_z(t')- \nu t'  + {y_0}) \\
& = & \int Q^2+2b_z(t')(P,Q)+O(e^{- \frac {\sqrt{\nu}} {10} {y_0}})+ O(b_z^{2-\gamma}({t'}))
\eee
where we used 
$b^2 \int P^2\chi_b^2= O(b^{2-\gamma})$. Now by H\"older:
\bea
\label{cneonoeoeo}
\nonumber &&2b_z({t'})\left| \int \e_z(t') \chi_{b_z} P\phi(\lambda_z(t')y+x_z(t')- \nu t'  + {y_0})\right|\\
\nonumber & \lesssim &  (b_z(t'))^{\frac {1-\gamma}2} \int \e_z^2(t') \phi(\lambda_z(t')y+x_z(t')- \nu t'  + {y_0})+   (b_z(t'))^{\frac {3+\gamma}2}  \int P^2 \chi_{b_z}^2\\ 
&\leq&    (b_z(t'))^{\frac {1-\gamma}2}\int \tz^2(t',x')\phi(x'- \nu t'  + {y_0}) dx'+   (b_z(t'))^{\frac {3-\gamma}2}.
\eea
We now inject these estimates into \fref{monol2} and use from \fref{assumptions} and the definition of $\theta_z$: 
\be
\label{roughboudn}
|b_z(t')|\lesssim \theta_z,
\ee 
and thus derive from the initialization \fref{intientif} the bound (note $\gamma = \frac 34$): $\forall t'\in [0,T_z)$,
\bea
&&\int {\tz}^2(t',x') \phi(x' - \nu t'  + {y_0}) dx +\int_0^{t'}\int ( {z}_x^2   +  \nu {z}^2 ) (t'')\phi'(x' - \nu t''  + {y_0}) dx'dt''\nonumber\\
&& \lesssim\theta_z^{\frac 98}+\int {\tz}^2(0,x') \phi(x' + {y_0}) dx'+\frac 1{\sqrt{\nu}}e^{- \frac {\sqrt{\nu}} {10} {y_0}}.
\label{toutbis}
\eea
Reinjecting this bound into \fref{cneonoeoeo} and \fref{monol2}, keeping track of the $b_z$ powers now yields \eqref{tout}. 

\medskip

{\bf step 3} Energy monotonicity for $z$.

We claim the energy monotonicity:
\begin{align}
& E_0(t')-E_0(0)+ \frac 1{4} \int_0^{t'}\int ( {z}_{xx}^2   +\nu   {z}_x^2 ) (t'')\phi\left(  \frac 54 ({x' - \nu t''  + {y_0}}   )\right) dx'dt''  \nonumber  \\ &  \lesssim  \left( \theta_z^{\frac 98}+  \int \widetilde z^2(t_0) \phi(x'+y_0) dx'+   \frac 1{\sqrt{\nu}} e^{- \frac {\sqrt{\nu}} {10} {y_0}}\right)^{\frac 54}.\label{monoh1}
\end{align}
Indeed, we estimate from  formula \eqref{enerkato} and \eqref{defphi}:
\begin{align}
&\frac d{dt'}   E_0 (t' ) \\ & =  -  \frac 54 \int \left( ( z_{xx} + z^5)^2   + 2 z_{xx}^2
- 10 z^4 z_x^2 +  \nu (z_x^2 - \frac 13 z^6)\right)   \phi'\left(  \frac 54 ({x' - \nu t'  + {y_0}}   )\right)\nonumber \\
& +  \left(\frac 54\right)^3 \int z_x^2 \phi'''\left(  \frac 54 ({x' - \nu t'  + {y_0}}   )\right)\nonumber \\
& \leq -  \frac 54   \int \left( 2 z_{xx}^2 +  \frac {{\nu}}2 z_x^2 - \frac \nu{3} z^6 -  10 z^4 z_x^2  \right)    \phi'\left(  \frac 54 ({x' - \nu t'  + {y_0}}   )\right).\label{finire}
\end{align}
We need to treat the non linear terms. We  claim:
\begin{align}
& \int_{0}^{T_z} \int  z^4 z_x^2      \phi'\left(  \frac 54 ({x' - \nu t'  + {y_0}}   )\right) dx'dt'  \nonumber\\
 &\lesssim  \delta_0 \int_0^{T_z} \int \left(z_{xx}^2+\nu z_x^2\right)  \phi'\left(  \frac 54 ({x' - \nu t'  + {y_0}}   )\right) dx'dt'  + \frac 1{\sqrt{\nu}} e^{-\frac 54 \frac {\sqrt{\nu}}{10} y_0}\nonumber \\
& +   \int_{0}^{T_z} \int  z^6      \phi'\left(  \frac 54 ({x' - \nu t'  + {y_0}}   )\right) dx'dt',
 \label{pourz4zx2}
\end{align}
for some small enough $\delta_0>0$, and
\begin{align}& \nonumber 
\int_0^{{T_z}} \int      z^6(t')    \phi'\left(\frac 54 ( {x' - \nu t'  + {y_0}})  \right) dt'
  \\ &\lesssim \left( \theta_z^{\frac  98}+ \int \tz^2(t_0) \phi(x'+y_0)dx'  +   \frac 1{\sqrt{\nu}} e^{- \frac {\sqrt{\nu}} {10} {y_0}}\right)^{\frac 54}.
\label{pourz6}\end{align}
Integrating  \fref{finire} in time and injecting \fref{pourz4zx2}, \fref{pourz6} yields \eqref{monoh1}.
\medskip

{\it Proof of \fref{pourz4zx2}}: For $a_1>0$ large enough, we have\footnote{using $x^2e^{-x}\lesssim 1$ for $x\geq 0$}
$$\frac 1{\lambda_z^2(t')} \int_{|x|>a_1} (Q')^2 \left( \frac x{\lambda_z(t')}\right) dx \lesssim  \frac 1{\lambda_z^2(t')} e^{- \frac {2 a_1}{\lambda_z(t')}}\lesssim \frac{1}{a_1^2}\leq \delta_0,$$ then:
\bea
\label{cnkoneokneo}
\nonumber \int_{|x-x_z(t')|>a_1}  z_x^2 & \lesssim & \int_{|x-x_z(t')|>a_1}   \widetilde z_x^2(x) + \frac 1{\lambda_z^2(t')} \int_{|x|>a_1} (Q')^2 \left( \frac {x}{\lambda_z(t')}\right) \\
& \lesssim &  \delta_0
\eea
 where we used the smallness in the $H^1$ bound \fref{cnecnonoeoe}.\\
We now write
\begin{align*}
& \int_{|x-x_z(t)|>a_1} z^4 z_x^2  \phi'\left(  \tfrac 54 ({x' - \nu t'  + y_0}   ) \right)
\\ &\lesssim \int_{|x-x_z(t)|>a_1} \left(z^2 z_x^4 +  z^6\right)\phi'\left(  \tfrac 54 ({x' - \nu t'  + y_0}   )\right),
\end{align*}
and need only treat the first term according to the expected bound \fref{pourz4zx2}. We estimate the outer integral by using the localized Gagliardo-Nirenberg inequality \eqref{gnlocal} and the outer smallness by \fref{cnkoneokneo}:
\bee
& & \int_{|x-x_z(t)|>a_1} z^2 z_x^4 \phi'\left(  \tfrac 54 ({x' - \nu t'  + y_0}   )\right)\\
& \lesssim & \|z_x^2  \left(\phi' (  \tfrac 54 ({x' - \nu t'  + y_0}   )\right)^{\frac 12}\|_{L^\infty(|x-x_z(t')|>a_1)}^2 \|z\|_{L^2}^2   \\
& \lesssim &\|z_x\|_{L^2(|x-x_z(t')|>a_1)}^2   \int  \left(z_{xx}^2  + \nu z_x^2\right)   \phi' \left(  \tfrac 54 ({x' - \nu t'  + y_0} )\right)\\
& \lesssim  & \delta_0  \int  \left(z_{xx}^2  + \nu z_x^2\right)   \phi' \left(  \tfrac 54 ({x' - \nu t'  + y_0} )\right).
 \eee
The inner integral is estimated from Sobolev $$\int z^4 z_x^2 \lesssim \|z\|_{L^\infty}^4 \|z_x\|_{L^2}^2 \lesssim \|z\|_{L^2}^2 \|z_x\|_{L^2}^4\lesssim \frac{1}{\l_z^4},$$ and hence using the structure of $\phi$ and \fref{lowevound}:
\bee
&&\int_{|x-x_z(t)|<a_1} z^4 z_x^2  \phi' \left(  \tfrac 54 ({x' - \nu t'  + y_0} )\right)\\ &\lesssim &\frac{1}{\l_z^4(t')}   \|\phi' \left(  \tfrac 54 ({x' - \nu t'  + y_0} )\right)\|_{L^\infty(|x-x_z(t')|<a_1)}\\
&  \lesssim & \frac 1{\lambda_z^4(t')} e^{-\frac {\sqrt{\nu}}{100} x_z(t') - \frac 54 \frac {\sqrt{\nu}}{10} y_0}.
\eee
We now claim 
\begin{equation}\label{beurk}
\frac{1}{c_0\l^2_z(t')}e^{-c_0x_z(t')}+\int_0^{T_z} \frac 1 {\lambda_z^4(t')} e^{- c_0x_z(t')} dt' \lesssim \frac{1}{c_0}
\end{equation}
with $c_0=C\sqrt{\nu}$, which completes the proof of \eqref{pourz4zx2}.

 Indeed, first observe from the definition of $\theta_z$ and the rough modulation equation \fref{eq:2002}:
$$|(\lambda_z)_t|=\left|\frac{1}{\lambda_z^2}\frac{-(\lambda_z)_s}{\lambda_z}\right|\lesssim \frac{1}{\lambda_z^2}(|b_z|+\sqrt{\theta_z}\lambda_z)\lesssim \frac{\sqrt{\theta_z}}{\lambda_z},$$ and thus  from \fref{lowerboundseppd} and an integration by parts in time:
\bee  
&&\int_0^{t'} \frac 1 {\lambda_z^4} e^{- c_0x_z} d\tau\lesssim \int_0^{t'} \frac{(x_z)_t}{\lambda_z^2} e^{- c_0x_z} d\tau\\
& = & \left[\frac{-1}{c_0\lambda_z^2}e^{- c_0x_z}\right]_0^{t'}-\frac{1}{c_0}\int_0^{t'}\frac{2(\lambda_z)_t}{\lambda_z^3}e^{- c_0x_z} d\tau\\
& \leq & \frac{1}{c_0}\left[1-\frac{1}{\l^2_z(t')}e^{-c_0x_z(t')}\right]+\frac{2\sqrt{\mathcal \theta_z}}{c_0}\int_0^{t'} \frac 1 {\lambda_z^4} e^{- c_0x_z} d\tau,
\eee
and \fref{beurk} now follows from the a priori smallness \fref{aswell}, \eqref{conditionnubis}.
\medskip

{\it Proof of \fref{pourz6}}: Since $\phi'(\frac 54 x)\lesssim (\phi')^{\frac 54}(x)$, \eqref{gnlocal} yields:
\begin{align*}
& \int      z^6    \phi'\left(\frac 54 ( {x' - \nu t'  + {y_0}})  \right)
 \leq \| z^2 (\phi')^{\frac 12}( {x' - \nu t'  + {y_0}})\|_{L^\infty}^2 \int z^2 (\phi')^{\frac 14}( {x' - \nu t'  + {y_0}})\\
& \lesssim \left(\int z^2\right)^{\frac 74} \left( \int z^2 \phi' ( {x' - \nu t'  + {y_0}})\right)^{\frac 14} \int (z_x^2+ \nu z^2) \phi'({x' - \nu t'  + {y_0}})  \\
& \lesssim \left( \int z^2 \phi' ( {x' - \nu t'  + {y_0}})\right)^{\frac 14} \int (z_x^2+ \nu z^2) \phi'({x' - \nu t'  + {y_0}}) .
\end{align*}
We now estimate:
\begin{align*}
& \int z^2 \phi' ( {x' - \nu t'  + {y_0}})
\\&  \lesssim \int  \widetilde z^2 \phi' ( {x' - \nu t'  + {y_0}})  +  \int Q_{{b_z}}^2(y) \phi' ( {{\lambda}_z(t')y + {x_z}(t) - \nu t'  + {y_0}}).
\end{align*}
On the one hand by \eqref{toutbis} and $\phi'\lesssim \phi$:
$$
\int  \widetilde z^2 \phi' ( {x' - \nu t'  + {y_0}}) \lesssim  \theta_z^{\frac 98} +  \int \tz^2(0) \phi(x'+y_0)dx' +  \frac 1{\sqrt{\nu}} e^{- \frac {\sqrt{\nu}} {10} {y_0}}.
$$
On the other hand, from the space decoupling \fref{lowevound}:
\begin{align*}
& \int Q_{b}^2(y) \phi' ( {{\lambda}_z(t')y + {x_z}(t) - \nu t'  + {y_0}})\\ &\lesssim
 |b|^{2-\gamma}(t') + \int Q^2(y) \phi' ( {{\lambda}_z(t')y + {x_z}(t) - \nu t'  + {y_0}})\\
&  \lesssim  \theta_z^{\frac 54} +   \frac 1{\sqrt{\nu}} e^{- \frac {\sqrt{\nu}} {10} {y_0}}.
\end{align*}
The space-time estimate \eqref{pourz6} now follows from \eqref{toutbis}.

\medbreak

{\bf step 4} Energy monotonicity for $\tilde z$. Proof of \fref{monoh1bis}.

We now rewrite the monotonicity \fref{monoh1} using the decomposition \fref{neoneononeono}. We compute:
\bee
&& 2\lambda^2_z(t')E_0(t')\\
&& =\int \left[ (Q_{b_z}+\e_z)_y^2  - \frac 13(Q_{b_z}+\e_z)^6 \right](t',y) \phi\left(  \frac 54 (\lambda_z(t')y+x_z(t') - \nu t'  + {y_0})\right) dy
\eee
and develop this expression.  The contribution of the $Q_b$ term is estimated using $E(Q)=0$ and the separation in space \fref{lowevound} which implies:
\begin{align*}
&\int [(Q_b)_y^2+Q_b^6]\left[1-\phi\left(  \frac 54 (\lambda_z(t')y+x_z(t') - \nu t'  + {y_0})\right)\right] dy\\ &\lesssim |b_z|^{1+\gamma}+\frac{1}{\sqrt{\nu}}e^{-\frac{\sqrt{\nu}}{10}(x_z(t')+y_0)}.
\end{align*} The cross terms are treated using the orthogonality condition \fref{ortho1} and we obtain similarily like for the proof of \fref{energbound}:
\begin{align}
\label{pofouopef}
 & 2\lambda^2_z(t')E_0(t') \\
\nonumber & =  -2b_z(t')(P,Q)+\int\left[(\e_z)_y^2-\frac13\e_z^6\right](t',y)\phi\left( \frac 54 (\lambda_z(t')y+x_z(t') - \nu t'  + {y_0})\right) dy\\
 &   + O\left[\frac{1}{\sqrt{\nu}}e^{-\frac{\sqrt{\nu}}{10}(x_z(t')+y_0)}+{|b_z(t')|^{1+\gamma}} +   |b_z(t')|^{1-\gamma}  \left(\int (\e_z)_y^2 + \int \e_z^2 {e^{-|y|}} \right)\right].\nonumber
\end{align}
We now divide by $\lambda_z(t')$. We estimate from \fref{assumptions}:
$$ \frac{1}{\lambda^2_z(t')}\left[{|b_z(t')|^{1+\gamma}} +    |b_z(t')|^{1-\gamma}  \left(\int (\e_z)_y^2 + \int \e_z^2 {e^{-|y|}} \right)\right]\lesssim (\theta_z)^{\frac 18}$$ and conclude using \fref{beurk}, \fref{pofouopef}:
\bee
2E_0(t') & = & -\frac{2b_z(t')}{\lambda_z^2(t')}(P,Q)+\int\left[(\tz)_x^2-\frac13\tz^6\right]\phi\left( \frac 54 (x' - \nu t'  + {y_0})\right) dx'\\
& + & O\left((\theta_z)^{\frac 18}+\frac{1}{\sqrt{\nu}}e^{-\frac{\sqrt{\nu}}{10}y_0}\right).
\eee
which together with the monotonicity \fref{monoh1} and   $L^2$ smallness of $\tz$ yields \fref{monoh1bis}.
\end{proof}

%%%%%%%%%%%%%%%%%%%%%%%%%%%%%%%%%%%%
%%%%%%%%%%%%%%%%%%%%%%%%%%%%%%%%%%%%

%%%%%%%%%%%%%%%%%%%%%%%%%%%%%%%%%%%%
%%%%%%%%%%%%%%%%%%%%%%%%%%%%%%%%%%%%

%%% A COPIER-COLLER ENTRE ICI
%%%%%%%%%%%%%%%%%%%%%%%%%%%%%%%%%%

\subsection{Proof of Lemma \ref{cl:9}}\label{s:A:2} 
The  proof of Lemma \ref{cl:9} is based on coercivity properties of the viriel quadratic form under suitable repulsivity properties. We recall this property in the following lemma.

\begin{lemma}[\cite{MMjmpa}, Proposition 4]\label{le:6}
There exists $\mu>0$ such that, for all $v\in H^1(\mathbb{R})$,
\begin{align*}
&3 \int v^2_y +   \int v^2 - 5 \int Q^4 v^2 
+ 20 \int y Q' Q^3 v^2
\\ & \geq \mu \int v_y^2 + v^2 - \frac 1{\mu} \left(\int v y \Lambda Q \right)^2
- \frac 1{\mu} \left( \int vQ \right)^2  
\end{align*}
\end{lemma}

We now turn to the proof of Lemma \ref{cl:9}  which is a simple consequence of Lemma \ref{le:6} using a standard localization argument (see for example the proof of  Proposition 9 in \cite{MMannals}).
Indeed, let $\zeta$ be a smooth function such that
$$
\zeta(y) = 0 \hbox{ for $|y|>\frac 14$}; \ \zeta(y)=1 \hbox{ for $|y|<\frac 18$};\ 0\leq  \zeta\leq 1 \hbox{ on $\RR$}.
$$
Set 
$$
\et(y) = \e(y) \zeta_B(y) \quad \hbox{where} \quad \zeta_B(y) = \zeta\left( \frac {y}{B}\right).
$$
Lemma \ref{le:6} applied to $\et$ gives
 \begin{align}
& \left( 3-  \mu \right) \int \et^2_y +  \left( 1-  \mu\right) \int \et^2 - 5 \int Q^4 \et^2 
+ 20 \int y Q' Q^3 \et^2\nonumber
\\ & \geq   - \frac 1{\mu} \left(\int \et y \Lambda Q \right)^2
- \frac 1{\mu} \left( \int \et Q \right)^2  \label{appendA2}
\end{align}
On the one hand,  
$$
\int \et_y^2   = \int \e_y^2 \zeta_B^2+ \int \e^2 (\zeta_B')^2 - \frac 12 \int \e^2 (\zeta_B^2)''
  \leq \int_{|y|< \frac B4} \e_y^2  + \frac {C}{B^2} \int_{|y|< \frac B4} \e^2,
$$
$$
\int \et^2 = \int \e^2 \zeta_B^2 \leq \int_{|y|< \frac B4} \e^2 ,
$$
and by $yQ'<0$ and then by the exponential decay of $Q$ and $Q'$
\begin{align*}
& - 5 \int Q^4 \et^2 
+ 20  \int  y Q' Q^3 \et^2  \leq - 5 \int_{|y|< \frac B4} Q^4 \et^2 
+ 20 \int_{|y|< \frac B4} y Q' Q^3 \et^2\\
& \leq - 5 \int_{|y|< \frac B2} Q^4 \et^2  + 20 \int_{|y|< \frac B2} y Q' Q^3 \et^2
+ C e^{-\frac B{16}}   \int_{\frac B4 <|y| <\frac B2} \e^2  .
\end{align*}
Thus, for $B$ large,
\begin{align*}
&(3 -\mu)\int \et^2_y + (1-\mu)  \int \et^2 - 5 \int Q^4 \et^2 
+ 20 \int y Q' Q^3 \et^2 \leq (3 -\mu)\int_{|y|< \frac B4} \e_y^2 \\ &+(1-\mu)\int_{|y|< \frac B4} \e^2
- 5 \int_{|y|< \frac B2} Q^4 \et^2  + 20 \int_{|y|< \frac B2} y Q' Q^3 \et^2
+ \frac {C}{B^2} \int_{|y|< \frac B4} \e^2\\ 
&\leq (3 -\mu)\int_{|y|< \frac B2} \e_y^2  +\left(1-\frac \mu 2\right)\int_{|y|< \frac B2} \e^2
- 5 \int_{|y|< \frac B2} Q^4 \et^2  + 20 \int_{|y|< \frac B2} y Q' Q^3 \et^2
\end{align*}
On the other hand, by \eqref{ortho1},
\begin{align*}
\left|\int \et y \Lambda Q \right|= \left|\int \e \zeta_B y \Lambda Q \right|
= \left| \int \e  (1-\zeta_B) y \Lambda Q\right|
\lesssim e^{-\frac B{16}} \left( \int \e^2 e^{-\frac {|y|}{2}}\right)^{\frac 12}
\end{align*}
and similarly for $\int \et Q$. Inserted in \eqref{appendA2}, these estimates finish the proof of Lemma~\ref{le:6}.

\end{document}